\documentclass{article}
\usepackage[utf8]{inputenc}
\usepackage[T1]{fontenc}
\usepackage{amssymb,amsmath,amsthm,mathrsfs}
\usepackage[english]{babel}
\usepackage{mathtools}
\usepackage{enumerate}
\usepackage{authblk}
\usepackage{graphicx}
\usepackage{subcaption}
\usepackage[margin=0.3in]{geometry}
\newgeometry{includefoot,left=3cm,right=3cm, bottom=3cm,top=3cm}

\providecommand{\E}{\mathbb{E}}
\providecommand{\prob}{\mathbb{P}}
\providecommand{\UH}{\mathbb{H}}
\providecommand{\R}{\mathbb{R}}
\providecommand{\N}{\mathbb{N}}
\DeclareMathOperator{\dist}{dist}
\DeclareMathOperator{\SLE}{SLE}
\newcommand{\squig}{{\scriptstyle\sim\mkern-4.3mu}}
\renewcommand{\epsilon}{\varepsilon}

\newtheorem{thm}{Theorem}[section]
\newtheorem{lem}[thm]{Lemma}
\newtheorem{prop}[thm]{Proposition}
\newtheorem{cor}[thm]{Corollary}
\theoremstyle{definition}
\newtheorem{rmk}[thm]{Remark}

\title{A multifractal boundary spectrum for $\SLE_\kappa(\rho)$}
\author{Lukas Schoug\thanks{lschoug@maths.cam.ac.uk}}
\affil{University of Cambridge}
\date{}

\begin{document}

\maketitle
\begin{abstract}
    We study $\SLE_\kappa(\rho)$ curves, with $\kappa$ and $\rho$ chosen so that the curves hit the boundary. More precisely, we study the sets on which the curves collide with the boundary at a prescribed ``angle'' and determine the almost sure Hausdorff dimensions of these sets. This is done by studying the moments of the spatial derivatives of the conformal maps $g_t$, by employing the Girsanov theorem and using imaginary geometry techniques to derive a correlation estimate.
\end{abstract}

\section{Introduction}
The Schramm-Loewner evolution ($\SLE$) is a one parameter family of random fractal curves, that was introduced by Oded Schramm as a conformally invariant candidate for the scaling limit of two-dimensional discrete models in statistical mechanics. Consider the half-plane Loewner differential equation
\begin{align} \label{eq:LDE}
    \partial_t g_t(z) = \frac{2}{g_t(z)-W_t}, \qquad g_0(z) = z,
\end{align}
where the driving function, $W_t$, is continuous and real-valued. The chordal Schramm-Loewner evolution with parameter $\kappa>0$ ($\SLE_\kappa$) is the curve with corresponding conformal maps given by the Loewner equation with $W_t = \sqrt{\kappa} B_t$. $\SLE_\kappa$ exhibits interesting geometric behaviour. If $0 < \kappa \leq 4$, the curves are simple and do not intersect the real line, if $\kappa > 4$ they have non-traversing self-intersections and collide with the real line and if $\kappa \geq 8$, the curves are space-filling. For $\kappa > 0$, the almost sure Hausdorff dimension is $d_\kappa = \min\{2,1+\frac{\kappa}{8}\}$, see e.g. \cite{Bef08}. For $\kappa > 4$, the intersection of an $\SLE_\kappa$ curve with the real line is a random fractal of almost sure Hausdorff dimension $\min\{2-8/\kappa,1\}$, see \cite{AS08} and \cite{SZ10}. In \cite{ABV16}, Alberts, Binder and Viklund studied and computed the almost sure Hausdorff dimension spectrum of random sets of points, where the $\SLE_\kappa$ curve, for $\kappa > 4$, hits the real line at a prescribed ``angle''.

If we again, consider the Loewner equation, but instead let $W_t$ be the solution to the following system of SDEs
\begin{align*}
    dW_t &= \sqrt{\kappa} dB_t + \frac{\rho dt}{W_t - V_t}, \quad W_0 = 0; \\
    dV_t &= \frac{2dt}{V_t - W_t}, \quad V_0 = x_R,
\end{align*}
where $x_R$ is a point on the real line, called the force point, and $\rho > -2$ is an associated weight, then the Loewner chain is generated by a random curve, called an $\SLE_\kappa(\rho)$ curve (for the definition when $\rho \leq -2$, see \cite{MS19}). This two parameter family of random fractal curves is a natural generalization of $\SLE_\kappa$, in which one keeps track of the force point as well as the curve. The weight $\rho$ determines a repulsion (if $\rho > 0$) or attraction (if $\rho < 0$) of the curve, from the boundary. For $\rho = 0$, it is the ordinary $\SLE_\kappa$. For $\kappa \in (0,8)$, $\rho \in ((-2)\vee(\frac{\kappa}{2}-4), \frac{\kappa}{2}-2)$ and $x_R=0^+$, the Hausdorff dimension of the intersection of $\SLE_\kappa(\rho)$ with the real line is almost surely $1 - (\rho+2)(\rho+4-\kappa/2)/\kappa$ (see \cite{MW17} and \cite{WW13}). What we are interested in studying in this article, is the dimension spectrum studied by Alberts, Binder and Viklund in \cite{ABV16}, but for $\SLE_\kappa(\rho)$. In \cite{AS18}, the authors apply the main result of this paper to describe the boundary hitting behaviour of the loops of the two-valued sets of the Gaussian free field.

Let $\eta$ be the $\SLE_\kappa(\rho)$ curve with $x_R = 0^+$ and let $\tau_s(x)$ be the first time $\eta$ is within distance $e^{-s}$ of $x$, i.e.,
\begin{align*}
    \tau_s(x) = \inf \{ t \geq 0: \dist(x,\eta([0,t])) \leq e^{-s}\}.
\end{align*}
Let $(g_t)$ be the $\SLE_\kappa(\rho)$ Loewner chain, let $g_t'$ denote the spatial derivative of $g_t$ and for $\beta \in \R$, define
\begin{align} \label{eq:Vbeta}
    V_\beta = \left\{ x > 0: \lim_{s \rightarrow \infty} \frac{1}{s} \log g_{\tau_s}'(x) = -\beta, \ \tau_s = \tau_s(x) < \infty \ \forall s>0 \right\},
\end{align}
that is, $V_\beta$ is the set of points $x$ in $\R_+$, at which $g_{\tau_s}'$ decays like $e^{-\beta s}$. This can be viewed as a generalized hitting ``angle'' of the real line (see the discussion in the introduction in \cite{ABV16}). We shall view the decay of $g_{\tau_s}'(x)$ as a decay of a certain harmonic measure. Indeed, let $r_t$ be the rightmost point of $\eta([0,t]) \cap \R$ and let $H_t$ be the unbounded connected component of $\UH \setminus \eta([0,t])$. Then the harmonic measure from infinity of $(r_t,x]$ in $H_t$ is defined as
\begin{align*}
    \omega_\infty((r_t,x],H_t) = \lim_{y \rightarrow \infty} y \omega(iy,(r_t,x],H_t).
\end{align*}
Assume that we have
\begin{align*}
    c^{-1} e^{-\beta s} \leq g_{\tau_s}'(x) \leq c e^{-\beta s}
\end{align*}
for some $\beta$. By using first Schwarz reflection and then the Koebe $1/4$ theorem, we have
\begin{align*}
    c^{-1} e^{-\alpha s} \leq \omega_\infty((r_{\tau_s},x],H_{\tau_s}) \leq c e^{-\alpha s}, \quad \alpha = \beta+1,
\end{align*}
for some constant $c$, possibly different from the previous, see Figure \ref{fig:harmview}. As we see in (\ref{eq:Vbeta}), however, we allow a subexponential error, as up to constant asymptotics are too restrictive to require.

\begin{figure}[ht!]
\centering
\includegraphics[width=120mm]{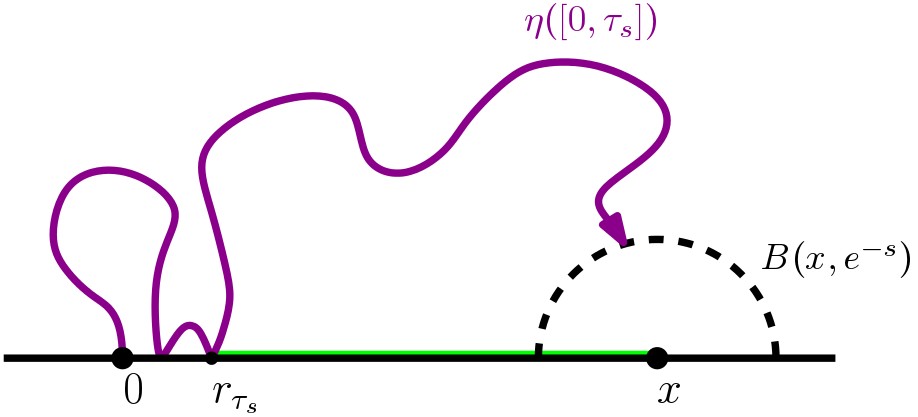}
\caption{Having a certain rate of decay of $g_{\tau_s}'(x)$ as $s\rightarrow \infty$ is equivalent to having a related decay rate of the harmonic measure from infinity of the set $(r_{\tau_s},x]$ (the green interval in the figure). \label{fig:harmview}}
\end{figure}

For fixed $\kappa>0$, $\rho \in \R$, let
\begin{align}\label{dim}
    d(\beta) = 1 - \frac{\beta}{\kappa} \left( \frac{\kappa-2\rho}{4} - \frac{1+\rho/2 + 2\beta}{\beta} \right)^2,
\end{align}
and write $\beta_- = \inf \{\beta:d(\beta)>0\}$ and $\beta_+ = \sup \{\beta:d(\beta)>0\}$. The main theorem of the paper is the following.
\begin{thm}
\label{mainresult}
Let $\kappa \in (0,4]$, $\rho \in (-2, \frac{\kappa}{2}-2)$, $x_R = 0^+$ and $\beta \in [\beta_-,\beta_+]$. Then, almost surely,
\begin{align*}
    \textup{dim}_H V_\beta = d(\beta).
\end{align*}
\end{thm}
The reason for the restricting $\rho$ to values smaller than $\kappa/2 - 2$ is that this is the critical value for the curve to hit the boundary, that is, the curve will not hit the boundary if $\rho \geq \kappa/2 -2$.

An interesting observation is that there is a typical behaviour of the curve hidden in this theorem; there is a $\beta$ such that $V_\beta$ has full Hausdorff dimension, that is, such that $\text{dim}_H V_\beta = \text{dim}_H \eta \cap \R_+$. Indeed, if $\beta_0 = (4+2\rho)/(8-\kappa+2\rho)$, then
\begin{align*}
    d(\beta_0) = 1 - \frac{(\rho+2)(\rho+4-\kappa/2)}{\kappa},
\end{align*}
which was proven in \cite{MW17} to be the Hausdorff dimension of the curve intersected with the real line.

\begin{rmk}
If we let
\begin{align*}
    \Omega_\alpha = \left\{ x>0: \lim_{s \rightarrow \infty} \frac{1}{s} \log \omega_\infty((r_{\tau_s},x],H_{\tau_s}) = -\alpha, \ \tau_s = \tau_s(x) < \infty \ \forall s>0 \right\},
\end{align*}
then we can rephrase Theorem \ref{mainresult} as
\begin{align*}
    \textup{dim}_H \Omega_\alpha = 1-\frac{\alpha-1}{\kappa} \left( \frac{\kappa-2\rho}{4} - \frac{2\alpha -1 +\rho/2}{\alpha-1}\right)^2.
\end{align*}
\end{rmk}

To prove Theorem \ref{mainresult}, it will be more convenient to consider the sets
\begin{align} \label{eq:Vbeta*}
    V_\beta^* = \left\{ x > 0: \lim_{s \rightarrow \infty} \frac{1}{s} \log g_{\tau_s}'(x) = -\beta(1+\rho/2), \ \tau_s = \tau_s(x) < \infty \ \forall s>0 \right\},
\end{align}
and then use that $V_\beta = V_{\beta/(1+\rho/2)}^*$. 

We now give an overview of the paper. In Section \ref{prel}, we introduce the preliminary material needed in the rest of the paper, such as $\SLE_\kappa(\underline{\rho})$ processes, the Gaussian free field and the imaginary geometry coupling. In order to make the paper more self-contained, the section on imaginary geometry is more extensive than necessary. Furthermore, we use the Girsanov theorem to weight the measure with the local martingale given by the product of a time change of $g_t'(x)^\zeta$ (where $\zeta$ is a parameter in one-to-one correspondence with $\beta$) and a compensator. With this new measure we can compute the asymptotics of $g_{\tau_s}'(x)^\zeta$, which we use in Section \ref{secOPE} to find a one-point estimate. It turns out to be strong enough to give the upper bound on the dimension of $V_\beta$, so this can actually be achieved immediately after Corollary 3.2. The rest of Section \ref{secOPE} is dedicated to studying the mass concentration of the weighted measures, which we need for the correlation estimate. Section \ref{sec2PEIG} is dedicated to the proof of the two-point estimate needed to prove Theorem \ref{mainresult}. This is done by employing the coupling between SLE and the Gaussian free field. We finish the paper in Section \ref{secdim} by first establishing the upper bound on the dimension of $V_\beta$ using the one-point estimate of Section \ref{secOPE} and then constructing Frostman measures and using the two-point estimate to show that the $s$-dimensional energy is finite for every $s < d(\beta)$, and hence that the Hausdorff dimension can not be smaller than $d(\beta)$.

We believe Theorem \ref{mainresult} to be true for all $\kappa$ and $\beta \in [\beta_-(\kappa,\rho),\beta_+(\kappa,\rho)]$, and our upper bound is actually valid for all parameters. Given the spectrum for $\kappa \in (0,4]$ it seems natural to try and prove the result for $\kappa > 4$ using SLE duality, i.e., that the outer boundary of the $\SLE_\kappa(\rho)$ curves are variants of $\SLE_{16/\kappa}$ curves, similar to what is done in \cite{MW17}. This is not as straightforward here, however, as what we are interested in is not the dimension of the intersection of the curve with the real line, but the set of points where the curve intersects the boundary with the prescribed behaviour of the derivatives of the conformal maps (or equivalently, the decay of $\omega_\infty$) as the curve approaches the boundary. How to do this is not clear at the moment.

However, using the method of \cite{ABV16} to get a two-point estimate, one can deduce that in the case $\kappa > 4$, the theorem holds for $\beta \in [\beta_-,\beta_0]$. This is done by considering three events which exhaust the possible geometries of the curve approaching the two points (this is possible as for boundary interactions, the geometries are rather simple) and then separately estimating each of them. The correlation estimate that we have is actually more closely related to the one in \cite{MW17}.

An almost sure multifractal spectrum of SLE curves was first derived in \cite{JVL12}, where the reverse flow of SLE was used to study the behaviour of the conformal maps close to the tip of the curve. Another result in this direction is \cite{GMS18}, where the imaginary geometry techniques, developed and demonstrated in the articles \cite{MS16a}, \cite{MS16b}, \cite{MS16c} and \cite{MS17}, were used to find an almost sure bulk multifractal spectrum. In \cite{Mil18}, Miller used the imaginary geometry techniques to compute Hausdorff dimensions for other sets related to SLE. We also mention \cite{Law15}, where Lawler proved the existence of the Minkowski content of an SLE curve intersected with the real line, which is related to what was done in \cite{ABV16} and what we do here. Lastly, we mention \cite{BS09}, where the authors computed an average integral means spectrum of SLE.

As for notation, we write $f(x) \asymp g(x)$ if there is a constant $C$ such that $C^{-1} g(x) \leq f(x) \leq C g(x)$ and $f(x) \lesssim g(x)$ if there is a constant $C$ such that $f(x) \leq C g(x)$, and the constants do not depend on $f$, $g$ or $x$. We say that $\phi$ is a subpower function if $\lim_{x \rightarrow \infty} \phi(x) x^{-\epsilon} = 0$ for every $\epsilon > 0$. In the same way, we say that $\psi$ is a subexponential function if for every $\epsilon > 0$, $\lim_{x \rightarrow \infty} \psi(x) e^{-\epsilon x} = 0$. In what follows, implicit constants, subpower functions and subexponential functions may change between the lines, without a change of notation.

\section{Preliminaries}\label{prel}
We begin by introducing some preliminaries on complex analysis, $\SLE_\kappa(\underline{\rho})$ processes, the Gaussian free field and imaginary geometry.

\subsection{Measuring distances and sizes} \label{CA}
Let $D$ be a simply connected domain, $z \in D$, and let $\varphi: D \rightarrow \mathbb{D}$ be a conformal map of $D$ onto $\mathbb{D}$ such that $\varphi(z) = 0$. We define the \textit{conformal radius} of $D$ with respect to $z$ as
\begin{align*}
    \text{crad}_D(z) = \frac{1}{|\varphi'(z)|}.
\end{align*}
It behaves well under conformal transformations; if $f:D \rightarrow f(D)$ is a conformal transformation, then $\text{crad}_{f(D)}(f(z)) = |f'(z)| \text{crad}_D(z)$, $z \in D$, (that is, it is \textit{conformally covariant}) and by the Schwarz lemma and the Koebe 1/4 theorem, one easily sees that 
\begin{align}\label{eq:distcrad}
    \dist(z,\partial D) \leq \text{crad}_D(z) \leq 4 \dist(z,\partial D).
\end{align}

For any domain $\UH\setminus U$, where $U\subset \UH$ is bounded, we define the harmonic measure from infinity of $E \subset \partial(\UH \setminus U)$ as
\begin{align}\label{eq:hinfty}
    \omega_\infty(E,\UH\setminus U) = \lim_{y \rightarrow \infty} y \omega(iy,E,\UH\setminus U),
\end{align}
where $\omega$ denotes the usual harmonic measure. It turns out that $\omega_\infty$ is a very convenient notion of size of subsets of $\UH$. We say that $A$ is a compact $\UH$-hull if $A = \UH \cap \overline{A}$ and $\UH \setminus A$ is simply connected and we let $\mathcal{Q}$ be the set of compact $\UH$-hulls. By Proposition 3.36 of \cite{Law05}, there exists, for every $A\in\mathcal{Q}$, a unique conformal map $g_A:\UH\setminus A \rightarrow \UH$ which satisfies the hydrodynamic normalization, i.e.,
\begin{align}\label{eq:hydro}
    \lim_{z \rightarrow \infty} |g_A(z)-z| = 0.
\end{align}
The function $g_A$ is called the mapping-out function of $A$. By \eqref{eq:hydro} and the conformal invariance of $\omega$, we have that
\begin{align}\label{eq:harminflength}
    \omega_\infty(E,\UH\setminus A) &= \lim_{y\rightarrow \infty} y\omega(g_A(iy),g_A(E),\UH) \nonumber\\
    &= \lim_{y \rightarrow \infty} \int_{g_A(E)} \frac{y \text{Im}(g_A(iy))}{(t-\text{Re}(g_A(iy)))^2 + \text{Im}(g_A(iy))^2} \frac{dt}{\pi} \nonumber\\
    &= \int_{g_A(E)} \frac{dt}{\pi} = \frac{|g_A(E)|}{\pi},
\end{align}
where $|g_A(E)|$ is the total length of the set $g_A(E) \subset \R$.

While $\omega$ is defined on the boundary of the domain, it is convenient to speak about the harmonic measure of subsets of the domain, so we make the following extension: if $U \subset D$, then
\begin{align}\label{eq:harmextend}
    \omega(z,U,D) \coloneqq \omega(z,D \cap \partial U, D\setminus U).
\end{align}
This extends naturally to the case of $\omega_\infty$ as well.

\subsection{$\SLE_\kappa(\underline{\rho})$ processes}\label{secslekr}
In this section we will introduce $\SLE_\kappa$ and $\SLE_\kappa(\underline{\rho})$ processes. As stated in the introduction, a chordal $\SLE_\kappa$ Loewner chain is the collection of random conformal maps $(g_t)_{t \geq 0}$, given by solving (\ref{eq:LDE}) with $W_t = \sqrt{\kappa} B_t$, where $B_t$ is a standard Brownian motion with $B_0 = 0$ and filtration $\mathscr{F}_t$, satisfying \eqref{eq:hydro}. We define $(f_t)_{t\geq 0}$ to be the centered Loewner chain, that is,
\begin{align*}
    f_t(z) = g_t(z) - W_t.
\end{align*}
For fixed $z \in \UH$, the solution to (\ref{eq:LDE}) exists until time $T_z = \inf \{t \geq 0: f_t(z) = 0 \}$. The domain of $g_t$ is $H_t = \UH \setminus K_t$ where $K_t = \{z: T_z \leq t \} \in \mathcal{Q}$ is the SLE hull at time $t$ and $g_t$ is the unique conformal map from $H_t$ onto $\UH$ such that $\lim_{z \rightarrow \infty} |g_t(z) - z| = 0$. Rohde and Schramm proved that the family of $\SLE_\kappa$ hulls is almost surely generated by a curve $\eta:[0,\infty] \rightarrow \overline{\UH}$, i.e., $H_t$ is the unbounded component of $\UH \setminus \eta([0,t])$ (see \cite{RS05}). We call $\eta$ the $\SLE_\kappa$ process or $\SLE_\kappa$ curve and say that $K_t$ is the filling of $\eta([0,t])$.

Now we will define the $\SLE_\kappa(\underline{\rho})$ process. Let $\underline{x}_L = (x_{l,L},\hdots,x_{1,L})$, $\underline{x}_R = (x_{1,R},\hdots,x_{r,R})$, where $x_{l,L} < \hdots < x_{1,L} \leq 0 \leq x_{1,R} < \hdots < x_{r,R}$. Also, let $\underline{\rho}_L = (\rho_{1,L},\hdots,\rho_{l,L})$, $\underline{\rho}_R = (\rho_{1,R},\hdots,\rho_{r,R})$, where $\rho_{j,q} \in \R$, $q \in \{L,R\}$. We call $\rho_{j,q}$ the weight of $x_{j,q}$. Let $W_t$ be the solution to the system of SDEs
\begin{align}\label{eq:krdriver}
    dW_t &= \sum_{j=1}^l \frac{\rho_{j,L}}{W_t - V_t^{j,L}} dt + \sum_{j=1}^r \frac{\rho_{j,R}}{W_t-V_t^{j,R}} dt + \sqrt{\kappa}dB_t, \\
    dV_t^{j,q} &= \frac{2}{V_t^{j,q}-W_t} dt, \quad V_0^{j,q} = x_{j,q}, \quad j =1,\dots,N_q, \quad q \in \{L,R\}, \nonumber
\end{align}
where $N_L= l$ and $N_R = r$. An $\SLE_\kappa(\underline{\rho}_L;\underline{\rho}_R)$ Loewner chain with force points $(\underline{x}_L;\underline{x}_R)$ is the family of conformal maps $(g_t)_{t \geq 0}$ obtained by solving (\ref{eq:LDE}) with $W_t$ being the solution to ($\ref{eq:krdriver}$). The $\SLE_\kappa(\underline{\rho}_L;\underline{\rho}_R)$ hulls, $(K_t)$, are defined analogously and they are almost surely generated by a continuous curve, $\eta$, the $\SLE_\kappa(\underline{\rho}_L;\underline{\rho}_R)$ process or $\SLE_\kappa(\underline{\rho}_L;\underline{\rho}_R)$ curve (see Theorem 1.3 in \cite{MS16a}). $\SLE_\kappa(\underline{\rho}_L;\underline{\rho}_R)$ is a generalization of $\SLE_\kappa$ ($\SLE_\kappa$ = $\SLE_\kappa(0;0)$), where one also keeps track of the force points and their assigned weights either attract ($\rho_{j,q} < 0$) or repel ($\rho_{j,q} > 0$) the curve. If $\eta$ is an $\SLE_\kappa(\underline{\rho})$ curve, we write $\eta \sim \text{SLE}_\kappa(\underline{\rho})$. Exactly how the weights of the force points affect the curve $\eta$ is explained in Lemma \ref{SLEkrprop}.

The solution to the system of SDEs (\ref{eq:krdriver}) exists up until the \textit{continuation threshold} is hit, that is, the first time $t$ such that either
\begin{align*}
    \sum_{j:V_t^{j,L}=W_t} \rho_{j,L} \leq -2 \quad \text{or} \quad \sum_{j:V_t^{j,R}=W_t} \rho_{j,R} \leq -2,
\end{align*}
as is explained in Section 2.2 of \cite{MS16a}. Moreover, for every $t>0$ before the continuation threshold, $\prob(W_t=V_t^{j,q}) = 0$ for $j \in \N$ and $q\in \{L,R\}$.

Geometrically, hitting the continuation threshold means the curve $\eta$ swallowing force points on either side such that the sum of their weights is less than $-2$, that is, $\eta$ hits an interval $(x_{m+1,L},x_{m,L})$ (or $(x_{n,R},x_{n+1,R})$) such that $\sum_{j=1}^m \rho_{j,L} \leq -2$ (resp. $\sum_{j=1}^n \rho_{j,R} \leq -2$). 

Now, we write $\rho_{0,L} = \rho_{0,R} = 0$, $x_{0,L} = 0^-$, $x_{l+1,L} = -\infty$, $x_{0,R} = 0^+$ and $x_{r+1,R} = +\infty$, and let for $q \in \{L,R\}$ and $j \in \N$
\begin{align*}
    \overline{\rho}_{j,q} = \sum_{k=0}^j \rho_{k,q}.
\end{align*}
The following lemma describes the interaction $\eta$ with the real line. It is written down in \cite{MW17}, and just as they did, we refer to Remark 5.3 and Theorem 1.3 of \cite{MS16a} and Lemma 15 of \cite{Dub09} for the proof. 
\begin{lem}\label{SLEkrprop}
Let $\eta$ be an $\SLE_\kappa(\underline{\rho}_L;\underline{\rho}_R)$ curve in $\UH$, from $0$ to $\infty$, with force points $(\underline{x}_L;\underline{x}_R)$. Then,
\begin{enumerate}[(i)]
\item if $\overline{\rho}_{k,R} \geq \frac{\kappa}{2}-2$, then $\eta$ almost surely does not hit $(x_{k,R},x_{k+1,R})$,
\item if $\kappa \in (0,4)$ and $\overline{\rho}_{k,R} \in (\frac{\kappa}{2}-4,-2]$, then $\eta$ can hit $(x_{k,R},x_{k+1,R})$, but then can not be continued afterwards,
\item if $\kappa > 4$ and $\overline{\rho}_{k,R} \in (-2,\frac{\kappa}{2}-4]$, then $\eta$ can hit $(x_{k,R},x_{k+1,R})$, be continued afterwards and $\eta \cap (x_{k,R},x_{k+1,R})$ is almost surely an interval,
\item if $\overline{\rho}_{k,R} \in ((-2)\vee (\frac{\kappa}{2}-4),\frac{\kappa}{2}-2)$, then $\eta$ can hit and bounce off of $(x_{k,R},x_{k+1,R})$ and $\eta \cap (x_{k,R},x_{k+1,R})$ has empty interior.
\end{enumerate}
The same holds if we replace $R$ by $L$ and consider $(x_{k+1,L},x_{k,L})$.
\end{lem}
Note that in \textit{(ii)} in the above lemma, the curve has swallowed force points with a total weight at least as negative as $-2$, and hence it cannot be continued.  In \textit{(iii)} and \textit{(iv)}, the total weight of the force points swallowed is greater than $-2$, and hence the curve can be continued.

The following was proved in \cite{MS16a} (see also Lemma 2.2 in \cite{MW17}).
\begin{lem}\label{krconv}
Fix $\kappa>0$ and let $(\underline{x}_L^n)$ and $(\underline{x}_R^n)$ be sequences of vectors of numbers $x_{l,L}^n<\dots<x_{1,L}^n<0<x_{1,R}^n<\dots<x_{r,R}^n$, converging to vectors $(\underline{x}_L)$ and $(\underline{x}_R)$ such that $x_{1,L} = 0^-$ and $x_{1,R} = 0^+$. For each $n$, denote by $(W^n,V^{n,L},V^{n,R})$ the driving processes of an $\SLE_\kappa(\underline{\rho}_L,\underline{\rho}_R)$ process with force points $(\underline{x}_L^n;\underline{x}_R^n)$. Then $(W^n,V^{n,L},V^{n,R})$ converges weakly in law, with respect to the local uniform topology, to the driving process $(W,V^L,V^R)$ of an $\SLE_\kappa(\underline{\rho}_L;\underline{\rho}_R)$ with force points $(\underline{x}_L;\underline{x}_R)$, as $n \rightarrow \infty$.
\end{lem}

It turns out that if we, using the Girsanov theorem, reweight an $\SLE_\kappa$ process by a certain martingale (how this is done is explained briefly below), then we obtain an $\SLE_\kappa(\underline{\rho})$ process at least until the first time that $W_t = V_t^{j,q}$ for some $(j,q)$. Let $x_{1,L} < 0 < x_{1,R}$ and define 
\begin{align}\label{eq:krmtg}
    M_t = &\prod_{j,q} |g_t'(x_{j,q})|^{\frac{(4-\kappa+\rho_{j,q})\rho_{j,q}}{4\kappa}} \prod_{j,q} |W_t-V_t^{j,q}|^{\frac{\rho_{j,q}}{\kappa}} \nonumber \\
    &\times \prod_{(j_1,q_1) \neq (j_2,q_2)} |V_t^{j_1,q_1}-V_t^{j_2,q_2}|^{\frac{\rho_{j_1,q_1}\rho_{j_2,q_2}}{2\kappa}}.
\end{align}
Then $M_t$ is a local martingale and an $\SLE_\kappa$ process weighted by $M_t$ has the law of an $\SLE_\kappa(\underline{\rho}_L;\underline{\rho}_R)$ process with force points $(\underline{x}_L;\underline{x}_R)$ (see Theorem 6 of \cite{SW05}).

So far, we have only defined chordal $\SLE_\kappa(\underline{\rho})$ processes in $\UH$, but we can define them in any Jordan domain, by a conformal coordinate change. More precisely, an $\SLE_\kappa(\underline{\rho}_L;\underline{\rho}_R)$ in a Jordan domain $D$, from $z_0$ to $z_\infty$, with force points $(\underline{x}_L,\underline{x}_R)$ is the image of an $\SLE_\kappa(\underline{\rho}_L;\underline{\rho}_R)$ in $\UH$ from $0$ to $\infty$ under a conformal map $\varphi$ such that $\varphi(\UH)=D$, $\varphi(0) = z_0$, $\varphi(\infty)=z_\infty$ and such that the force points of the $\SLE_\kappa(\underline{\rho}_L;\underline{\rho}_R)$ in $\UH$ are mapped to $(\underline{x}_L;\underline{x}_R)$. We say that the constructed $\SLE_\kappa(\underline{\rho}_L;\underline{\rho}_R)$ in $D$ is an $\SLE_\kappa(\underline{\rho}_L;\underline{\rho}_R)$ process with \textit{configuration} 
\begin{align}\label{config}
    c = (D, z_0, \underline{x}_L, \underline{x}_R, z_\infty).
\end{align}    
The configuration of the $\SLE_\kappa(\underline{\rho}_L;\underline{\rho}_R)$ process we defined in the beginning of the section is \newline
$(\UH, 0, \underline{x}_L, \underline{x}_R, \infty)$.

\subsubsection{The case of one force point}
Let $a=2/\kappa$. We will parametrize the $\SLE_\kappa(\rho)$ so that $\text{hcap}(K_t) = at$, i.e., as the solution to
\begin{align}
    \partial_t g_t(z) = \frac{a}{g_t(z) - W_t}, \qquad g_0(z) = z,
\end{align}
with
\begin{align} \label{eq:driver}
    dW_t = dB_t + \frac{a\rho/2}{W_t - V_t}dt, \ W_0 = 0; \ dV_t = \frac{a}{V_t - W_t}dt, \ V_0 = x_R,
\end{align}
where $B_t$ is a one-dimensional standard Brownian motion with $B_0 = 0$ and filtration $\mathscr{F}_t$. We say that the conformal maps $(g_t)_{t \geq 0}$ are driven by $W$. The solution to (\ref{eq:driver}) exists for all $t \geq 0$, if $\kappa > 0$ and $\rho > -2$, so henceforth, we assume this. If $\rho \geq \frac{\kappa}{2}-2$, then $\eta$ does not hit the real line, almost surely, and hence we are interested in the case $\rho < \frac{\kappa}{2} -2$. If $\kappa > 4$ and $\rho \in (-2,\frac{\kappa}{2}-4]$, then by Lemma \ref{SLEkrprop}, $\eta \cap (x_R,\infty)$ is almost surely an interval, and thus we will consider $\rho \in ((-2) \vee (\frac{\kappa}{2}-4), \frac{\kappa}{2}-2)$.

Fix $\kappa > 0$ and $\rho > -2$ and let $(K_t, t \geq 0)$ be the hulls of an $\SLE_\kappa(\rho)$ process with force point $x_R$. The $\SLE_\kappa(\rho)$ satisfies two important properties. The first is the following scaling rule: for any $m > 0$, $(m^{-1} K_{m^2 t},t \geq 0)$ has the same law as the hulls of an $\SLE_\kappa(\rho)$ process with force point $x_R/m$. If $x_R = 0^+$, then it is scaling invariant. The second is the \textit{domain Markov property}: for any finite stopping time, $\tau$, the curve defined as $(\hat{\eta}(t) = f_\tau(\eta(t+\tau)), t \geq 0)$ is an $\SLE_\kappa(\rho)$ curve with force point $V_\tau - W_\tau$, where again, $f_t =g_t-W_t$.

Note that $f_t$ follows the SDE
\begin{align} \label{eq:dh}
    df_t(z) = \frac{a}{f_t(z)}dt - dW_t = \left[ \frac{a}{f_t(z)} + \frac{a\rho/2}{W_t - V_t} \right]dt - dB_t.
\end{align}
Taking the spatial derivative of (\ref{eq:dh}) results in an ODE that, upon solving, yields
\begin{align}\label{eq:gprime}
   g_t'(z) = f_t'(z) = \exp \left\{ -\int_0^t \frac{a}{f_s(z)^2}ds \right\}.
\end{align}
While $g_t$ is defined in $H_t$, it (and hence $f_t$ and $g_t'$) extends continuously to the real line and is real-valued there. For $x \in \R_+$, $g_t'(x) \in [0,1]$ and is decreasing in $t$. Due to symmetry, it is enough to consider $x, x_R \in \R_+$. By applying Itô's formula to $\log f_t(z)$ and exponentiating, we see that
\begin{align}\label{eq:ft}
    f_t(z) = z\exp \left\{ \int_0^t \frac{a-1/2}{f_s(z)^2}ds - \int_0^t \frac{a\rho/2}{f_s(z)(W_s-V_s)}ds - \int_0^t \frac{dB_s}{f_s(z)} \right\}.
\end{align}
The same procedure, applied to $v_t(z) \coloneqq g_t(z) - V_t = f_t(z) + W_t - V_t$, yields
\begin{align}\label{eq:vt}
    v_t(z) = (z-x_R)\exp \left\{ \int_0^t \frac{a}{f_s(z)(W_s-V_s)}ds \right\}.
\end{align}
Observe that, considered as functions on $\R_+$, $f_t$, $g_t'$ and $v_t$ are increasing in $x$. We will mostly work with these functions on or close to $\R_+$.

\subsubsection{Local martingales and weighted measures}\label{weighted}
Fix $\kappa > 0$, $\rho \in ((-2)\vee(\frac{\kappa}{2}-4), \frac{\kappa}{2} - 2)$ and let, as above, $a = 2/\kappa$. For each such pair $(\kappa,\rho)$, we will define a one-parameter family of local martingales which will play a major role in our analysis. Let $\zeta$ be the variable with which we will parametrize the martingales and let
\begin{align}
    \mu_c = 2a - \frac{1}{2} + \frac{a\rho}{2}.
\end{align}
For $-\frac{\mu_c^2}{2a} < \zeta < \infty$, define the parameters
\begin{align}
    \mu = \mu_c + \sqrt{\mu_c^2 + 2a\zeta}, \quad \beta = \frac{a}{\sqrt{\mu_c^2 + 2a\zeta}}.
\end{align}
Note that with our choice of $\zeta$, $\mu > \mu_c$ and $\beta > 0$. The restriction $\mu > \mu_c$ is necessary for a certain invariant density of a diffusion to exist, see \eqref{eq:indensQ} and the appendix. The parameters are related as follows:
\begin{align*}
    \mu_c < \mu < \infty &: \zeta = \frac{\mu}{2a}(\mu-2\mu_c), \quad \beta = \frac{a}{\mu-\mu_c}, \\
    0 < \beta < \infty &: \mu = \frac{a}{\beta} + \mu_c, \quad \zeta = \frac{1}{2a}\left(\frac{a}{\beta}+\mu_c \right) \left(\frac{a}{\beta}-\mu_c \right).
\end{align*}
For each $x > 0$, we have by equations \eqref{eq:gprime}, \eqref{eq:ft} and \eqref{eq:vt}, that
\begin{align}
    M_t^\zeta(x) &= g_t'(x)^{\zeta + \mu(1+\rho/2)} f_t(x)^{-\mu} v_t(x)^{-\mu\rho/2} \nonumber \\
    &= x^{-\mu} (x-x_R)^{-\mu\rho/2} \exp \left\{ \mu \int_0^t \frac{dB_s}{f_s(x)} - \frac{\mu^2}{2} \int_0^t \frac{ds}{f_s(x)^2} \right\}
\end{align}
is a local martingale on $0 \leq t \leq T_x$ (see Theorem 6 of \cite{SW05}), such that
\begin{align*}
    \frac{dM_t^\zeta (x)}{M_t^\zeta (x)} = \frac{\mu}{f_t(x)} dB_t, \quad M_0^\zeta (x) = x^{-\mu} (x-x_R)^{-\mu\rho/2}.
\end{align*}
Note that under the measure weighted by $M_t^\zeta$, the $\SLE_\kappa(\rho)$ process becomes an $\SLE_\kappa(\rho,-\mu\kappa)$ process with force points $(x_R,x)$. We shall write the local martingale in a different way, which is very convenient for our analysis. Define the random processes
\begin{align}\label{eq:deltanormal}
    \delta_t (x) = \frac{v_t(x)}{g_t'(x)} = \frac{g_t(x)-V_t}{g_t'(x)}
\end{align}
and
\begin{align}\label{eq:Q}
    Q_t(x) = \frac{v_t(x)}{f_t(x)} = \frac{g_t(x)-V_t}{g_t(x)-W_t}.
\end{align}
We will often not write out the dependence on $x$, as it will cause no confusion. Since we $V_t \geq W_t$, we see that $Q_t(x) \in [0,1]$ for every $t\geq 0$. We note that if $x_R = 0^+$, then $Q_t(x)$ is the ratio between the harmonic measure from infinity of two sets, more precisely, if $\eta^R$ denotes the right side of the curve $\eta$ and $r_t$ is the rightmost point of $\eta([0,t])\cap\R$, then
\begin{align}\label{eq:Qharm}
    Q_t(x) = \frac{\omega_\infty([r_t,x],\UH \setminus \eta([0,t]))}{\omega_\infty(\eta^R([0,t])\cup [r_t,x],\UH\setminus \eta([0,t]))}.
\end{align}
With these processes, we can write
\begin{align} \label{eq:mtg}
    M_t^\zeta(x) = g_t'(x)^\zeta Q_t(x)^\mu \delta_t(x)^{-\mu(1+\frac{\rho}{2})}.
\end{align}
This will be very convenient, as we will relate the process $\delta_t(x)$ to the conformal radius at the point $x$, and thus, it will be comparable to $\dist(x,\eta([0,t]))$. With this in mind, we will then make a random time change so that the time-changed process decays deterministically, and it will give us good control over the decay of the distance between the curve and a certain point. It does not, however, make sense to talk about the conformal radius of a boundary point, so we will begin by sorting this out.

We let $\widehat{H}_t = H_t \cup \{z:\overline{z} \in H_t \} \cup \{x \in \R_+: t<T_x \}$ be the union of the reflected domain and the points on $\R_+$ that have not been swallowed at time $t$. Then, it makes sense to talk about $\text{crad}_{\widehat{H}_t}(x)$ for $0<x\in \widehat{H}_t$, and a calculation shows that if $x_R = 0^+$, then for $t < T_x$,
\begin{align*}
    \delta_t(x) = \frac{1}{4} \text{crad}_{\widehat{H}_t}(x).
\end{align*}
By \eqref{eq:distcrad}, this implies that
\begin{align}
    \frac{1}{4} \dist(x,\eta([0,t])) \leq \delta_t(x) \leq \dist(x,\eta([0,t])).
\end{align}
While this only holds for $x_R = 0^+$, similar bounds can be acquired for other $x_R$, as the following lemma shows. See also \cite{WW15}.
\begin{lem}\label{deltadist}
Let $\kappa > 0$, $\rho > -2$, and let $\eta$ be an $\SLE_\kappa(\rho)$ curve with force point $x_R \geq 0$. Let $(g_t)_{t \geq 0}$ be the conformal maps, driven by $W$, and let $T_x$ be the swallowing time of $x$. Let $x > x_R$ and $0<t<T_x$, then
\begin{align*}
    \frac{x-x_R}{4x}\dist(x,\eta([0,t])) \leq \delta_t(x) \leq 4 \dist(x,\eta([0,t])).
\end{align*}
\end{lem}
\begin{proof}
Denote by $(K_t)_{t \geq 0}$ the compact hulls corresponding to the $\SLE_\kappa(\rho)$ process. Extend the maps $(g_t)_{t \geq 0}$ by Schwarz reflection to $\mathbb{C} \setminus ( K_t \cup \{ z : \overline{z} \in K_t \} \cup \R_-)$, let $r_t$ be the rightmost point of $K_t \cap \R$ and let $O_t = g_t(r_t)$. Then, by the Koebe 1/4 theorem,
\begin{align*}
    \frac{1}{4} \frac{\dist(g_t(x),O_t)}{\dist(x,\eta([0,t]))} \leq g_t'(x) \leq 4\frac{\dist(g_t(x),O_t)}{\dist(x,\eta([0,t]))},
\end{align*}
that is,
\begin{align*}
    \frac{1}{4} \dist(x,\eta([0,t])) \leq \frac{g_t(x)-O_t}{g_t'(x)} \leq 4 \dist(x,\eta([0,t])),
\end{align*}
since $\dist(g_t(x),O_t) = g_t(x) - O_t \geq g_t(x) - V_t$, and thus the upper bound is done. Since $V_t = g_t(x_R)$, we have that if $x_R = 0^+$, $V_t = O_t$, and the proof is done. Thus, we now consider $x_R > 0$. For $t \geq T_{x_R}$, again $O_t = V_t$, so we now consider $t<T_{x_R}$. Let $u \in (0,x_R]$ and define
\begin{align*}
    G_u(t) = \frac{g_t(x)-g_t(u)}{g_t(x)-g_t(x_R)}.
\end{align*}
for $t< T_u$. Then,
\begin{align*}
    G_u(0) = \frac{x-u}{x-x_R}.
\end{align*}
Using the Loewner equation, we see that
\begin{align*}
    G_u'(t) = \frac{a(g_t(x)-g_t(u))(g_t(u)-g_t(x_R))}{(g_t(x)-g_t(x_R))(g_t(x)-W_t)(g_t(x_R)-W_t)(g_t(u)-W_t)} \leq 0,
\end{align*}
since $g_t(x)-g_t(u) \geq 0$ and $g_t(u)-g_t(x_R) \leq 0$. Thus, for every $u \in (0,x_R]$ and $t < T_u$,
\begin{align*}
    \frac{g_t(x)-g_t(u)}{g_t(x)-g_t(x_R)} \leq \frac{x-u}{x-x_R}.
\end{align*}
Fixing $t$ and letting $u \rightarrow r_t$, we get
\begin{align*}
    \frac{g_t(x)-O_t}{g_t(x)-V_t} = \frac{g_t(x)-g_t(r_t)}{g_t(x)-g_t(x_R)} \leq \frac{x-r_t}{x-x_R} \leq \frac{x}{x-x_R},
\end{align*}
which implies that
\begin{align*}
    \frac{1}{4} \dist(x,\eta([0,t])) \leq \frac{g_t(x) - O_t}{g_t'(x)} \leq \frac{x}{x-x_R} \frac{g_t(x)-V_t}{g_t'(x)},
\end{align*}
and thus
\begin{align*}
    \frac{x-x_R}{4x}\dist(x,\eta([0,t])) \leq \frac{g_t(x) - V_t}{g_t'(x)} \leq 4 \dist(x,\eta([0,t])).
\end{align*}
\end{proof}
With this lemma in mind, we will proceed and make a random time change. The process $\delta_t$ satisfies the (stochastic) ODE
\begin{align}
    d\delta_t(x) = - a \delta_t(x) \frac{v_t(x)}{f_t(x)^2(V_t-W_t)}dt = - a \delta_t \frac{Q_t(x)}{1-Q_t(x)}\frac{dt}{f_t(x)^2},
\end{align}
so we define the process $\tilde{t}(s) = \tilde{t}_x(s)$ as the solution to the equation
\begin{align}\label{eq:timechange}
    s = \int_0^{\tilde{t}(s)} \frac{Q_u(x)}{1-Q_u(x)}\frac{du}{f_u(x)^2}.
\end{align}
Then, $\tilde{\delta}_s(x) \coloneqq \delta_{\tilde{t}(s)}(x)$ satisfies the ODE $d\tilde{\delta}_s(x) = -a\tilde{\delta}_s(x) ds$, i.e.,
\begin{align} \label{eq:delta}
    \tilde{\delta}_s(x) = \delta_0(x) e^{-as} = (x-x_R) e^{-as}.
\end{align}
This time change is called the radial parameterization. Note that this time change is depending on $x$. We let $\tilde{g}_s = g_{\tilde{t}(s)}$ etc. denote the time-changed processes. They are all adapted to the filtration $\tilde{\mathscr{F}}_s = \mathscr{F}_{\tilde{t}(s)}$. By differentiating \eqref{eq:timechange} to get an expression for $\frac{d}{ds}\tilde{t}(s)$, and combining this with equations \eqref{eq:gprime}, \eqref{eq:ft} and \eqref{eq:vt}, we have that in this new parametrization,
\begin{align} \label{eq:g'tc}
    \tilde{g}_s'(x) = \exp\left\{-a \int_0^s \frac{1-\tilde{Q}_u(x)}{\tilde{Q}_u(x)} du \right\},
\end{align}
and $\tilde{Q}_s$ follows the SDE
\begin{align*}
    d\tilde{Q}_s = ((1-2a-a\rho/2) - (1-a)\tilde{Q}_s)ds + \sqrt{\tilde{Q}_s(1-\tilde{Q}_s)}d\tilde{B}_s,
\end{align*}
where $\tilde{B}_s$ is a Brownian motion with respect to the filtration $\tilde{\mathscr{F}}_s$. Combining (\ref{eq:mtg}) and (\ref{eq:delta}), we see that
\begin{align}\label{eq:Mtilde}
    \tilde{M}_s^\zeta(x) = (x-x_R)^{-\mu(1+\rho/2)} \tilde{g}_s'(x)^\zeta \tilde{Q}_s(x)^\mu e^{a\mu(1+\rho/2)s}.
\end{align}
For each $x>0$ and $\zeta > -\frac{\mu_c^2}{2a}$, we define the new probability measure, which we denote by $\prob^* = \prob_{x,\zeta}^*$, as
\begin{align*}
    \prob^*(A) = \tilde{M}_0^\zeta(x)^{-1} \E\left[ \tilde{M}_s^\zeta(x)1_A \right],
\end{align*}
for every $A \in \tilde{\mathscr{F}}_s$. We denote the expectation with respect to $\prob^*$ by $\E^*$. Let $\tilde{B}_s^*$ denote a Brownian motion with respect to $\prob^*$. By the Girsanov theorem, we have that under the measure $\prob^*$, $\tilde{Q}_s$ follows the SDE
\begin{align}\label{eq:Qs}
    d\tilde{Q}_s = \left[(1-2a-a\rho/2+\mu) - (1-a+\mu)\tilde{Q}_s \right]ds + \sqrt{\tilde{Q}_s(1-\tilde{Q}_s)} d\tilde{B}_s^*.
\end{align}
Under $\prob^*$, $\tilde{Q}_s$ is positive recurrent and has the invariant density
\begin{align}\label{eq:indensQ}
    p_{\tilde{Q}}(x) = \tilde{c}x^{1-4a-a\rho+2\mu}(1-x)^{2a+a\rho-1},
\end{align}
where
\begin{align*}
    \tilde{c} = \frac{\Gamma(2-2a+2\mu)}{\Gamma(2a+a\rho)\Gamma(2-4a-a\rho+2\mu)},
\end{align*}
see Corollary \ref{Qprop}.
Throughout, we will denote by $\tilde{X}_s$, a process that follows the same SDE as $\tilde{Q}_s$, but started according to the invariant density. For $s \geq 0$ and $y > 0$, short calculations show that
\begin{align}\label{eq:basiccalc}
    \prob^*( \tilde{X}_s \leq y ) \lesssim y^{2\mu-4a-a\rho+2}, \quad \E^* \big[ \tilde{X}_s^{-\mu} \big] = \frac{\Gamma(2-2a+2\mu) \Gamma(2-4a-a\rho+\mu)}{\Gamma(2-2a+\mu) \Gamma(2-4a-a\rho+2\mu)}.
\end{align}
In Section \ref{secOPE} we will use that $\prob^*(\tilde{t}(s)<\infty) = 1$ for every $s$ (which is shown in Appendix A), and the following lemma.
\begin{lem}
\label{time}
Fix $x$ and let $\tau_s$ and $\tilde{t}(s)$ be as above. Then
\begin{align*}
    \tilde{t}\left(\left(\frac{s}{a} + \frac{1}{a} (\log (x-x_R) - \log 4)\right) \vee 0\right) \leq \tau_s \leq \tilde{t}\left(\frac{s}{a} + \frac{1}{a}(\log x + \log 4)\right),
\end{align*}
for all $s > \max\{0,-\log(x-x_R)\}$.
\end{lem}
\begin{proof}
Assume that $s > \max\{0,-\log (x-x_R)\}$. Trivially, $\tau_s \geq \tilde{t}(0) = 0$. By Lemma \ref{deltadist}, we have that $\delta_{\tau_s} \leq 4 e^{-s}$ and thus, if $s - \log 4 +\log (x-x_R) > 0$, then (recall that $\delta_0 = x-x_R$)
\begin{align*}
    \delta_{\tau_s} \leq 4 e^{-s} = (x-x_R) e^{-s + \log 4 -\log (x-x_R)} = \delta_{\tilde{t}(\frac{s}{a} - \frac{1}{a} \log 4 + \frac{1}{a} \log (x-x_R))},
\end{align*}
and since $t \mapsto \delta_t$ is decreasing, $\tau_s \geq \tilde{t}((\frac{s}{a} + \frac{1}{a} \log (x-x_R) - \frac{1}{a} \log 4) \vee 0)$.

Also, by Lemma \ref{deltadist}, we get
\begin{align*}
    \delta_{\tau_s} \geq \frac{x-x_R}{4x}e^{-s} = (x-x_R) e^{-s -\log 4 -\log x} = \delta_{\tilde{t}(\frac{s}{a} + \frac{1}{a}\log x + \frac{1}{a} \log 4)}.
\end{align*}
Thus, $\tau_s \leq \tilde{t}\left(\frac{s}{a} + \frac{1}{a}(\log x + \log 4)\right)$, and the proof is done.
\end{proof}
In what follows, $x$ and $x_R$ will be kept constant (time is the parameter which will change), and hence we will instead treat the inequality as
\begin{align}\label{C*}
    \tilde{t}\left(\left(\frac{s}{a} -C^* \right) \vee 0\right) \leq \tau_s \leq \tilde{t}\left(\frac{s}{a} + C^*\right),
\end{align}
as we can choose such a $C^* = C^*(x,x_R)$ for each $x>1$, $x_R < x$. The only place where we have to be careful with this is in the proof of Lemma \ref{lem3}, but we will discuss that there.

\subsection{The Gaussian free field}
We will now introduce and discuss the Gaussian free field (GFF), for more on the GFF, see \cite{She07}. Let $D \subset \mathbb{C}$ be a Jordan domain and let $C_c^\infty(D)$ be the set of compactly supported smooth functions on $D$. The \textit{Dirichlet inner product} on $D$ is defined as
\begin{align*}
(f,g)_\nabla = \frac{1}{2\pi} \int_D \nabla f(z) \cdot \nabla g(z) dz,
\end{align*}
and the Hilbert space closure of $C_c^\infty(D)$ with this inner product is the Sobolev space $H(D) \coloneqq H_0^1(D)$. Let $( \phi_n )_{n=1}^\infty$ be a $(\cdot,\cdot)_\nabla$-orthonormal basis of $H(D)$. Then, the zero-boundary GFF $h$ on $D$ can be expressed as
\begin{align*}
    h = \sum_{n=1}^\infty \alpha_n \phi_n,
\end{align*}
where $( \alpha_n )_{n=1}^\infty$ is a sequence of i.i.d. $N(0,1)$ random variables. This does not converge in any space of functions, however, it converges almost surely in a space of distributions. The GFF is conformally invariant, i.e., if $h$ is a GFF on $D$ and $\varphi: D \rightarrow \widetilde{D}$ is a conformal transformation, then $h \circ \varphi^{-1} = \sum_n \alpha_n \phi_n \circ \varphi^{-1}$ is a GFF on $\widetilde{D}$.

If we denote the standard $L^2(D)$ inner product by $(\cdot,\cdot)$ and $U \subseteq D$ is open, then for $f \in C_c^\infty(U)$, $g \in C_c^\infty(D)$, we obtain by integration by parts that
\begin{align*}
    (f,g)_\nabla = -\frac{1}{2\pi} (f,\Delta g),
\end{align*}
that is, every function in $C_c^\infty(D)$ which is harmonic in $U$ is $(\cdot,\cdot)_\nabla$-orthogonal to every function in $C_c^\infty(U)$. From this, one can see that $H(D)$ can be $(\cdot,\cdot)_\nabla$-orthogonally decomposed as $H(U) \oplus H^\perp(U)$, where $H^\perp(U)=H_D^\perp(U)$ is the set of functions in $H(D)$ which are harmonic in $U$ and have finite Dirichlet energy. Hence, we may write
\begin{align*}
    h = h_U + h_{U^\perp} = \sum_n \alpha_n^U \phi_n^U + \sum_n \alpha_n^{U^\perp} \phi_n^{U^\perp},
\end{align*}
where $( \alpha_n^U )$ and $( \alpha_n^{U^\perp} )$ are independent sequences of i.i.d. $N(0,1)$ random variables and $( \phi_n^U)$ and $( \phi_n^{U^\perp} )$ are orthonormal bases of $H(U)$ and $H^\perp(U)$, respectively. Note that $h_U$ is a GFF on $U$ and that $h_{U^\perp}$ is a random distribution which agrees with $h$ on $D \setminus U$ and can be viewed as a harmonic function on $U$ and that $h_U$ and $h_{U^\perp}$ are independent. Hence, the law of $h$ restricted to $U$, given the values of $h$ restricted to $\partial U$, is that of a GFF on $U$ plus the harmonic extension of the values of $h$ on $\partial U$. This is the so-called Markov property of the GFF. With this in mind, one can make sense of GFF with non-zero boundary conditions: let $f: \partial D \rightarrow \R$, let $F$ be the harmonic extension of $f$ to $D$ and let $h$ be a zero-boundary GFF on $D$, then the law of the of the GFF with boundary condition $f$ is given by the law of $h+F$.

The GFF also exhibits certain absolute continuity properties, the key (for us) of which we state now. (This is the content of Proposition 3.4 (ii) in \cite{MS16a}, where the reader can find a proof.)

\begin{prop}\label{abscont}
Let $D_1,D_2$ be simply connected domains, such that $D_1 \cap D_2 \neq \emptyset$ and for $j=1,2$, let $h_j$ and $F_j$ be a zero-boundary GFF and a harmonic function on $D_j$, respectively. If $U \subseteq D_1 \cap D_2$ is a bounded simply connected domain and $U' \supset \overline{U}$ is such that $\overline{D}_1 \cap U' = \overline{D}_2 \cap U'$ and $F_1-F_2$ tends to zero when approaching $\partial D_j \cap U'$ for $j=1,2$, then the laws of $(h_1+F_1)|_U$ and $(h_2+F_2)|_U$ are mutually absolutely continuous.
\end{prop}

In other words, if $h_1$ and $h_2$ are GFFs on $D_1$ and $D_2$, whose boundary conditions agree in some set $E \subset \partial D_1 \cap \partial D_2$, then the laws of $h_1$ and $h_2$ restricted to any simply connected bounded subdomain $U$ of $D_1$ and $D_2$, such that $\dist(\partial U, (\partial D_1 \cap \partial D_2) \setminus E) >0$, are mutually absolutely continuous.

The result holds for unbounded domains $U$ as well, but we shall only need the bounded case.

\subsection{Imaginary geometry}
In this section, we describe the coupling of SLE with the GFF. As stated in the introduction, this section will be slightly longer than necessary, in order to make this paper more self-contained.

Suppose for now that $h$ is a smooth, real-valued function on a Jordan domain $D$ and fix constants $\chi > 0$ and $\theta \in [0,2\pi)$. A flow line of the complex vector field $e^{i(h/\chi + \theta)}$, with initial point $z$, is a solution to the ordinary differential equation
\begin{align}\label{eq:sfl}
    \eta'(t) = e^{i(h(\eta(t))/\chi + \theta)} \quad \text{for} \quad t>0, \quad \eta(0) = z.
\end{align}
If $\eta$ is a flow line of $e^{ih/\chi}$ and $\psi: \widetilde{D} \rightarrow D$ is a conformal map, then $\tilde{\eta} = \psi^{-1} \circ \eta$ is a flow line of $e^{i\tilde{h}/\chi}$, where $\tilde{h} \coloneqq h \circ \psi - \chi \arg \psi'$ is a smooth function on $\widetilde{D}$. This follows by the chain rule and the fact that a reparametrization of a flow line is a flow line. Hence, the following definition makes sense: we say that an \textit{imaginary surface} is an equivalence class of pairs $(D,h)$ under the equivalence relation
\begin{align}\label{eq:isurface}
    (D,h) \sim (\psi^{-1}(D), h \circ \psi - \chi \arg \psi') = (\widetilde{D}, \tilde{h}).  
\end{align}
We say that $\psi$ is a \textit{conformal coordinate change} of the imaginary surface.

The idea is that if $h$ is a GFF, then we are interested in the flow lines of $h$ and we want to see that these are $\SLE_\kappa(\underline{\rho})$ curves. However, while (\ref{eq:isurface}) makes sense if $h$ is a GFF, the ODE (\ref{eq:sfl}) does not, as $h$ is then a random distribution and not a continuous function. Thus, the approach to defining the flow lines of the GFF will be a little less ``direct''. Instead, the following characterization will be used: let $h$ be a smooth function and $\eta$ a smooth simple curve in $\overline{\UH}$, with $\eta(0) = 0$ and $\eta(t) \in \UH$ for $t>0$, starting in the vertical direction (that is, $\eta'(0)$ has zero real part and positive imaginary part), so that as $t \rightarrow 0$ the winding number is $\approx \pi/2$. Furthermore, let $(f_t)_{t\geq 0}$ be the centered Loewner chain of $\eta$. Then, for any $t>0$, we have two parametrizations of $\eta |_{[0,t]}$:
\begin{align*}
    s \mapsto f_t^{-1}(-s)|_{(s_-,0)} \quad \text{and} \quad s \mapsto f_t^{-1}(s)|_{(0,s_+)},
\end{align*}
where $s_-$ and $s_+$ are the left and right images of zero under $f_t$, respectively. Since $\eta$ is smooth, there exist smooth, decreasing functions $\phi_-,\phi_+:(0,\infty) \rightarrow (0,\infty)$ such that
\begin{align*}
    \eta(s) = f_t^{-1}(-\phi_-(s)) \quad \text{and} \quad \eta(s) = f_t^{-1}(\phi_+(s)).
\end{align*}
By differentiation and (\ref{eq:sfl}), it is then easy to see that $\eta$ is a flow line of $h$ if and only if for each $z \in \eta((0,t))$
\begin{align}
    \chi \arg f_t'(w) \rightarrow -h(z)-\chi\pi/2
\end{align}
as $w$ approaches $z$ from the left side of $\eta$ and
\begin{align}
    \chi \arg f_t'(w) \rightarrow -h(z)+\chi\pi/2
\end{align}
as $w$ approaches $z$ from the right side of $\eta$. With this in mind, we will now introduce the coupling.

With the notation as in Section \ref{secslekr}, the coupling of SLE and GFF is given by the following theorem (for the proof, see \cite{MS16a}).
\begin{thm}\label{coupling}
Fix $\kappa > 0$ and a vector of weights $(\underline{\rho}_L;\underline{\rho}_R)$. Let $(K_t)$ and $(f_t)$ be the hulls and centered Loewner chain, respectively, of an $\SLE_\kappa(\underline{\rho}_L;\underline{\rho}_R)$ process in $\UH$ from $0$ to $\infty$ with force points $(\underline{x}_L;\underline{x}_R)$ and let $h$ be a zero-boundary GFF in $\UH$. Furthermore, let 
\begin{align*}
    \chi = \frac{2}{\sqrt{\kappa}} - \frac{\sqrt{\kappa}}{2}, \quad \lambda = \frac{\pi}{\sqrt{\kappa}},
\end{align*}
let $\Phi_t^0$ be the harmonic function in $\UH$ with boundary values
\begin{align*}
    -\lambda(1+\overline{\rho}_{j,L}) \quad &\text{if} \quad x \in [f_t(x_{j+1,L}),f_t(x_{j,L})), \\
    \lambda(1+\overline{\rho}_{j,R}) \quad &\text{if} \quad x \in (f_t(x_{j,R}),f_t(x_{j+1,R})],
\end{align*}
and define
\begin{align*}
    \Phi_t(z) = \Phi_t^0(f_t(z))-\chi\arg f_t'(z).
\end{align*}
If $\tau$ is any stopping time for the $\SLE_\kappa(\underline{\rho})$ process which almost surely occurs before the continuation threshold, then the conditional law of $(h+\Phi_0)|_{\UH \setminus K_\tau}$ given $K_\tau$ is equal to the law of $h \circ f_\tau + \Phi_\tau$.
\end{thm}

\begin{figure}[ht!]
\centering
\includegraphics[width=150mm]{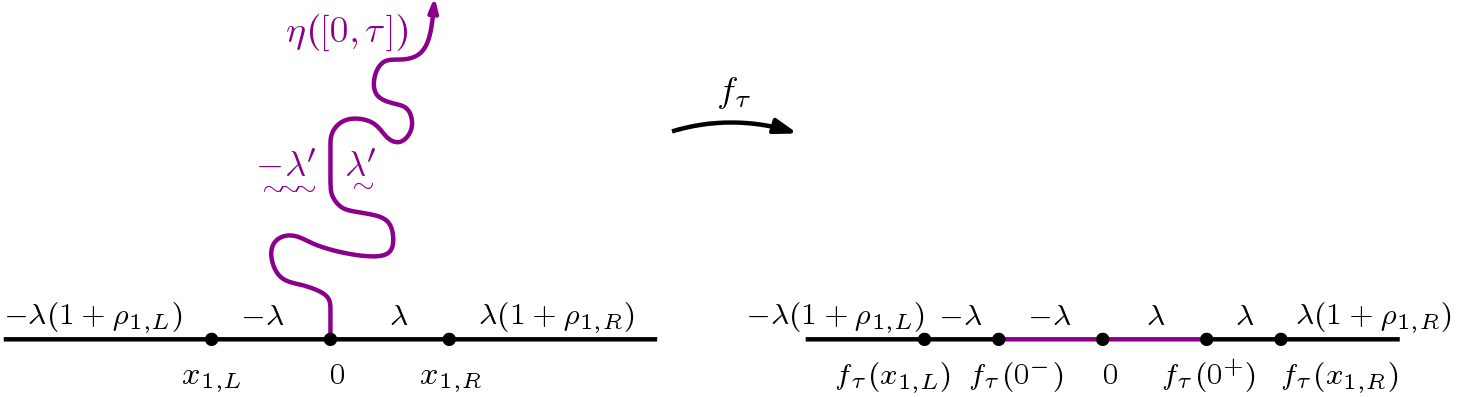}
\caption{Theorem \ref{coupling}. \label{fig:couplingfig}}
\end{figure}

In this coupling, $\eta \sim \text{SLE}_\kappa(\underline{\rho})$ is almost surely determined by $h$, that is, $\eta$ is a deterministic function of $h$ (see Theorem 1.2 of \cite{MS16a}). When $\kappa \in (0,4)$, a \textit{flow line} of the GFF $h+\Phi_0$ on $\UH$ is an $\SLE_\kappa(\underline{\rho})$ curve, $\eta$, coupled with $h+\Phi_0$ as in Theorem \ref{coupling}. This definition can be extended to other simply connected domains than $\UH$, using the conformal coordinate change described in (\ref{eq:isurface}), see Remark \ref{otherdomains}. If we add $\theta \chi$ to the boundary values, i.e., replace $h+\Phi_0$ by $h+\Phi_0 +\theta \chi$ then the resulting flow line is called a \textit{flow line of angle} $\theta$, and we denote it by $\eta_\theta$.

Note that if $\kappa \in (0,4)$ then $\chi>0$ and if $\kappa>4$ then $\chi<0$. If we let $\kappa \in (0,4)$ and write $\kappa' = 16/\kappa$, then $\kappa'>4$ and $\chi(\kappa) = -\chi(\kappa')$. From Theorem \ref{coupling}, it is clear that the conditional law of $h+\Phi_0$ given an $\SLE_\kappa$ or $\SLE_{\kappa'}$ curve is transformed in the same way under a conformal map, up to a sign change, which motivates the following definition. A \textit{counterflow line} of the GFF $h+\Phi_0$ is an $\SLE_{\kappa'}(\underline{\rho})$ curve coupled with $-(h+\Phi_0)$ as in Theorem \ref{coupling}. Note that the sign of the GFF is changed so that it matches the sign of $\chi(\kappa')$ and that in the notation of the theorem, the $\lambda$ is replaced by $\lambda' = \frac{\pi}{\sqrt{\kappa'}} = \lambda - \frac{\pi}{2}\chi$.

In the figures, we often write $\underset{^\squig}{a}$, where $a$ is some real number. This is to be interpreted as $a$ plus $\chi$ times the winding of the curve, see Figure \ref{fig:winding}. This makes perfect sense for piecewise smooth curves, but for fractal curves the winding is not defined pointwise. However, the harmonic extension of the winding of the curve makes sense, as we can map conformally to $\UH$ with piecewise constant boundary conditions. The term $-\chi \arg f_\tau'(z)$ in Theorem \ref{coupling} is interpreted as $\chi$ times the harmonic extension of the winding of the curve $\eta$. 

\begin{figure}[ht!]
\centering
\includegraphics[width=155mm]{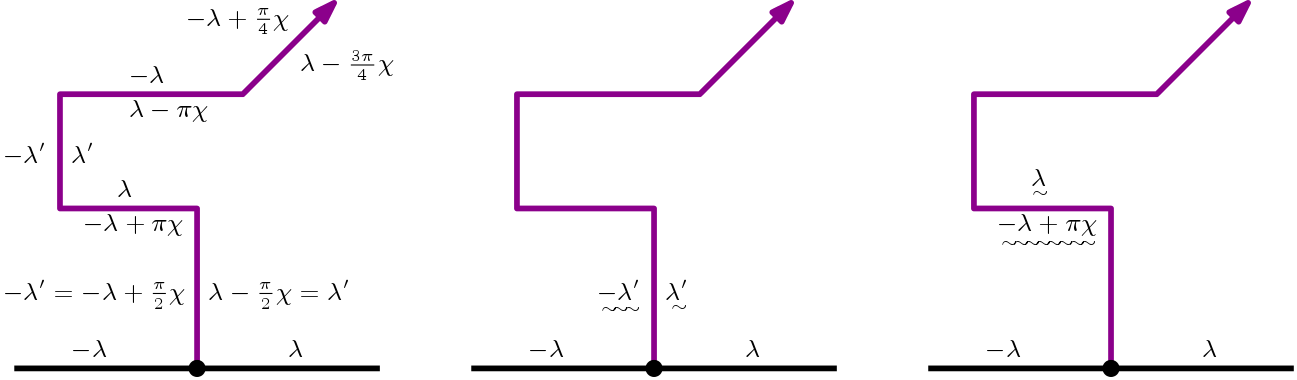}
\caption{Each time the curve makes a turn to the right, the boundary values decrease by $\chi$ times the angle and each time it makes a turn to the left the angle increases by $\chi$ times the angle, for example, a quarter turn to the right (left) decreases (increases) the boundary values by $\frac{\pi}{2} \chi$. We illustrate $a + \chi\cdot \text{winding}$ as $\underset{^\squig}{a}$, and hence, the three pictures of this figure give the same information. \label{fig:winding}}
\end{figure}

\begin{rmk}\label{otherdomains}
Let $D$ be a simply connected domain, with $x,y \in \partial D$ distinct and let $\psi: D \rightarrow \UH$ be a conformal transformation with $\psi(x) = 0$ and $\psi(y) = \infty$. Let $\underline{x}_L$ ($\underline{x}_R$) consist of $l$ ($r$) marked prime ends in the clockwise (counterclockwise) segment of $\partial D$, which are in clockwise (counterclockwise) order. The orientation of $\partial D$ is as defined by $\psi$. Write $x_{0,L} = x_{0,R} = x$ and $x_{l+1,L} = x_{r+1,R} = y$ and let $\underline{\rho}_L$ and $\underline{\rho}_R$ be vectors of weights corresponding to the points in $\underline{x}_L$ and $\underline{x}_R$ respectively. Let $h$ be a GFF on $D$ with boundary values given by
\begin{align*}
    - &\lambda(1+\overline{\rho}_{j,L})-\chi\arg\psi' \quad \text{for} \quad z' \in (x_{j,L},x_{j+1,L}) \\
    &\lambda(1+\overline{\rho}_{j,R})-\chi\arg\psi' \quad \text{for} \quad z' \in (x_{j,R},x_{j+1,R})
\end{align*}
where $(x_{j,L},x_{j+1,L})$ denotes the clockwise segment of $\partial D$ from $x_{j,L}$ to $x_{j+1,L}$ and $(x_{j,R},x_{j+1,R})$ the counterclockwise segment of $\partial D$ from $x_{j,R}$ to $x_{j+1,R}$. Let $\kappa \in (0,4)$. We say that an $\SLE_\kappa(\underline{\rho})$ curve, $\eta$, from $x$ to $y$ in $D$, coupled with $h$, is a \textit{flow line} of $h$ if $\psi(\eta)$ is coupled as a flow line of the GFF $h \circ \psi^{-1} - \arg(\psi^{-1})'$ on $\UH$.

The same statement holds for counterflow lines with $\kappa' \in (4,\infty)$ if we replace $\lambda$ with $\lambda'$ in the boundary values (but keep $\chi = \chi(\kappa)$).
\end{rmk}

We write the following statements for flow lines in $\UH$, but they hold true for other simply connected domains as well.

Let $h$ be a GFF in $\UH$ with piecewise constant boundary values. It turns out that (Theorem 1.4, \cite{MS16a}) if $\eta'$ is a counterflow line of $h$ in $\UH$, from $\infty$ to $0$, then the range of $\eta'$ is almost surely equal to the points that can be reached by the flow lines in $\UH$, from $0$ to $\infty$, with angles in the interval $\left[ -\frac{\pi}{2},\frac{\pi}{2} \right]$. Also, it almost surely holds that the left boundary of $\eta'$ is equal to the trace of the flow line of angle $-\frac{\pi}{2}$ and the right boundary is equal to the trace of the flow line of angle $\frac{\pi}{2}$ (seen from the viewpoint of travelling along $\eta'$, from the flow lines' point of view, it is the other way). Here, we talk about counterflow lines corresponding to the parameter $\kappa'$ and flow lines corresponding to $\kappa$, so that they can be coupled with the same GFF.

\begin{figure}[ht!]
\centering
\includegraphics[width=80mm]{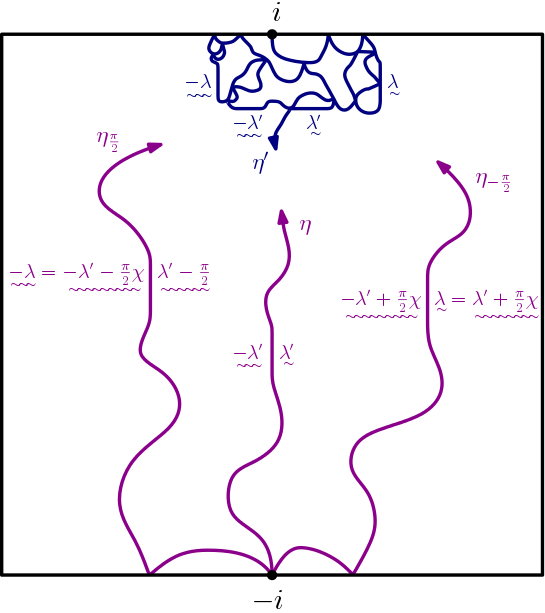}
\caption{The flow lines $\eta_{-\frac{\pi}{2}}$, $\eta$, $\eta_\frac{\pi}{2}$ from $-i$ to $i$, and the counterflow line $\eta'$ from $i$ to $-i$ in the square $[-1,1]^2$. $\eta_{-\frac{\pi}{2}}$ and $\eta_{\frac{\pi}{2}}$ will hit and merge with the respective sides of $\eta'$, as they will be the outer boundary of $\eta'$. Note that the outer boundary conditions agree. \label{fig:flowcounterflow}}
\end{figure}

Again, let $h$ be a GFF in $\UH$ with piecewise constant boundary values. For each $x \in \R$ and $\theta \in \R$, we denote by $\eta_\theta^x$ the flow line of $h$ from $x$ to $\infty$ with angle $\theta$. Fix $x_1,x_2 \in \R$ such that $x_1 \geq x_2$, then the following holds (see Figure \ref{fig:interactions} for illustrations).
\begin{enumerate}[(i)]
\item If $\theta_1 < \theta_2$, then $\eta_{\theta_1}^{x_1}$ almost surely stays to the right of $\eta_{\theta_2}^{x_2}$. If $\theta_2 - \theta_1 < \frac{\pi \kappa}{4-\kappa}$, then the paths might hit and bounce off of each other, otherwise they almost surely never collide away from the starting point.
\item If $\theta_1 = \theta_2$, then $\eta_{\theta_1}^{x_1}$ and $\eta_{\theta_2}^{x_2}$ can intersect and if they do, they merge and never separate.
\item If $\theta_2 < \theta_1 < \theta_2 + \pi$, then $\eta_{\theta_1}^{x_1}$ and $\eta_{\theta_2}^{x_2}$ can intersect, and if they do, they cross and never cross back. If $\theta_1 - \theta_2 < \frac{\pi\kappa}{4-\kappa}$, then they can hit and bounce off of each other, otherwise they never intersect after crossing.
\end{enumerate}
The above flow line interactions are the content of Theorem 1.5 of \cite{MS16a}. We shall make use of property (ii), as it is instrumental in our two-point estimate.

\begin{figure}[ht!]
\begin{subfigure}{.5\linewidth}
\centering
\includegraphics[width=55mm]{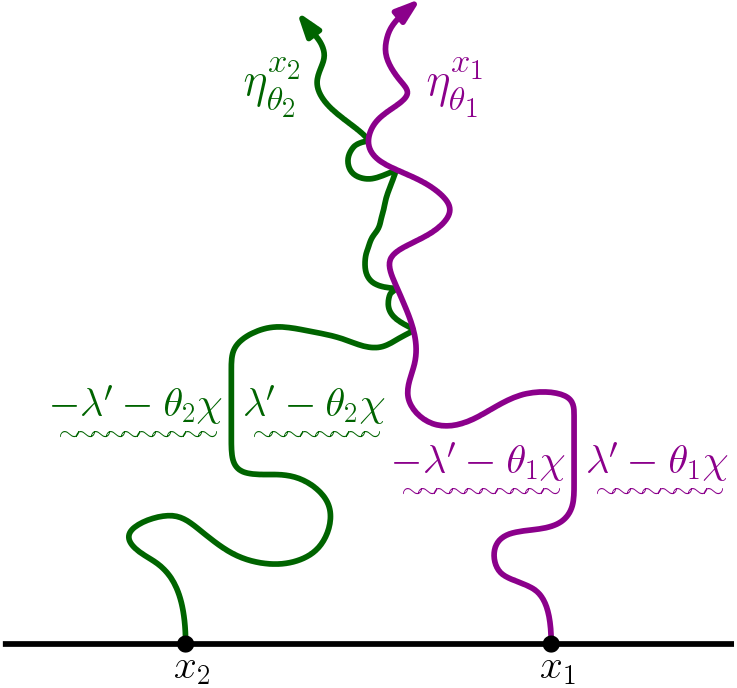}
\caption{If $\theta_1 < \theta_2$, then $\eta_{\theta_1}^{x_1}$ stays to the right of $\eta_{\theta_2}^{x_2}$.}
\label{fig:int1}
\end{subfigure}%
\begin{subfigure}{.5\linewidth}
\centering
\includegraphics[width=55mm]{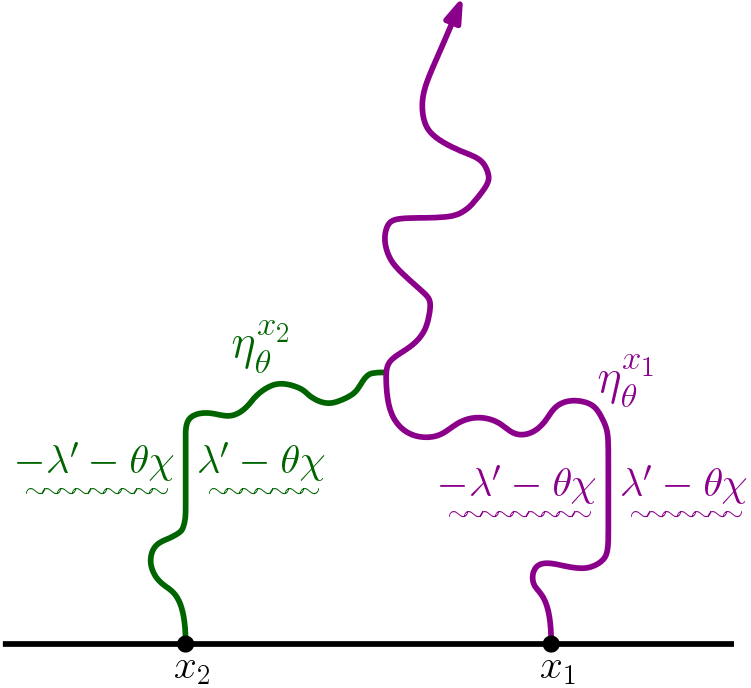}
\caption{If $\theta = \theta_1 = \theta_2$, then $\eta_{\theta_1}^{x_1}$ and $\eta_{\theta_2}^{x_2}$ merge when intersecting.}
\label{fig:int2}
\end{subfigure}\\[1ex]
\begin{subfigure}{\linewidth}
\centering
\includegraphics[width=55mm]{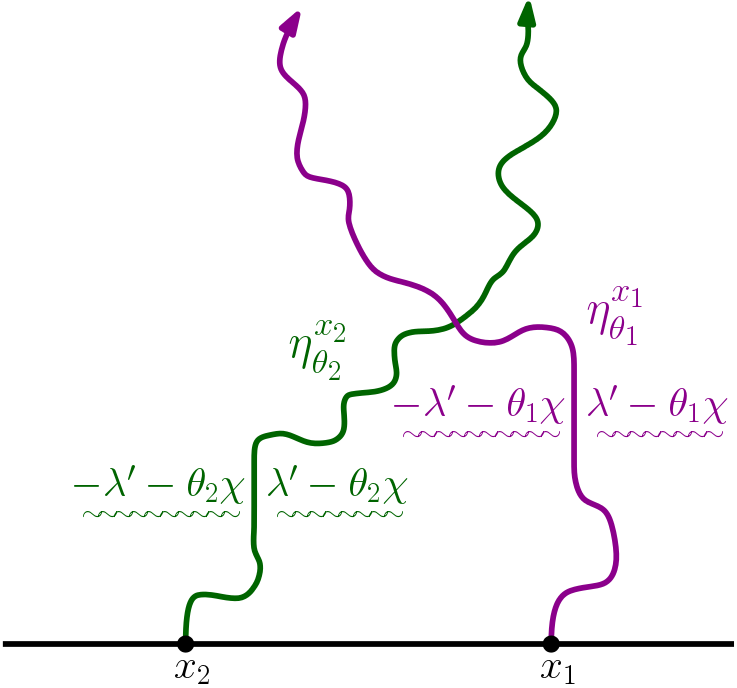}
\caption{If $\theta_2 < \theta_1 < \theta_2 + \pi$, then $\eta_{\theta_1}^{x_1}$ and $\eta_{\theta_2}^{x_2}$ cross when intersecting for the first time, and never cross back.}
\label{fig:int3}
\end{subfigure}
\caption{Flow line interactions for different pairs $(\theta_1,\theta_2)$.}
\label{fig:interactions}
\end{figure}

\subsubsection{Level lines}
The coupling is valid for $\kappa = 4$ as well. We then interpret the resulting $\SLE_\kappa(\underline{\rho})$ curve as the level line of the GFF. Note that $\chi(4)=0$, that is, there is no extra winding term, and hence the boundary values of the level line are constant along the curve, $-\lambda$ on the left and $\lambda$ on the right. As in the case of flow and counterflow lines, level lines can be defined in other domains and with different starting and ending points via conformal maps. For level lines, the terminology is a bit different: we say that $\eta$ is a level line of height $u\in\R$ if it is a level line of the GFF $h+u$. The same interactions as for flow lines hold for level lines. We let $\eta_u^x$ denote the level line of height $u$ starting from $x$. Let $x_1 \geq x_2$, then 

\begin{enumerate}[(i)]
\item If $u_1 < u_2$, then $\eta_{u_1}^{x_1}$ almost surely stays to the right of $\eta_{u_2}^{x_2}$.
\item If $u_1 = u_2$, then $\eta_{u_1}^{x_1}$ and $\eta_{u_2}^{x_2}$ can intersect and if they do, they merge and never separate.
\end{enumerate}

For more on the level lines of a GFF with piecewise constant boundary data, see \cite{WW17}.

\subsubsection{Deterministic curves and Radon-Nikodym derivatives}
We now recall some consequences of Proposition \ref{abscont}: two lemmas about flow and counterflow line behaviour and two lemmas on absolute continuity, all from \cite{MW17}, which we will need in Section \ref{sec2PEIG}. Note that while they are proven for $\kappa \neq 4$, the case $\kappa = 4$ follows by the same argument as $\kappa < 4$, when the $\SLE_4(\underline{\rho})$ curves are coupled as level lines.

\begin{figure}[ht!]
\centering
\includegraphics[width=100mm]{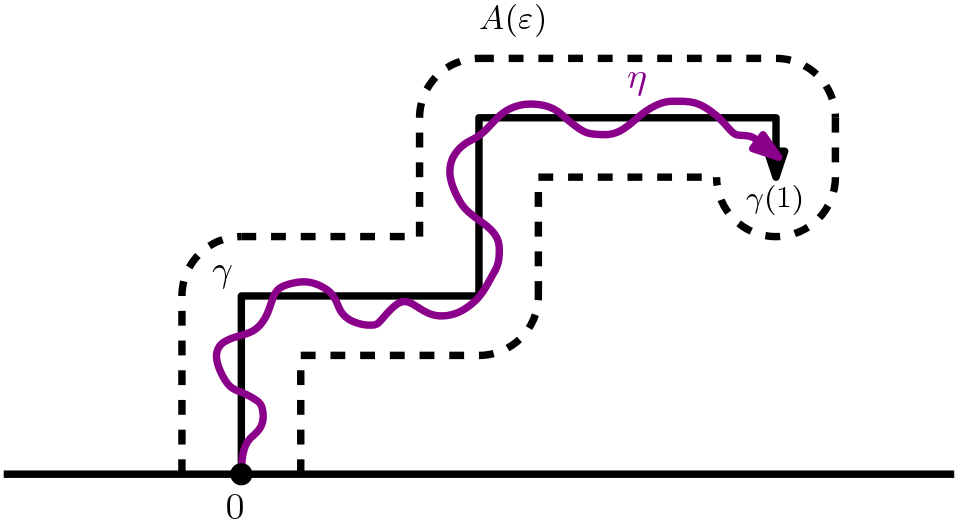}
\caption{Lemma \ref{IGlem1} states that if $\eta$ is an $\SLE_\kappa(\underline{\rho}_L,\underline{\rho}_R)$ curve in $\UH$ from $0$ to $\infty$, with $x_{1,L} = 0^-$, $x_{1,R}=0^+$ and $\rho_{1,L},\rho_{1,R} > -2$, then for any fixed deterministic curve $\gamma:[0,1] \rightarrow \overline{\UH}$, such that $\gamma(0) = 0$ and $\gamma((0,1]) \subset \UH$, and each $\epsilon>0$, the probability that $\eta$ comes within distance $\epsilon$ of $\gamma(1)$ before leaving the $\epsilon$-neighborhood of $\gamma$ is positive. \label{fig:IGlem1}}
\end{figure}

\begin{lem}[Lemma 2.3 of \cite{MW17}]\label{IGlem1}
Fix $\kappa > 0$ and let $\eta$ be an $\SLE_\kappa(\underline{\rho}_L;\underline{\rho}_R)$ curve in $\UH$ from $0$ to $\infty$, with force points $(\underline{x}_L;\underline{x}_R)$ such that $x_{1,L} = 0^-$, $x_{1,R} = 0^+$ and $\rho_{1,L},\rho_{1,R} > -2$. Let $\gamma:[0,1] \rightarrow \overline{\UH}$ be a deterministic curve such that $\gamma(0) = 0$ and $\gamma((0,1]) \subset \UH$. Fix $\epsilon > 0$, write $A(\epsilon) = \{ z: \dist(z,\gamma([0,1])) < \epsilon \}$ and define the stopping times
\begin{align*}
    \sigma_1 = \inf \{t \geq 0: |\eta(t) - \gamma(1)| \leq \epsilon \}, \quad \sigma_2 = \inf \{ t \geq 0: \eta(t) \notin A(\epsilon) \}.
\end{align*}
Then $\prob(\sigma_1 < \sigma_2) > 0$.
\end{lem}

\begin{figure}[ht!]
\centering
\includegraphics[width=100mm]{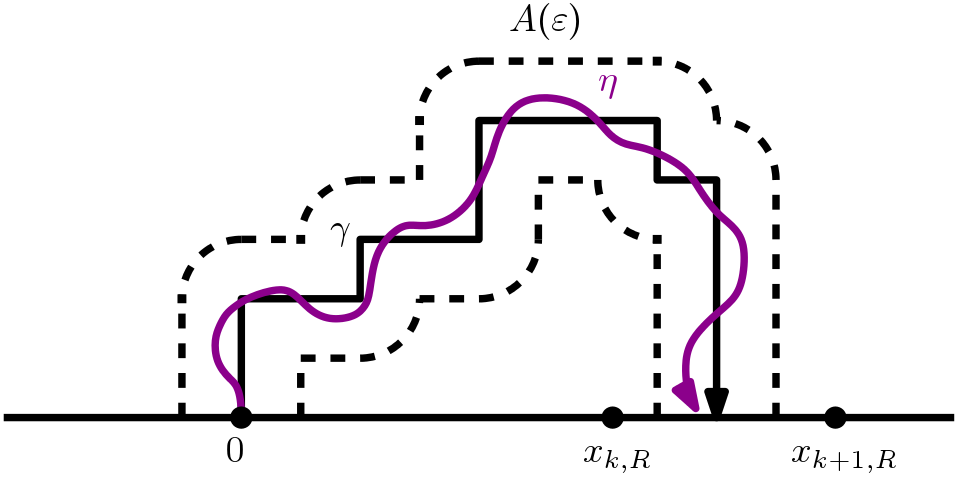}
\caption{By Lemma \ref{IGlem3} we have that if $\eta$ is an $\SLE_\kappa(\underline{\rho}_L,\underline{\rho}_R)$ curve in $\UH$ from $0$ to $\infty$, with $x_{1,L} = 0^-$, $x_{1,R}=0^+$ and $\rho_{1,L},\rho_{1,R} > -2$, then for any fixed deterministic curve $\gamma:[0,1] \rightarrow \overline{\UH}$, such that $\gamma(0) = 0$, $\gamma(1) \in [x_{k,R},x_{k+1,R}]$ and $\gamma((0,1)) \subset \UH$, where $k$ is such that $\overline{\rho}_{k,R} \in (\frac{\kappa}{2}-4,\frac{\kappa}{2}-2)$, and each $\epsilon>0$, the probability that $\eta$ hits $[x_{k,R},x_{k+1,R}]$ before leaving the $\epsilon$-neighborhood of $\gamma$ is positive. \label{fig:IGlem3}}
\end{figure}

\begin{lem}[Lemma 2.5 of \cite{MW17}]\label{IGlem3}
Fix $\kappa > 0$ and let $\eta$ be an $\SLE_\kappa(\underline{\rho}_L;\underline{\rho}_R)$ curve in $\UH$ from $0$ to $\infty$, with force points $(\underline{x}_L;\underline{x}_R)$ such that $x_{1,L} = 0^-$, $x_{1,R} = 0^+$ and $\rho_{1,L},\rho_{1,R} > -2$. Fix $k \in \N$ such that $\overline{\rho}_{k,R} \in (\frac{\kappa}{2}-4,\frac{\kappa}{2}-2)$ and an $\epsilon > 0$ such that $|x_{2,q}| \geq \epsilon$ for $q \in \{L,R\}$, $x_{k+1,R}-x_{k,R} \geq \epsilon$ and $x_{k,R} \leq \epsilon^{-1}$. Let $\gamma:[0,1] \rightarrow \overline{\UH}$, with $\gamma(0) = 0$, $\gamma((0,1)) \subset \UH$ and $\gamma(1) \in [x_{k,R},x_{k+1,R}]$ and $A(\epsilon) = \{ z: \dist(z,\gamma([0,1])) < \epsilon \}$ and define the stopping times
\begin{align*}
    \sigma_1 = \inf \{ t\geq 0: \eta(t) \in (x_{k,R},x_{k+1,R}) \}, \quad \sigma_2 = \inf \{ t\geq 0: \eta(t) \notin A(\epsilon) \}.
\end{align*}
Then there exists a $p_1 = p_1(\kappa,\max_{j,q} |\rho_{j,q}|,\overline{\rho}_{k,R},\epsilon) > 0$, such that $\prob(\sigma_1 < \sigma_2) \geq p_1$.
\end{lem}

Next, we shall describe the Radon-Nikodym derivatives between $\SLE_\kappa(\underline{\rho})$ processes in different domains, see also \cite{Dub09} and \cite{MW17}. The results that we need are Lemma \ref{IGlem4} and Lemma \ref{IGlem5}, which we will use in Section \ref{sec2PEIG}. Let 
\begin{align*}
    c = (D,z_0,\underline{x}_L,\underline{x}_R,z_\infty)
\end{align*}
be a configuration, that is, a Jordan domain $D$ with boundary points $z_0$, $\underline{x}_L$, $\underline{x}_R$ and $z_\infty$, and let $U$ be an open neighborhood of $z_0$. Denote the law of an $\SLE_\kappa(\underline{\rho}_L;\underline{\rho}_R)$ process with configuration $c$, stopped the first time $\tau$ it exits $U$, by $\mu_c^U$. Let $H_D$ be the Poisson excursion kernel of $D$, that is, if $\varphi: D \rightarrow \UH$ is conformal, then
\begin{align*}
    H_D(x,y) = \frac{\varphi'(y)}{(\varphi(y)-\varphi(x))^2}.
\end{align*}
Furthermore, let $\rho_\infty = \kappa-6-\sum_{j,q} \rho_{j,q}$ and 
\begin{align*}
    Z(c) &= H_D(z_0,z_\infty)^{-\frac{\rho_\infty}{2\kappa}} \times \prod_{j,q} H_D(z_0,x_{j,q})^{-\frac{\rho_{j,q}}{2\kappa}} \\
    &\times \prod_{(j_1,q_1) \neq (j_2,q_2)} H_D(x_{j_1,q_1},x_{j_2,q_2})^{-\frac{\rho_{j_1,q_1} \rho_{j_2,q_2}}{4\kappa}} \times \prod_{j,q} H_D(x_{j,q},z_\infty)^{-\frac{\rho_{j,q} \rho_\infty}{4\kappa}}.
\end{align*}
Moreover, let
\begin{align*}
    c_\tau = (D \setminus K_\tau, \eta(\tau),\underline{x}_L^\tau,\underline{x}_R^\tau,z_\infty),
\end{align*}
where $x_{j,q}^\tau = x_{j,q}$ if $x_{j,q}$ is not swallowed by $\eta$ at time $\tau$ and $x_{j,q}^\tau$ the leftmost (resp. rightmost) point of $K_\tau \cap \partial D$ on the clockwise (resp. counterclockwise) arc of $\partial D$ if $q=L$ (resp. $q=R$). Moreover, let $\mu^\text{loop}$ be the Brownian loop measure, a $\sigma$-finite measure on unrooted loops (see \cite{LW04}), and write
\begin{align*}
    m(D;K,K') = \mu^\text{loop}(l:l \subseteq D, l \cap K \neq \emptyset, l \cap K' \neq \emptyset),
\end{align*}
where
\begin{align*}
    \xi = \frac{(6-\kappa)(8-3\kappa)}{2\kappa}.
\end{align*}

\begin{lem}[Lemma 2.7 of \cite{MW17}]\label{IGlem4}
Let $c = (D,z_0,\underline{x}_L,\underline{x}_R,z_\infty)$ and $\tilde{c} = (\widetilde{D},z_0,\underline{\tilde{x}}_L,\underline{\tilde{x}}_R,\tilde{z}_\infty)$ be configurations and $U$ an open neighborhood of $z_0$ such that $c$ and $\tilde{c}$ and the weights of the marked points agree in $U$, and the distance from $U$ to the marked points of $c$ and $\tilde{c}$ which differ, is positive. Then, the probability measures $\mu_c^U$ and $\mu_{\tilde{c}}^U$ are mutually absolutely continuous and the Radon-Nikodym derivative between them are given by
\begin{align*}
    \frac{d\mu_{\tilde{c}}^U}{d\mu_{c}^U} (\eta) = \frac{Z(\tilde{c}_\tau)/Z(\tilde{c})}{Z(c_\tau)/Z(c)} \exp \left( -\xi m(D;K_\tau,D \setminus \widetilde{D}) + \xi m(\widetilde{D};K_\tau,\widetilde{D} \setminus D) \right).
\end{align*}
\end{lem}

\begin{lem}[Lemma 2.8 of \cite{MW17}]\label{IGlem5}
Assume that we have the same setup as in Lemma \ref{IGlem4}, with $D = \UH$, $\widetilde{D} \subseteq \UH$, $U \subset \UH$ bounded and $z_0 = 0$. Fix $\varsigma > 0$ and suppose that $\dist(U,\UH \setminus \widetilde{D}) > \varsigma$ and that the force points which are outside $U$ are at least at distance $\varsigma$ from $U$. Then there exists a constant $C \geq 1$, depending on $U$, $\varsigma$, $\kappa$ and the weights of the force points, such that 
\begin{align*}
    \frac{1}{C} \leq \frac{d\mu_{\tilde{c}}^U}{d\mu_{c}^U} \leq C.
\end{align*}
\end{lem}

\section{One-point estimates}\label{secOPE}
In this section we will find first moment estimates, which will be of importance, as they will give us the means to get good two-point estimates as well as give us the upper bound of the dimension of $V_\beta^*$. Recall that $\tilde{g}_s = g_{\tilde{t}(s)}$ is the Loewner chain under the radial time change, see Section \ref{weighted}. 
\begin{prop}
\label{1PEts}
Let $\zeta > -\mu_c^2/2a$ and $x_R \geq 0$. For all $x > x_R$, we have
\begin{align*}
    \E\left[ \tilde{g}_s'(x)^\zeta 1\{\tilde{t}(s)<\infty\} \right] = K \left( \frac{x-x_R}{x} \right)^\mu e^{-a\mu(1+\rho/2)s}(1+O(e^{-(1-a+\mu) s})),
\end{align*}
where $K = \frac{\Gamma(2-2a+2\mu) \Gamma(2-4a-a\rho+\mu)}{\Gamma(2-2a+\mu) \Gamma(2-4a-a\rho+2\mu)}$.
\end{prop}
\begin{proof}
By \eqref{eq:basiccalc}, $f(x) = x^{-\mu}$ is in $L^1(\nu_{\tilde{Q}})$ (see Corollary \ref{Qprop}) and thus Corollary \ref{Qprop} gives that 
\begin{align}\label{eq:QasympX}
    \E^*[\tilde{Q}_s^{-\mu}] = \E^*[\tilde{X}_s^{-\mu}](1+O(e^{-(1-a+\mu) s})).
\end{align}
Thus, since $\prob^*(\tilde{t}(s) < \infty) = 1$ for every $s$, we have
\begin{align*}
&\E[ \tilde{g}_s'(x)^\zeta 1\{ \tilde{t}(s)<\infty\}] = \E \left[ \tilde{M}_s^\mu \tilde{Q}_s^{-\mu} (x-x_R)^{\mu(1+\rho/2)} e^{-a\mu(1+\rho/2)s} 1\{ \tilde{t}(s)<\infty\} \right] \\
&= (x-x_R)^{\mu(1+\rho/2)} e^{-a\mu(1+\rho/2)} \E\left[ \tilde{M}_s^\mu \tilde{Q}_s^{-\mu} 1\{ \tilde{t}(s) < \infty \} \right] \\
&= (x-x_R)^{\mu(1+\rho/2)} e^{-a\mu(1+\rho/2)} x^{-\mu} (x-x_R)^{-\mu \rho/2} \E^*[\tilde{Q}_s^{-\mu}] \\
&= \left( \frac{x-x_R}{x} \right)^\mu e^{-a\mu (1+\rho/2)s} \E^*[\tilde{Q}_s^{-\mu}] \\
&= \left( \frac{x-x_R}{x} \right)^\mu e^{-a\mu (1+\rho/2)s} \E^*[\tilde{X}_s^{-\mu}] (1+O(e^{-(1-a+\mu) s})) \\
&= K \left( \frac{x-x_R}{x} \right)^\mu e^{-a\mu (1+\rho/2)s} (1+O(e^{-(1-a+\mu) s})),
\end{align*}
using \eqref{eq:Mtilde} for the first equality, changing to the measure $\prob^*$ for the third inequality, using \eqref{eq:QasympX} in the fifth and \eqref{eq:basiccalc} in the last equality.
\end{proof}
Using Proposition \ref{1PEts} and Lemma \ref{time}, recalling that we can choose a $C^* = C^*(x,x_R)$ such that $\eqref{C*}$ holds, we get the following corollary.
\begin{cor}
\label{1PE}
Suppose $\zeta \geq 0$ and $x_R \geq 0$. For every $x > x_R$, there is a constant $C=C(x,x_R)$ such that 
\begin{align*}
    \frac{1}{C} e^{-\mu(1+\rho/2)s} \leq \E \Big[ g_{\tau_s}'(x)^\zeta 1\{ \tau_s < \infty\} \Big] \leq C e^{-\mu(1+\rho/2)s}.
\end{align*}
\end{cor}
\begin{proof}
By Lemma \ref{time} and that the map $t \mapsto g_t'$ is decreasing, we have
\begin{align*}
    \E\left[ \tilde{g}_{\frac{s}{a}+C^*}'(x)^\zeta 1\{\tilde{t}(s/a +C^*)<\infty\} \right] \leq \E \Big[ g_{\tau_s}'(x)^\zeta 1\{ \tau_s < \infty\} \Big] \leq \E\left[ \tilde{g}_{\frac{s}{a} -C^*}'(x)^\zeta 1\{\tilde{t}(s/a - C^*)<\infty\} \right].
\end{align*}
By the previous proposition, we have 
\begin{align*}
    \E\left[ \tilde{g}_{\frac{s}{a}+C^*}'(x)^\zeta 1\{\tilde{t}(s/a+C^*)<\infty\} \right] = K \left( \frac{x-x_R}{x} \right)^\mu e^{-\mu(1+\rho/2)(s+aC^*)} (1+O(e^{-\frac{1}{a}(1-a+\mu) s})),
\end{align*}
and 
\begin{align*}
    \E\left[ \tilde{g}_{\frac{s}{a} -C^*}'(x)^\zeta 1\{\tilde{t}(s/a -C^*)<\infty\} \right] = K \left( \frac{x-x_R}{x} \right)^\mu e^{-\mu(1+\rho/2)(s-aC^*)} (1+O(e^{-\frac{1}{a}(1-a+\mu) s})),
\end{align*}
where the constants in $O$ depend on $x$ and $x_R$ (since $C^*$ does). Thus, the proof is done.
\end{proof}
At this point, we already have what is needed for the upper bound of the dimension of $V_\beta^*$.

\subsection{Mass concentration}
In this subsection, we will see that the mass of the weighted measure $\prob^*$ is concentrated on an event where the behaviour of $\tilde{g}_s'(x)$, for fixed $x$, is nice. On this event, we will show that $\tilde{g}_s'(x)$ satisfies a number of inequalities which will be helpful in proving the two-point estimate of the next section. The ideas here are similar to those of Section 7 of \cite{Law09}.

We define the process $\tilde{L}_s$, by recalling (\ref{eq:g'tc}), as
\begin{align*}
    \tilde{L}_s = - \frac{1}{a} \log \tilde{g}_s'(x) = \int_0^s \tilde{Q}_u^{-1} (1-\tilde{Q}_u)du.
\end{align*}
As stated in Section \ref{weighted} (and shown in the appendix), $\tilde{Q}_s$ has an invariant distribution under $\prob^*$, with density $p_{\tilde{Q}}$ (recall \eqref{eq:indensQ}). Therefore, by the ergodicity of $\tilde{Q}_s$ (Corollary \ref{Qprop}) and a computation,
\begin{align*}
    \lim_{s \rightarrow \infty} \frac{\tilde{L}_s}{s} = \int_0^1 y^{-1}(1-y) p_{\tilde{Q}}(y) dy = \beta(1+\rho/2),
\end{align*}
holds $\prob^*$-almost surely, that is, the time average converges $\prob^*$-almost surely to the space average. We shall prove that, roughly speaking, as $s \rightarrow \infty$, $\tilde{L}_s \approx \beta(1+\rho/2)s$, with an error of order $\sqrt{s}$. To prove this, we need to prove the next lemma first.

\begin{lem}
\label{concentration}
Let $\zeta > -\mu_c^2/2a$. There is a positive constant $c<\infty$ such that for $p>0$ sufficiently small, and $t\geq 1$,
\begin{align*}
    \E^* \left[ \exp \left\{ p \frac{|\tilde{L}_t - \beta(1+\rho/2)t|}{\sqrt{t}} \right\} \right] \leq c.
\end{align*}
\end{lem}
The proof idea is as follows. 
Observe that if we view $\mu = \mu_c + \sqrt{\mu_c^2 + 2a\zeta}$ as a function of $\zeta$, then $\mu'(\zeta) = a/\sqrt{\mu_c^2 + 2a\zeta} = \beta$. We define the process
\begin{align*}
    \tilde{N}_s = e^{-a\delta \tilde{L}_s} e^{a(1+\rho/2)(\mu(\zeta+\delta) - \mu(\zeta))s} \tilde{Q}_s^{\mu(\zeta+\delta) - \mu(\zeta)},
\end{align*}
which by Itô's formula is seen to be a local martingale under $\prob^*$. Since it is bounded from below, it is a supermartingale. Then we use that $\mu(\zeta+\delta) - \mu(\zeta) = \delta\beta + O(\delta^2)$ and that we have good control of $\tilde{Q}_s$.
\begin{proof}
We have that
\begin{align*}
    \log \tilde{N}_t -(\mu(\zeta+\delta)-\mu(\zeta)) \log \tilde{Q}_t &= -a\delta \tilde{L}_t + at(1+\rho/2)(\mu(\zeta+\delta) - \mu(\zeta)) \\
    &= -a\delta(\tilde{L}_t - \beta(1+\rho/2)t) + O(\delta^2 t),
\end{align*}
since $\mu(\zeta+\delta)-\mu(\zeta)=\delta\beta + O(\delta^2)$. This implies that
\begin{align*}
    \log \tilde{N}_t = -a\delta(\tilde{L}_t - \beta(1+\rho/2)t) + O(\delta^2 t) + (\delta\beta + O(\delta^2)) \log\tilde{Q}_t.
\end{align*}
Let $\delta = \pm \frac{\epsilon}{\sqrt{t}}$, where $\epsilon$ is small enough for $\tilde{N}_t$ to be well-defined. Then,
\begin{align*}
    \log \tilde{N}_t = \mp a\frac{\epsilon}{\sqrt{t}}(\tilde{L}_t - \beta(1+\rho/2)t) + O(\epsilon^2) + (\beta\frac{\pm \epsilon}{\sqrt{t}} + O(\epsilon^2/t))\log\tilde{Q}_t,
\end{align*}
and exponentiating, we get
\begin{align*}
    \E^*\left[ \exp\left\{ \mp a\epsilon \frac{\tilde{L}_t-\beta(1+\rho/2)t}{\sqrt{t}} \right\} \tilde{Q}_t^{\beta \frac{\pm \epsilon}{\sqrt{t}}+O(\epsilon^2/t)} \right] \leq c,
\end{align*}
since $\tilde{N}_t$ is a supermartingale and hence $\E[\tilde{N}_t] \leq \E[\tilde{N}_0] = \tilde{Q}_0^{\mu(\zeta+\delta)-\mu(\zeta)} = \left(\frac{x-x_R}{x}\right)^{\mu(\zeta+\delta)-\mu(\zeta)}$. Consider the case $\delta<0$, i.e.,
\begin{align*}
    \E^*\left[ \exp\left\{a\epsilon \frac{\tilde{L}_t-\beta(1+\rho/2)t}{\sqrt{t}} \right\} \tilde{Q}_t^{-\beta \frac{\epsilon}{\sqrt{t}}+O(\epsilon^2/t)} \right] \leq c.
\end{align*}
Since $\tilde{Q}_t \in [0,1], \beta > 0$, we have $\tilde{Q}_t^{-\beta \frac{\epsilon}{\sqrt{t}} + O(\epsilon^2/t)} \geq 1$ for sufficiently small $\epsilon$, and thus
\begin{align*}
    \E^*\left[ \exp\left\{a\epsilon \frac{\tilde{L}_t-\beta(1+\rho/2)t}{\sqrt{t}} \right\} \right] \leq c.
\end{align*}
Consider the case $\delta>0$. We will split the expectation into the cases $\tilde{Q}_t \leq y$ and $\tilde{Q}_t > y$ for some $y \in (0,1]$. First,
\begin{align*}
    c &\geq \E^*\left[ \exp\left\{-a\epsilon \frac{\tilde{L}_t-\beta(1+\rho/2)t}{\sqrt{t}} \right\} \tilde{Q}_t^{\beta \frac{\epsilon}{\sqrt{t}}+O(\epsilon^2/t)} 1\left\{ \tilde{Q}_t > y \right\} \right] \\
    &\geq \E^*\left[ \exp\left\{-a\epsilon \frac{\tilde{L}_t-\beta(1+\rho/2)t}{\sqrt{t}} \right\} y^{\beta \frac{\epsilon}{\sqrt{t}}+O(\epsilon^2/t)} 1\left\{ \tilde{Q}_t > y \right\} \right],
\end{align*}
which implies
\begin{align*}
    \E^*\left[ \exp\left\{-a\epsilon \frac{\tilde{L}_t-\beta(1+\rho/2)t}{\sqrt{t}} \right\} 1\left\{ \tilde{Q}_t > y \right\} \right] \leq cy^{-\beta \frac{\epsilon}{\sqrt{t}} + O(\epsilon^2/t)} \leq cy^{-2\beta \frac{\epsilon}{\sqrt{t}}}
\end{align*}
for sufficiently small $\epsilon$. For the other part, note that since $\tilde{L}_t \geq 0$, 
\begin{align*}
    &\E^* \left[ \exp\left\{ -a\epsilon \frac{\tilde{L}_t - \beta(1+\rho/2)t}{\sqrt{t}} \right\} 1\left\{ \tilde{Q}_t \leq y \right\} \right] \leq \E^* \left[ e^{a\epsilon\beta(1+\rho/2)\sqrt{t}} 1\left\{ \tilde{Q}_t \leq y \right\} \right] \\
    &= e^{a\epsilon\beta(1+\rho/2)\sqrt{t}} \prob^*(\tilde{Q}_t \leq y) \leq c' e^{a\epsilon\beta(1+\rho/2)\sqrt{t}} y^{2\mu-4a-a\rho+2},
\end{align*}
for some constant, $c'$, where the last equality follows by Corollary \ref{Qprop} and \eqref{eq:basiccalc}. If we let
\begin{align*}
    y = \exp\left\{-\frac{a\epsilon\beta(1+\rho/2)}{2\mu-4a-a\rho+2}\sqrt{t} \right\},
\end{align*}
then we see that both the ``$\tilde{Q}_t \leq y$''-part and the ``$\tilde{Q}_t > y$''-part are bounded by positive constants. Thus, we are done.
\end{proof}

With the previous lemma at hand, we can now prove the following.

\begin{prop}\label{set}
There exists a constant, $c$, such that if we fix $t > 0$ and let $\tilde{I}_t^u$ be the event that for all $0 \leq s \leq t$,
\begin{align*}
    |\tilde{L}_s-\beta(1+\rho/2)s| \leq u\sqrt{s} \log(2+s) + c,
\end{align*}
then, for every $\epsilon > 0$ there exists a $u < \infty$ such that
\begin{align*}
    \prob^*(\tilde{I}_t^u) \geq 1-\epsilon
\end{align*}
for every $t$.
\end{prop}
\begin{proof}
There is a constant, $c$, such that for any $k \in \N$,
\begin{align*}
    &\prob^* \left( |\tilde{L}_s - \beta(1+\rho/2)s| > u \sqrt{s} \log(2+s) + c \ \text{for some} \ s \in [k,k+1] \right) \\
    &\leq \prob^* \left( |\tilde{L}_{k+1} - \beta(1+\rho/2)(k+1)| > u \sqrt{k+1} \log(2+(k+1)) \right).
\end{align*}
Thus, by splitting into subintervals of length 1, Chebyshev's inequality and Lemma \ref{concentration} (with $p>0$ accordingly)
\begin{align*}
    &\prob^* \left( |\tilde{L}_s - \beta(1+\rho/2)s| > u \sqrt{s} \log(2+s) + c \ \text{for some} \ s \geq 0 \right) \\
    &\leq \sum_{k=1}^\infty \prob^* \left( |\tilde{L}_k - \beta(1+\rho/2)k| > u \sqrt{k} \log(2+k) \right) \\
    &= \sum_{k=1}^\infty \prob^* \left( p\frac{|\tilde{L}_k - \beta(1+\rho/2)k|}{\sqrt{k}} > pu \log(2+k) \right) \\
    &\leq \sum_{k=1}^\infty \E^* \left[ \exp \left\{ p \frac{|\tilde{L}_k-\beta(1+\rho/2)k|}{\sqrt{k}} \right\} \right] (2+k)^{-pu} \\
    &\leq \sum_{k=1}^\infty c_0 (2+k)^{-pu},
\end{align*}
which is $o(1)$ in $u$.
\end{proof}
We shall denote both the event and the indicator function of the event as $\tilde{I}_t^u$, and we will more often than not drop the $u$ in the notation and write $\tilde{I}_t$. Straightforward calculations, using that $\tilde{g}_s'(x) = e^{-a\tilde{L}_s}$, show that on the event of the above proposition, we have
\begin{align}\label{eq:setbounds}
    \psi_0(s)^{-1} e^{-a\beta(1+\rho/2)s} \leq \tilde{g}_s'(x) \leq \psi_0(s) e^{-a\beta(1+\rho/2)s}
\end{align}
where $\psi_0$ is the subexponential function $\psi_0(s) = e^{au \sqrt{s} \log(2+s) +c}$.

Next, we want to convert these facts into the corresponding for $g_{\tau_t}'(x)$. We let $C^* = C^*(x,x_R)$ denote the constant as remarked after Lemma \ref{time}, that is, the constant such that, for $t>0$, $\tilde{t}((t/a -C^*) \vee 0) \leq \tau_t(x) \leq \tilde{t}(t/a + C^*)$. What we will do now, is to define an $\mathscr{F}_{\tau_t}$-measurable version of $\tilde{I}_t^u$ (the indicator of the event of Proposition \ref{set}) and the natural way is to define this as the conditional expectation with respect to this filtration. Fix $u>0$ and write
\begin{align}\label{eventcond}
    I_t^u = \E \left[ \tilde{I}_{\frac{t}{a} + C^*}^u \middle| \mathscr{F}_{\tau_t} \right].
\end{align}
In the next proposition, we will see that this indeed works the same way for $g_{\tau_t}'(x)$ as $\tilde{I}_t^u$ does for $\tilde{g}_t'(x)$. We will omit the superscript and write $I_t = I_t^u$.
\begin{lem}
\label{condset}
Let $u>0$ and $I_t = I_t^u$ be as above. Then there is a subexponential function $\psi$ such that for $\max(0,-\log(x-x_R)) \leq s \leq t$,
\begin{align*}
    \psi(s)^{-1}  e^{-\beta(1+\rho/2)s} I_t \leq g_{\tau_s}'(x) I_t \leq \psi(s) e^{-\beta(1+\rho/2)s} I_t,
\end{align*}
where the implicit constants depend on $x$ and $x_R$. 
\end{lem}
\begin{proof}
Fix $u>0$ and write $s_+ = s/a + C^*$ and $s_- = s/a -C^*$ (where $C^*$ is as described above). Since $t \mapsto g_t'(x)$ is decreasing, we have 
\begin{align*}
    \tilde{g}_{s_+}'(x) \leq g_{\tau_s}'(x) \leq \tilde{g}_{s_-}'(x).
\end{align*}
Hence, by \eqref{eq:setbounds}
\begin{align*}
    g_{\tau_s}'(x) I_t &= \E\left[ g_{\tau_s}'(x) \tilde{I}_{t_+} \middle| \mathscr{F}_{\tau_t} \right] \leq \E\left[ \tilde{g}_{s_-}'(x) \tilde{I}_{t_+} \middle| \mathscr{F}_{\tau_t} \right] \leq \psi_0(s_-) e^{-a\beta(1+\rho/2)s_-} I_t \\
    &=\psi_0(s_-) e^{C^* a\beta(1+\rho/2)} e^{-\beta(1+\rho/2)s} I_t.
\end{align*}
In the same way,
\begin{align*}
    g_{\tau_s}'(x) I_t \geq \E\left[ \tilde{g}_{s_+}'(x) \tilde{I}_{t_+} \middle| \mathscr{F}_{\tau_t} \right] \geq \psi_0(s_+)^{-1} e^{-C^* a\beta(1+\rho/2)} e^{-\beta(1+\rho/2)s} I_t,
\end{align*}
and the lemma is proven.
\end{proof}
\begin{rmk}\label{setrmk}
We remark that we can allow a larger constant in the definition of the event in Proposition $\ref{set}$. The choice of constant $c$ is not important, as if we let $\tilde{I}_t^{u,\tilde{c}}$ denote the event where we replace $c$ by $\tilde{c}>c$, the same estimates hold with the subexponential function $\psi_1(s) = e^{au\sqrt{s} \log(2+s)+\tilde{c}}$ in place of $\psi_0$. Hence, the correct asymptotic behaviour of $\tilde{g}_s'(x)$ is preserved on $\tilde{I}_t^{u,\tilde{c}}$. Furthermore, $\tilde{I}_t^u \subset \tilde{I}_t^{u,\tilde{c}}$.
\end{rmk}


\section{Two-point estimate}\label{sec2PEIG}
\subsection{Outline}
In this section, we use the imaginary geometry techniques to prove a two-point estimate that we need for the lower bound on the dimension of $V_\beta^*$ (and hence $V_\beta$). We follow the ideas of Section 3.2 of \cite{MW17} and we will keep the notation similar. Note that we will write the proof for flow lines, i.e., $\kappa <4$, but the merging property and every lemma that we will need, hold for the level lines of the GFF as well, so the method also gives the two-point estimate in the case $\kappa = 4$. The main idea is to use the merging of the flow lines and the approximate independence of GFF in disjoint regions to ``move the problem between scales'' and separate the points when at the right scale.

We let $h$ be a GFF in $\UH$ with boundary conditions such that the flow line $\eta$ from $0$ is an $\SLE_\kappa(\rho)$ process from $0$ to $\infty$. We define a sequence of random variables $E^n(x)$, for $x \in \R$ and $n \in \N$, such that if $E^n(x)>0$ for every $n \in \N$, then $x \in V_\beta^*$ and we say that $x$ is a \textit{perfect point}. The idea for the construction of the random variables is as follows. Consider the event $A_0^1(x)$, that $\eta$ hits the ball $B(x,\epsilon_1)$, $\epsilon_1 = e^{-\alpha_1}$ and let $E^0(x) = 1_{A_0^1(x)} I_{\alpha_1}^{u,\Lambda}$, where $I_{\alpha_1}^{u,\Lambda}$ is the random variable of \eqref{eventcond} but with a larger constant $\Lambda$. That is, if $E^0(x)>0$, then $\eta$ gets within distance $\epsilon_1$ of $x$ and the derivative $g_{\tau_t}'(x)$ decays approximately as $e^{-\beta(1+\rho/2)t}$ until $\eta$ hits $B(x,\epsilon_1)$.

We proceed inductively. Assume that $E^k(x)$ is defined and that $\epsilon_j = e^{-\sum_{l=1}^j \alpha_l}$, $\alpha_l > 0$. Let $\eta^{x_{k+1}}$ be the flow line started from the point $x_{k+1} = x-\epsilon_{k+1}/4$. Let $A_{k+1}^1(x)$ be the event that $\eta^{x_{k+1}}$ hits $B(x,\epsilon_{k+2})$, plus some regularity conditions. Furthermore, let $I_{k+1}^{u,\Lambda,k+1}$ denote the random variable corresponding to \eqref{eventcond}, but for $\eta^{x_{k+1}}$ until hitting $B(x,\epsilon_{k+2})$. Next, given that $A_k^1(x)$ and $A_{k+1}^1(x)$ occur, let $A_{k+1}^2(x)$ be the event that $\eta^{x_k}$ hits $\eta^{x_{k+1}}$ plus some regularity conditions. We then set $E^{k+1}(x) = E^k(x) 1_{A_{k+1}^1(x) \cap A_{k+1}^2(x)} I_{k+1}^{u,\Lambda,k+1}$.

In short, we let a sequence of flow lines, on smaller and smaller scales, approach the point $x$, such that each flow line has the correct geometric behaviour as it approaches $x$. Moreover, each flow line hits and merges with the next. In this way, the $\SLE_\kappa(\rho)$ process $\eta$ inherits its geometric behaviour from each of the flow lines. This is very convenient when deriving the two-point estimate, that is, when proving that the correlation of $E^n(x)$ and $E^n(y)$ is small when $|x-y|$ is large. The key property that we use is that the flow lines started within the balls $B(x, |x-y|/(2+\delta_0))$ and $B(y,|x-y|/(2+\delta_0))$ are approximately independent when $\delta_0 > 0$ (in the sense that the Radon-Nikodym derivative between the measures with and without the other set of flow lines present is bounded above and below by a constant). Moreover, the flow lines outside of those balls will also be approximately independent, in the same sense, see Lemma \ref{lem1}. Furthermore, the probability of two subsequent flow lines merging is proportional to $1$, see Lemma \ref{lem2}.

Having a certain decay rate of the derivatives of the conformal maps is equivalent to having a certain decay rate of the harmonic measure from infinity of some set on the real line. This will be essential to us, as it is the tool with which we show that the perfect points actually belong to $V_\beta^*$. Moreover, it is important that $\alpha_j \rightarrow \infty$, but not too quickly. If $\alpha_j$ would not tend to $\infty$, then the perfect points would just be points where $\lim_{s\rightarrow \infty} \frac{1}{s} \log g_{\tau_s}'(x) \in [-\beta(1+\rho/2) - c,-\beta(1+\rho/2) + c]$.

In the next subsection, there will be parameters which at first may look redundant, but in fact play important roles in the regularity conditions. We conclude this subsection by listing them and give brief descriptions of how they are used.
\begin{itemize}
    \item $\underline{\delta \in (0,\frac{1}{2})}$: Chosen to be very small and makes sure that the curve $\eta^{x_k}$ does not hit $B(x,\frac{1}{M}\epsilon_k)$ or $B(x,\epsilon_{k+1})$ too close to the real line. Important, as it makes sure that the probability of $\eta^{x_k}$ and $\eta^{x_{k+1}}$ merging does not decrease in $k$. Furthermore, it is needed in the one-point estimate, Lemma \ref{lem3}, as it gives control of a certain martingale.
    \item $\underline{M>0}$: Crucial in the proof that the perfect points belong to $V_\beta^*$. It makes sure that the probability of exiting in the interval between the rightmost point on $\R$ of $\eta^{x_{k+1}}$ (stopped upon hitting $B(x,\epsilon_{k+2})$) and $x$ for a Brownian motion started in $B(x,\epsilon_{k+1})$ depends mostly on $\eta^{x_{k+1}}$ and not on $\eta^{x_k}$. It is chosen to be large, so that the process $Q^k$, under the measure $\prob^*$, will be close in law to its $\prob^*$-invariant distribution when $\eta^{x_{k+1}}$ reaches $B(x,\frac{1}{M}\epsilon_k)$. Moreover, this also makes sure that the probability of $\eta^{x_k}$ and $\eta^{x_{k+1}}$ merging does not decrease in $k$.
    \item $\underline{\Lambda>0}$: Chosen large so that the event $\tilde{I}_t^{u,\Lambda}$ for $\eta^{x_k}$ contains the event $\tilde{I}_t^u$ for the image of $\eta^{x_k}$ under some map $F$.
    \item $\underline{u>0}$: Chosen large enough, so that the event $\tilde{I}_t^u$ has sufficiently large $\prob^*$-probability.
\end{itemize}

\subsection{Perfect points and the two-point estimate}
Throughout this section we fix $\kappa \in (0,4)$ and $\rho \in (-2,\frac{\kappa}{2}-2)$ and let $h$ be a GFF in $\UH$ with boundary values $-\lambda$ on $\R_-$ and $\lambda(1+\rho)$ on $\R_+$, so that the flow line $\eta$ from $0$ to $\infty$ is an $\SLE_\kappa(\rho)$ curve with force point located at $0^+$ (so the configuration is $(\UH,0,0^+,\infty)$). Note that the interval for $\rho$ is chosen so that $\eta$ can hit $\R_+$. We denote the flow line from $x$ by $\eta^x$ and note that for $x>0$, $\eta^x$ is an $\SLE_\kappa(2+\rho,-2-\rho;\rho)$ with configuration $(\UH,x,(0,x^-),x^+,\infty)$. We fix $\delta \in (0,\frac{1}{2})$, $M>0$ large and an increasing sequence, $\alpha_j \rightarrow \infty$, write $\overline{\alpha}_k = \sum_{j=1}^k \alpha_j$ and let $\epsilon_k = e^{-\overline{\alpha}_k}$. The constants $\delta$ and $M$ will be chosen later. As for $\alpha_j$, we define it as $\alpha_j = \alpha_0+\log j$, where $\alpha_0 = \log N$ for some large integer $N$. For $x \geq 1$ and $k \in \N$, we write

\begin{equation*}
x_k = 
\begin{cases}
x-\frac{1}{4}\epsilon_k & \text{if } k \geq 1,\\
0 & \text{if } k = 0.
\end{cases}
\end{equation*}
For $U \subset \UH$, we define
\begin{align*}
    \sigma^x(U) = \inf\{ t\geq 0: \eta^x(t) \in \overline{U} \}
\end{align*}
(when $x=0$, we omit the superscript) and
\begin{align*}
    \sigma_k^x = \sigma^{x_k}(B(x,\epsilon_{k+1})),
\end{align*}
and note that $\sigma(B(x,\epsilon_k))=\tau_{\overline{\alpha}_k}(x)$. Furthermore, let $\sigma_{k,M}^x = \sigma^{x_k}(B(x,\frac{1}{M}\epsilon_k))$. We let $\eta^{x_k,R}$ denote the right side of the flow line $\eta^{x_k}$, $r_t^k = \max \{ \eta^{x_k}([0,t]) \cap \R\}$ and define $Q_t^k$ by
\begin{align*}
    Q_t^k(x) = \frac{\omega_\infty((r_t^k,x],\UH \setminus \eta^{x_k}([0,t]))}{\omega_\infty(\eta^{x_k,R}([0,t]) \cup (r_t^k,x],\UH \setminus \eta^{x_k}([0,t]))}.
\end{align*}
Recall that by \eqref{eq:Qharm}, $Q_t(x)=Q_t^0(x)$ is the diffusion \eqref{eq:Q}. For $k \geq 0$, let $\tilde{I}_t^{u,\Lambda,k}=\tilde{I}_t^{u,\Lambda,k}(x)$ denote the event (as well as the indicator of the event) of Proposition \ref{set}, with constant $\Lambda$ (see Remark \ref{setrmk}) but for the flow line $\eta^{x_k}$, and
\begin{align*}
    I_k^{u,\Lambda,k} = \E \left[ \tilde{I}_{\frac{\alpha_{k+1}}{a}+C_k^*(x)}^{u,\Lambda,k} \middle| \mathscr{F}_{\sigma_k^x} \right],
\end{align*}
as previously. The constants $u$ and $\Lambda$ will be chosen in Lemma \ref{lem3}. Note that the event $\tilde{I}_t^{u,\Lambda,k}$ is a condition on the geometry of the curve which does not change when we rescale (it can be expressed in terms of $Q^k(x)$, which is invariant under scaling of the $\SLE_\kappa(\underline{\rho})$ process). Moreover, if we let $\eta_*^{x_k} = \varphi_k(\eta^{x_k})$, where $\varphi_k(z) = (z-x)/\epsilon_k$ (so that $\eta_*^{x_k}$ is an $\SLE_\kappa(2+\rho,-2-\rho;\rho)$ process with configuration $(\UH,-1/4,(-x/\epsilon_k,0^-),0^+,\infty)$) and let $(g_t^{*,k})_{t \geq 0}$ denote its Loewner chain, then on the event $\{ I_k^{u,\Lambda,k} > 0\}$,
\begin{align}\label{eq:flowset}
    \psi(\alpha_{k+1})^{-1} e^{-\beta(1+\rho/2)\alpha_{k+1}} \leq (g_{\sigma_k^x}^{*,k})'(0) \leq \psi(\alpha_{k+1}) e^{-\beta(1+\rho/2)\alpha_{k+1}},
\end{align}
for some subexponential function $\psi=\psi_{u,\Lambda}$.

We let $A_k^1(x) = A_k^1(x,\delta,M,\alpha_0)$ be the event that 
\begin{enumerate}[(i)]
    \item $\sigma_k^x < \infty$,
    \item $Q_{\sigma_{k,M}^x}^k(x),Q_{\sigma_k^x}^k(x) \in [\delta,1-\delta]$, and
    \item $\sigma_k^x < \sigma^{x_k}(\UH \setminus B(x,\frac{1}{2} \epsilon_k))$,
\end{enumerate}
that is, $\eta^{x_k}$ hits $B(x,\epsilon_{k+1})$ before exiting $B(x,\frac{1}{2}\epsilon_k)$ and it does not hit $B(x,\frac{1}{M}\epsilon_k)$ or $B(x,\epsilon_{k+1})$ ``too far down'' (the latter being due to the condition on $Q_t^k(x)$). Now, we set
\begin{align*}
    E_k^1(x) = 1_{A_k^1(x)} I_k^{u,\Lambda,k}.
\end{align*}
We let $A_k^2(x)=A_k^2(x,\delta,M,\alpha_0)$ be the event that on $A_k^1(x)$ and $A_{k+1}^1(x)$,
\begin{enumerate}[(i)]
    \item $\eta^{x_{k-1}}|_{[\sigma_{k-1}^x,\infty)}$ merges with $\eta^{x_k}|_{[0,\sigma_k^x)}$ before exiting $B(x,\frac{3}{2} \epsilon_k) \setminus B(x,\frac{1}{M}\epsilon_k)$,
    \item $\arg(\eta^{x_{k-1}}(t)-x) \geq \frac{2}{3} \min(\arg(\eta^{x_{k-1}}(\sigma_{k-1}^x)-x),\arg(\eta^{x_k}(\sigma_k^x)-x))$ for $t> \sigma_{k-1}^x$ but before merging with $\eta^{x_k}$,
\end{enumerate}
that is, property (ii) makes sure that the curve does not get ``too close'' to $\R_+$, see Figure \ref{fig:motivation}. We let $E_k^2(x) = 1_{A_k^2(x)}$ be the indicator of that event. Next, we let $E_k(x) = E_k^1(x) E_k^2(x)$ and write
\begin{align*}
    E^{m,n}(x) = E_{m+1}^1(x) \prod_{k=m+2}^n E_k(x),
\end{align*}
and $E^n(x) = E^{-1,n}(x)$. 

\begin{figure}[ht!]
\centering
\includegraphics[width=140mm]{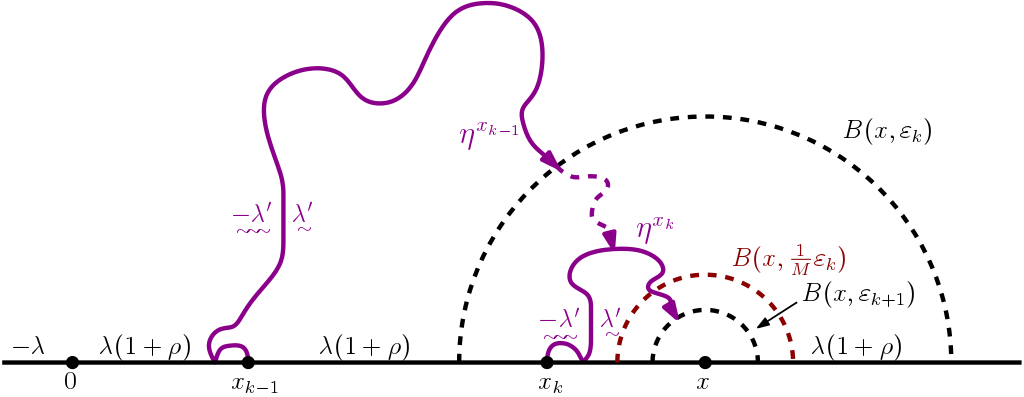}
\caption{If $E_k^1(x)>0$, then $\eta^{x_{k-1}}$ hits $B(x,\epsilon_k)$, $Q_{\sigma_{k,M}^x},Q_{\sigma_k^x}^k \in [\delta, 1-\delta]$ and the derivatives of the Loewner chain for $\eta^{x_{k-1}}$ behave as we want. Furthermore, given that $E_k^1(x) > 0$, we have that if $E_k^2(x) = 1$, then $\eta^{x_{k-1}}$ merges with $\eta^{x_k}$ before exiting $B(x,\frac{1}{2} \epsilon_k) \setminus B(x,\frac{1}{M}\epsilon_k)$. \label{fig:event}}
\end{figure}

\begin{figure}[ht!]
\centering
\includegraphics[width=150mm]{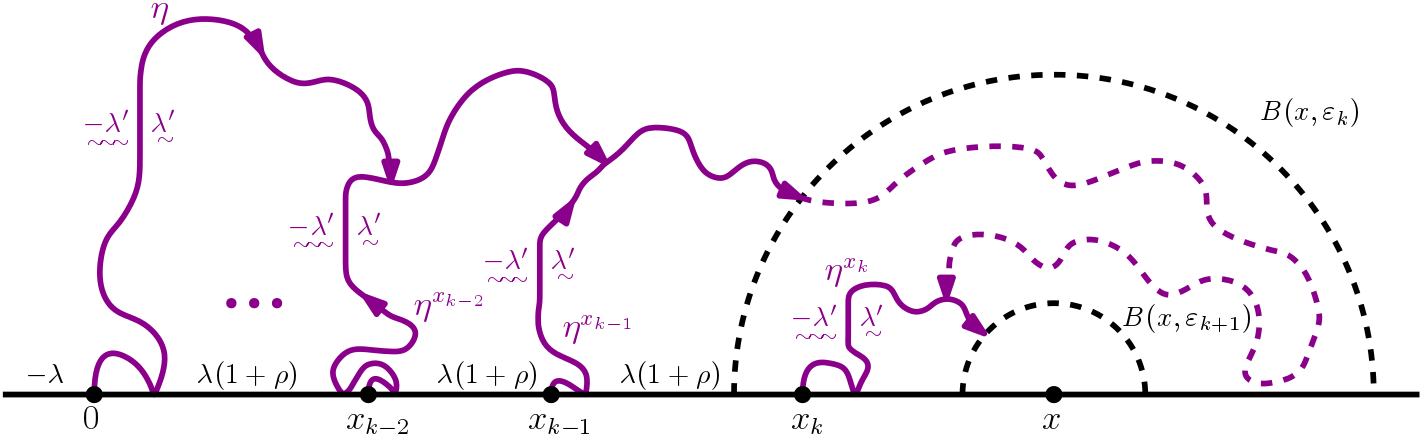}
\caption{Condition (ii) on $A_k^2(x)$ ensures that we will not have the case in the above figure -- instead there will be some sector which the flow lines will not enter. \label{fig:motivation}}
\end{figure}

Why this is the right setting and these conditions are the correct ones to look for might not be clear at first sight. This, we prove in the next lemma.
\begin{lem}\label{lemjust}
If $E^n(x)>0$ for each $n\in \N$, then $x \in V_\beta^*$.
\end{lem}
\begin{proof}
First, note that we are considering the decay of the conformal maps at a sequence of times $\overline{\alpha}_k \rightarrow \infty$, rather than as the limit over a continuum. However, by the monotonicity of the map $t \mapsto g_t'(x)$, this is sufficient. By the Koebe 1/4 theorem,
\begin{align}\label{eq:quarter}
    \frac{e^{-\overline{\alpha}_k}}{4}g_{\tau_{\overline{\alpha}_k}}'(x) \leq \omega_\infty((r_{\overline{\alpha}_k},x],\UH \setminus K_{\overline{\alpha}_k}) \leq 4 e^{-\overline{\alpha}_k} g_{\tau_{\overline{\alpha}_k}}'(x)
\end{align}
for each integer $k \in \N$. Hence, it is enough to see that the decay rate of $\omega_\infty$ is the correct one.

Let $\widehat{K}_n$ denote the closure of the complement of the unbounded connected component of $\UH \setminus (\eta([0,\tau_{\overline{\alpha}_{n+1}}]) \cup \eta^{x_n}([0,\sigma_n^x]))$. Clearly, on the event $\{E^n(x)>0\}$,
\begin{align*}
    \omega_\infty((r_{\sigma_n^x}^{n},x],\UH \setminus \widehat{K}_n) \leq \omega_\infty((r_{\tau_{\overline{\alpha}_{n+1}}},x],\UH \setminus K_{\tau_{\overline{\alpha}_{n+1}}}),
\end{align*}
since $K_{\tau_{\overline{\alpha}_{n+1}}} \subset \widehat{K}_n$ and $(r_{\sigma_n^x}^n,x] \subset (r_{\tau_{\overline{\alpha}_{n+1}}},x]$. In view of $\omega$ as the hitting probability of a Brownian motion, it is easy to see that
\begin{align}\label{eq:harmineq}
    \omega_\infty((r_{\sigma_n^x}^n,x],\UH \setminus \widehat{K}_n) \gtrsim \omega_\infty((r_{\tau_{\overline{\alpha}_{n+1}}},x],\UH \setminus K_{\tau_{\overline{\alpha}_{n+1}}}),
\end{align}
where the implicit constant is independent of $n$. Indeed, if $L_n$ denotes the line segment $[x,\eta^{x_n}(\sigma_n^x)]$, then a Brownian motion, started in the unbounded connected component of $\UH \setminus (\widehat{K}_n \cup L_n)$, which exits $\UH \setminus \widehat{K}_n$ in either of the two intervals $(r_{\tau_{\overline{\alpha}_{n+1}}},r_{\sigma_n^x}^n]$ and $(r_{\sigma_n^x}^n,x]$, must first hit the line segment $L_n$. However, from any point $z \in L_n$, $\dist(z,(r_{\tau_{\overline{\alpha}_{n+1}}},r_{\sigma_n^x}^n]) > \epsilon_{n+1}$, $\dist(z,(r_{\sigma_n^x}^n,x]) \leq \epsilon_{n+1}$ and $x-r_{\sigma_n^x}^n > \epsilon_{n+1}$. Hence, the conditional probability of the Brownian motion exiting in $(r_{\sigma_n^x}^n,x]$, given that it will exit in $(r_{\tau_{\overline{\alpha}_{n+1}}},x]$ is greater than some $\hat{p}>0$. That the constant is independent of $n$ follows from scale invariance. See Figure \ref{fig:justification} for the illustration of \eqref{eq:harmineq}. Thus, we have proven that on the event $\{E^n(x)>0\}$,
\begin{align}\label{harmequiv}
    \omega_\infty((r_{\sigma_n^x}^{n},x],\UH \setminus \widehat{K}_n) \asymp \omega_\infty((r_{\tau_{\overline{\alpha}_{n+1}}},x],\UH \setminus K_{\tau_{\overline{\alpha}_{n+1}}}).
\end{align}

Finally, we shall prove that $\omega_\infty((r_{\sigma_n^x}^{n},x],\UH \setminus \widehat{K}_n)$ has the correct decay rate, that is, that
\begin{align*}
    \lim_{n \rightarrow \infty} \frac{1}{\overline{\alpha}_{n+1}} \log \omega_\infty((r_{\sigma_n^x}^{n},x],\UH \setminus \widehat{K}_n) = -1-\beta(1+\rho/2).
\end{align*}
We start with the upper bound. Then, since $\eta([0,\tau_{\alpha_1}]) \cup \bigcup_{j=1}^n \eta^{x_j}([0,\sigma_j^x])) \subset \widehat{K}_n$, we have, by the Markov property for Brownian motion,
\begin{align}\label{eq:UBMarkov}
    \omega_\infty((r_{\sigma_n^x}^{n},x],\UH \setminus \widehat{K}_n) &\leq \omega_\infty\Big((r_{\sigma_n^x}^n,x],\UH \setminus (\eta([0,\tau_{\alpha_1}]) \cup \bigcup_{j=1}^n \eta^{x_j}([0,\sigma_j^x])) \Big) \nonumber \\
    &\leq \omega_\infty\left(\partial B\left(x,\frac{\epsilon_1}{2} \right),\UH\setminus\eta([0,\tau_{\alpha_1}])\right) \nonumber \\
    &\quad \times \prod_{j=1}^{n-1} \sup_{z \in \partial B\left(x,\frac{1}{2} \epsilon_{j}\right)} \omega\left(z,\partial B\left(x,\frac{\epsilon_{j+1}}{2} \right), \UH\setminus\eta^{x_j}([0,\sigma_j^x])\right) \nonumber \\
    &\quad \times \sup_{z \in \partial B\left(x,\frac{1}{2} \epsilon_n\right)} \omega\left(z,(r_{\sigma_n^x}^n,x], \UH\setminus\eta^{x_n}([0,\sigma_n^x])\right),
\end{align}
using that if $K^1 \subset K^2 \subset \UH$, then $\omega(z,E,K^1) \geq \omega(z,E,K^2)$ for $E \subset \R \setminus \partial K^2$ and $z \in \UH\setminus K^2$ (by removing obstacles, we allow more Brownian paths, and hence the probability of exiting in that interval increases). Next, note that for $z \in \partial B\left(x,\frac{1}{2} \epsilon_j\right)$, we have that
\begin{align*}
    \omega\left(z,\partial B\left(x,\frac{\epsilon_{j+1}}{2} \right), \UH\setminus\eta^{x_j}([0,\sigma_j^x])\right) \asymp  \omega(z,(r_{\sigma_j^x}^j,x],\UH\setminus \eta^{x_j}([0,\sigma_j^x])),
\end{align*}
where the implicit constant is independent of both $z$ and $j$. This holds since the sizes of, as well as distances to $z$ from $\partial B(x,\frac{\epsilon_{j+1}}{2})$ and $(r_{\sigma_j^x}^j,x]$ are of the same order. Thus,
\begin{align*}
    \sup_{z \in \partial B\left(x,\frac{1}{2} \epsilon_j\right)} \omega\left(z,\partial B\left(x,\frac{\epsilon_{j+1}}{2} \right), \UH\setminus\eta^{x_j}([0,\sigma_j^x])\right) \leq C \sup_{z \in \partial B\left(x,\frac{1}{2} \epsilon_j\right)} \omega(z,(r_{\sigma_j^x}^j,x],\UH\setminus \eta^{x_j}([0,\sigma_j^x])),
\end{align*}
where $C$ is independent of $j$. Moreover, note that on the event $\{E^n(x) > 0\}$, the condition $I_j^{u,\Lambda,j}>0$ implies that \eqref{eq:flowset} holds, and using the Koebe 1/4 theorem as in \eqref{eq:quarter}, we have
\begin{align}\label{eq:quarterharm}
    \psi(\alpha_{j+1})^{-1} e^{-\alpha_{j+1}(1+\beta(1+\rho/2))} \leq \omega_\infty((\varphi_j(r_{\sigma_j^x}^j),0],\UH\setminus\eta_*^{x_j}([0,\sigma_j^x])) \leq \psi(\alpha_{j+1}) e^{-\alpha_{j+1}(1+\beta(1+\rho/2))}
\end{align}
for some subexponential function $\psi = \psi_{u,\Lambda}$. Next, we need that
\begin{align}\label{eq:needthis}
    \omega_\infty((\varphi_j(r_{\sigma_j^x}^j),0],\UH\setminus\eta_*^{x_j}([0,\sigma_j^x])) \asymp \sup_{z \in \partial B\left(x,\frac{1}{2} \epsilon_j\right)} \omega(z,(r_{\sigma_j^x}^j,x],\UH\setminus \eta^{x_j}([0,\sigma_j^x])),
\end{align}
where the implicit constant is independent of $j$. By Harnack's inequality, we can choose an arc $S \subset \partial B(0,\frac{1}{2})$, depending only on the parameters $\delta$, $M$, $\Lambda$ and $u$, such that for each $z \in \varphi_j^{-1}(S) \subset \partial B\left(x,\frac{1}{2} \epsilon_j\right)$,
\begin{align}\label{eq:Sequiv}
    \omega(z,(r_{\sigma_j^x}^j,x],\UH\setminus \eta^{x_j}([0,\sigma_j^x])) \asymp \sup_{z \in \partial B\left(x,\frac{1}{2} \epsilon_j\right)} \omega(z,(r_{\sigma_j^x}^j,x],\UH\setminus \eta^{x_j}([0,\sigma_j^x])).
\end{align}
The fact that this will hold for every $j$ follows since the same geometric restrictions are imposed on each flow line $\eta^{x_j}$. We let $\varrho$ and $\tau_*$ denote the first exit time of $\UH\setminus B(0,\frac{1}{2})$ and $\UH \setminus \eta_*^{x_j}([0,\sigma_j^x])$, respectively (recall that $\eta_*^{x_j} = \varphi_j(\eta^{x_j})$). Then,
\begin{align*}
    \omega_\infty((\varphi_j(r_{\sigma_j^x}^j),0],\UH\setminus\eta_*^{x_j}([0,\sigma_j^x])) &\asymp \lim_{y \rightarrow \infty} y\prob^{iy}(B_\varrho \in S, B_{\tau_*} \in (\varphi_j(r_{\sigma_j^x}^j),0]) \\
    &= \lim_{y \rightarrow \infty} y\int_S \prob(B_{\tau_*} \in (\varphi_j(r_{\sigma_j^x}^j),0] | B_\varrho = z) d\prob^{iy}(B_\varrho = z) \\
    &= \lim_{y \rightarrow \infty} y\int_S \omega(\varphi_j^{-1}(z),(r_{\sigma_j^x}^j,x],\UH\setminus \eta^{x_j}([0,\sigma_j^x])) d\prob^{iy}(B_\varrho = z) \\
    &\asymp \sup_{z \in \partial B\left(x,\frac{1}{2} \epsilon_j\right)} \omega(z,(r_{\sigma_j^x}^j,x],\UH\setminus \eta^{x_j}([0,\sigma_j^x])) \lim_{y \rightarrow \infty} y\prob^{iy}(B_\varrho \in S) \\
    &\asymp \sup_{z \in \partial B\left(x,\frac{1}{2} \epsilon_j\right)} \omega(z,(r_{\sigma_j^x}^j,x],\UH\setminus \eta^{x_j}([0,\sigma_j^x])) 
\end{align*}
where we used the fact that $\omega_\infty(\partial B(0,\frac{1}{2}),\UH\setminus B(0,\frac{1}{2})) \asymp \omega_\infty(S, \UH \setminus B(0,\frac{1}{2}))$ together with \eqref{eq:Sequiv} on the first line, the conformal invariance of Brownian motion on the third line, \eqref{eq:Sequiv} on the fourth line and that $\omega_\infty(S,\UH \setminus B(0,\frac{1}{2})) \asymp 1$ on the fifth line. Thus \eqref{eq:needthis} holds. Combining this with \eqref{eq:quarterharm}, we have
\begin{align*}
    \sup_{z \in \partial B\left(x,\frac{1}{2} \epsilon_j\right)} \omega(z,(r_{\sigma_j^x}^j,x],\UH\setminus \eta^{x_j}([0,\sigma_j^x])) \leq \hat{C} \psi(\alpha_{j+1}) e^{-\alpha_{j+1}(1+\beta(1+\rho/2))},
\end{align*}
for some constant $\hat{C}$. Thus, combining this with \eqref{eq:UBMarkov}
\begin{align*}
    \omega_\infty((r_{\sigma_n^x}^{n},x],\UH \setminus \widehat{K}_n) \leq \tilde{C}^{n+1} \left(\prod_{j=1}^{n+1} \psi(\alpha_j) \right) e^{-\overline{\alpha}_{n+1} (1+\beta(1+\rho/2))},
\end{align*}
that is,
\begin{align*}
    \lim_{n \rightarrow \infty} \frac{1}{\overline{\alpha}_{n+1}} \log \omega_\infty((r_{\sigma_n^x}^n,x],\UH \setminus \widehat{K}_n) &\leq \lim_{n \rightarrow \infty} \frac{(n+1) \log \tilde{C}}{\overline{\alpha}_{n+1}} + \frac{\sum_{j=1}^{n+1} \psi(\alpha_j)}{\overline{\alpha}_{n+1}} - 1-\beta(1+\rho/2) \\
    &= -1-\beta(1+\rho/2).
\end{align*}
We now turn to the lower bound. We begin by writing $\tilde{\tau}_0 = \inf\{ t>0: \eta(t) \in \eta^{x_1}([0,\sigma_1^x]) \}$. Next, $A_1^2(x)$ and (ii) of $A_0^1(x)$ and $A_1^1(x)$ imply that $\eta((\tau_{\alpha_1},\tilde{\tau}_0])$ is not ``too close'' to $\R$, in the sense that there will be a sector $\{ z: 0<\arg (z-x) < c\}$ that the curve will not enter, and the distance from $\eta([0,\tilde{\tau}_0])$ to $x$ is of the same order as the distance from $\eta([0,\tau_{\alpha_1}])$ to $x$. Thus,
\begin{align*}
    \omega_\infty((r_{\tau_{\alpha_1}},x],\UH\setminus \eta([0,\tilde{\tau}_0])) \asymp \omega_\infty((r_{\tau_{\alpha_1}},x],\UH\setminus \eta([0,\tau_{\alpha_1}])),
\end{align*}
where the implicit constant depends only on $\delta$, $M$ and $\Lambda$. In fact, the probability of the Brownian motion hitting $(r_{\alpha_1},x]$ and some arc $S_1 = \{ \frac{1}{2} \epsilon_1 e^{i\theta}: \theta \in [\theta_1(\delta), \theta_2(\delta)], \ \theta_1(\delta) > 0 \}$, such that $S_1 \cap \eta([0,\tilde{\tau}_0]) = \emptyset$ is proportional to the probability of the Brownian motion hitting $(r_{\alpha_1},x]$. That is, let $\upsilon$ denote the exit time of $\UH \setminus \eta([0,\tilde{\tau}_0])$ for the Brownian motion, then 
\begin{align}\label{harmp1}
    \lim_{y \rightarrow \infty} y\prob^{iy}( B_\upsilon \in (r_{\tau_{\alpha_1}},x], \ B_{[0,\upsilon]} \cap S_1 \neq \emptyset) \asymp \omega_\infty((r_{\tau_{\alpha_1}},x],\UH\setminus \eta([0,\tau_{\alpha_1}])),
\end{align}
where, again, the implicit constants depend only on $\delta$, $M$ and $\Lambda$. By the same reasoning, if $S_k = \{ \frac{1}{2}\epsilon_k e^{i\theta}: \theta \in [\theta_1(\delta), \theta_2(\delta)], \ \theta_1(\delta) > 0 \}$, then for every $z \in S_{k-1}$,
\begin{align}\label{harmpk}
    \prob^z(B_{\upsilon_k} \in (r_{\sigma_k^x}^k,x], B_{[0,\upsilon_k]} \cap S_k \neq \emptyset) \asymp \omega(z,(r_{\sigma_k^x}^k,x],\UH \setminus \eta^{x_k}([0,\sigma_k^x])),
\end{align}
where $\upsilon_k$ denotes the exit time from $\UH \setminus \widehat{K}_k$ for the Brownian motion, and the implicit constant does not depend on $k$. Thus, by the Markov property, \eqref{harmp1}, and \eqref{harmpk} together with \eqref{eq:quarterharm},
\begin{align*}
    \omega_\infty((r_{\sigma_n^x}^{n},x],\UH \setminus \widehat{K}_n) &\geq \lim_{y \rightarrow \infty} y\prob^{iy}( B_\upsilon \in (r_{\tau_{\alpha_1}},x], \ B_{[0,\upsilon]} \cap S_1 \neq \emptyset) \\
    &\quad \times \prod_{j=1}^n \inf_{z \in S_j} \prob^z(B_{\upsilon_j} \in (r_{\sigma_j^x}^j,x], B_{[0,\upsilon_j]} \cap S_j \neq \emptyset) \\
    &\geq \tilde{c}^{n+1} \left( \prod_{j=1}^{n+1} \psi(\alpha_j)^{-1} \right) e^{-\overline{\alpha}_{n+1} (1+\beta(1+\rho/2))}.
\end{align*}
Hence, 
\begin{align*}
    \lim_{n \rightarrow \infty} \frac{1}{\overline{\alpha}_{n+1}} \log \omega_\infty((r_{\sigma_n^x}^n,x],\UH \setminus \widehat{K}_n) &\geq \lim_{n \rightarrow \infty} \frac{(n+1) \log \tilde{c}}{\overline{\alpha}_{n+1}} - \frac{\sum_{j=1}^{n+1} \psi(\alpha_j)}{\overline{\alpha}_{n+1}} - 1-\beta(1+\rho/2) \\
    &= -1-\beta(1+\rho/2).
\end{align*}
Therefore, by \eqref{harmequiv},
\begin{align*}
    \lim_{n \rightarrow \infty} \frac{1}{\overline{\alpha}_{n+1}} \log \omega_\infty((r_{\tau_{\overline{\alpha}_{n+1}}},x],\UH \setminus K_{\tau_{\overline{\alpha}_{n+1}}}) = -1-\beta(1+\rho/2),
\end{align*}
and consequently by \eqref{eq:quarter},
\begin{align*}
    \lim_{n \rightarrow \infty} \frac{1}{\overline{\alpha}_n} \log g_{\tau_{\overline{\alpha}_n}}'(x) = -\beta(1+\rho/2).
\end{align*}
Thus, if $E^n(x)>0$ for every $n \in \N$, then $x \in V_\beta^*$.
\end{proof}

\begin{figure}[ht!]
\centering
\includegraphics[width=140mm]{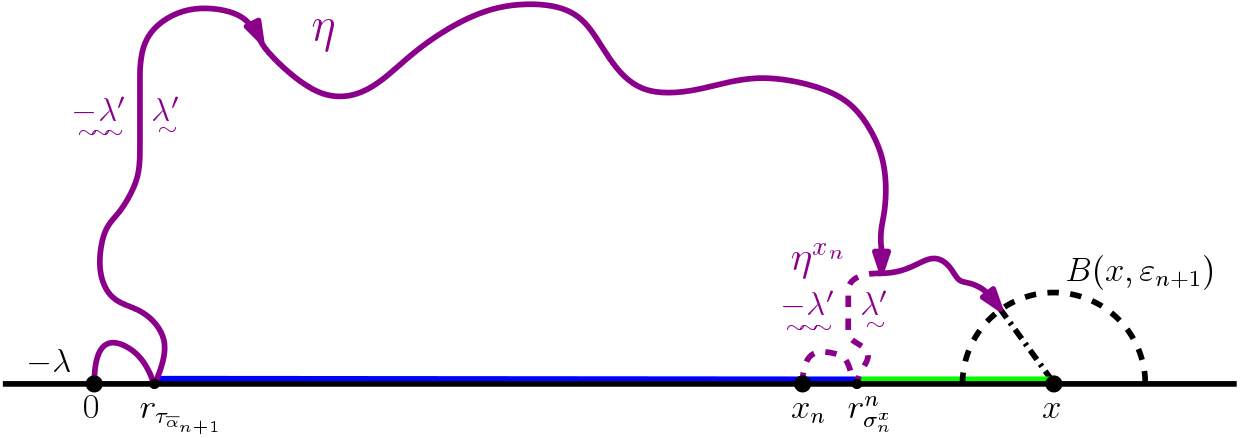}
\caption{A Brownian motion started from the dash-dotted line will hit the green interval with probability proportional to that of hitting the union of the green and the blue intervals. \label{fig:justification}}
\end{figure}

The two-point estimate which we will acquire here is the following. We consider $x,y\in [1,2]$ out of convenience, but it will be clear that the same proof works for every compact interval.
\begin{prop}\label{2PEIG}
For each sufficiently small $\delta \in (0,\frac{1}{2})$ there exists a subpower function $\tilde{\Psi}_\delta$ such that for all $x,y \in [1,2]$ and $m \in \N$ such that $2\epsilon_{m+2} \leq |x-y| \leq \frac{1}{2} \epsilon_m$, we have
\begin{align*}
    \E[E^n(x)E^n(y)] \leq \tilde{\Psi}_\delta(\epsilon_{m+2}^{-1}) \epsilon_{m+2}^{(\zeta\beta-\mu)(1+\rho/2)} \E[E^n(x)] \E[E^n(y)].
\end{align*}
\end{prop}

\begin{rmk}
Noting that $\epsilon_{m+2} = \epsilon_m^{1+o_m(1)}$ as $m \rightarrow \infty$, we can write Proposition \ref{2PEIG} as
\begin{align*}
    \E[E^n(x)E^n(y)] \leq \Psi_\delta(1/|x-y|) |x-y|^{(\zeta\beta-\mu)(1+\rho/2)} \E[E^n(x)] \E[E^n(y)],
\end{align*}
where $\Psi_\delta(s)$ is a subpower function.
\end{rmk}
The main ingredients in the proof are divided into three lemmas; the first of which establishes ``approximate independence'' between flow line interactions in different regions; the second states that merging of these flow lines happens with high enough probability and the third is a one-point estimate.

\begin{figure}[ht!]
\centering
\includegraphics[width=150mm]{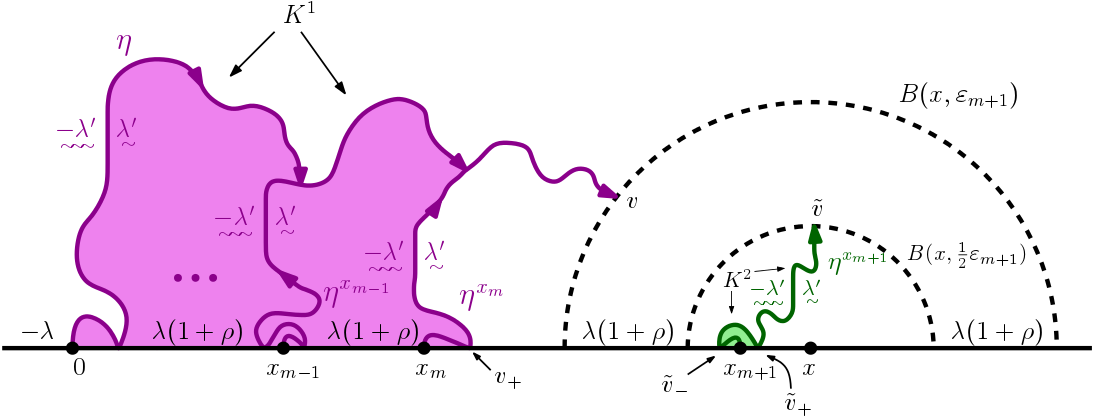}
\caption{$K^1$, shown in purple, and $K^2$, shown in green, are the closures of the complement of the unbounded connected component of $\UH \setminus (\eta([0,\tau_{\overline{\alpha}_{m+1}}]) \cup \eta^{x_m}([0,\sigma_m^x]))$ and $\UH \setminus \eta^{x_{m+1}}([0,\tau])$ respectively, where $\tau$ is the first exit time of $U = B(x,\frac{1}{2} \epsilon_{m+1})$ for $\eta^{x_{m+1}}$. \label{fig:lem1}}
\end{figure}

\begin{lem}\label{lem1}
For every $x \geq 1$ and $m,n \in \N$ such that $m \leq n$, it holds that
\begin{align*}
    \E[E^m(x) E^{m,n}(x)] \asymp \E[E^m(x)]\E[E^{m,n}(x)].
\end{align*}
Furthermore, if $y$ is such that $2\epsilon_{m+2} \leq |x-y| \leq \frac{1}{2} \epsilon_m$, then
\begin{align*}
    \E[E^{m-1}(x) E^{m+1,n}(x) E^{m+1,n}(y)] \asymp \E[E^{m-1}(x)]\E[E^{m+1,n}(x)]\E[E^{m+1,n}(y)].
\end{align*}
The constants in $\asymp$ may depend on $\kappa$ and $\rho$.
\end{lem}
\begin{proof}
In order to prove the first part, it suffices to prove that 
\begin{align*}
    \E[E^{m,n}(x)|E^m(x)] 1_{\{E^m(x)>0\}} \asymp \E[E^{m,n}(x)] 1_{\{E^m(x)>0\}}.
\end{align*}
Let $v = \eta(\tau_{\overline{\alpha}_{m+1}})$, denote by $K^1$ the closure of the unbounded component of 
\begin{align*}
    \UH \setminus (\eta([0, \tau_{\overline{\alpha}_{m+1}}])\cup \eta^{x_m}([0,\sigma_m^x]))
\end{align*}
(see Figure \ref{fig:lem1}) and let $v_+ = \max \{K^1 \cap \R \}$. As stated above, $\eta^{x_{m+1}}$ is an $\SLE_\kappa(2+\rho,-2-\rho;\rho)$ curve with configuration
\begin{align*}
c = (\UH,x_{m+1},(0,x_{m+1}^-),x_{m+1}^+,\infty).
\end{align*}
Thus, recalling Remark \ref{otherdomains} and Figure \ref{fig:winding} and noting the boundary conditions illustrated in Figure \ref{fig:lem1}, the conditional law of $\eta^{x_{m+1}}$ given $\eta|_{[0,\tau_{\overline{\alpha}_{m+1}}]}$, $\eta^{x_m}|_{[0,\sigma_m^x]}$ and $E^m(x)$, restricted to the event $\{E^m(x) > 0\}$ (recall that $E^m(x)$ is determined by $\eta|_{[0,\tau_{\overline{\alpha}_{m+1}}]}$ and $\eta^{x_m}|_{[0,\sigma_m^x]}$) is that of an $\SLE_\kappa(2,\rho,-2-\rho;\rho)$ process with configuration
\begin{align*}
    \tilde{c} = (\UH \setminus K^1, x_{m+1}, (v,v_+,x_{m+1}^-),x_{m+1}^+,\infty).
\end{align*}
(Note that the boundary data to the left of $v$ on $K^1$ is the same as on $\R_-$, so that the leftmost force point is indeed $v$.)

Let $U = B(x,\frac{1}{2} \epsilon_{m+1})$ and let $\tau = \sigma^{x_{m+1}}(\UH \setminus U)$ be the exit time of $U$. Also, let $K^2$ be the closure of the complement of the unbounded component of $\UH \setminus \eta^{x_{m+1}}([0,\tau])$, $\tilde{v} = \eta^{x_{m+1}}(\tau)$, $\tilde{v}_- = \min\{K^2 \cap \R\}$ and $\tilde{v}_+ = \max\{K^2 \cap \R\}$. Furthermore, let
\begin{align*}
    &c_\tau = (\UH \setminus K^2, \tilde{v},(0,\tilde{v}_-),\tilde{v}_+,\infty), \\
    &\tilde{c}_\tau = (\UH \setminus (K^1 \cap K^2), \tilde{v},(v,v_+,\tilde{v}_-),\tilde{v}_+,\infty).
\end{align*}
Then, by Lemma \ref{IGlem4},
\begin{align*}
\frac{d\mu_{\tilde{c}}^U}{d\mu_c^U} = \frac{Z(\tilde{c}_\tau)/Z(\tilde{c})}{Z(c_\tau)/Z(c)} \exp \left( -\xi m(\UH,K^1,K^2) \right).
\end{align*}
Since $\dist(K^1,x) = \epsilon_{m+1}$, and $K^2 \subseteq \overline{B(x,\frac{1}{2} \epsilon_{m+1})}$, we have that $\dist(K^1,K^2) \gtrsim \epsilon_{m+1}$. Also, $\text{diam}(U) = \epsilon_{m+1}$ and hence
\begin{align}\label{eq:canrescale}
    \frac{\dist(K^1,K^2)}{\text{diam}(U)} \gtrsim 1.
\end{align}
Thus, after rescaling, we are in the setting of Lemma \ref{IGlem5} and thus there exists a constant $C \geq 1$ such that 
\begin{align*}
    \frac{1}{C} \leq \frac{d\mu_{\tilde{c}}^U}{d\mu_c^U} \leq C,
\end{align*}
that is, the Radon-Nikodym derivative between the law of $\eta^{x_{m+1}}$ stopped upon exiting $U$, given $\eta([0,\tau_{\overline{\alpha}_{m+1}}])$, $\eta^{x_m}([0,\sigma_m^x])$ and $E^m(x)$, restricted to the event $\{E^m(x) > 0\}$, and the law of $\eta^{x_{m+1}}$ stopped upon exiting $U$ is bounded above and below by constants. Moreover, by \eqref{eq:canrescale}, the constant $C$ is independent of $m$. Hence,
\begin{align*}
    \E[E^{m,m+1}(x)|E^m(x)] 1_{\{E^m(x)>0\}} \asymp \E[E^{m,m+1}(x)] 1_{\{E^m(x)>0\}},
\end{align*}
and the case $n = m+1$ is proven. Now, suppose that $n \geq m+2$. Obviously, the Radon-Nikodym derivative between the law of $\eta^{x_n}$ stopped upon leaving the component of $B(x,\frac{1}{2}\epsilon_n) \setminus \eta^{x_{m+1}}([0,\tau])$ in which it starts growing and the law where we condition on $K^1$ and $E^m(x)$, restricted to $\{E^m(x) > 0\}$, as well, is bounded above and below by a constant, by the very same argument as above. (Note that the distances have the same lower bound and that $U$ is unchanged.) Furthermore, conditional on $\eta^{x_{m+1}}([0,\sigma^{x_{m+1}}(B(x,\epsilon_{n+1}))])$ and $\eta^{x_n}([0,\sigma_n^x])$ merging, the joint laws of $\eta^{x_j}|_{[0,\sigma_j^x]}$ for $j = 1,\dots,m$ and $\eta^{x_k}|_{[0,\sigma_k^x]}$ for $k=m+2,\dots,n-1$ are independent, hence the proof of the first part is done.

The second part is proven in the same way, noting that if $\tau^y = \sigma^{y_{m+2}}(\UH \setminus B(y,\frac{1}{2}\epsilon_{m+2}))$, $\tilde{\tau}^y = \inf\{ t>0: \eta^{y_n}(t) \notin B(y,\frac{1}{2}\epsilon_n) \setminus \eta^{y_{m+2}}([0,\tau^y]) \}$ and $K^3$ is the closure of the complement of the unbounded connected component of $\UH \setminus \left(\eta^{y_{m+2}}|_{[0,\tau^y]} \cup \eta^{y_n}|_{[0,\tilde{\tau}^y]} \right)$, then $\dist(K^1 \cup K^2,K^3) \gtrsim \epsilon_{m+2}$ and $\text{diam}(B(y,\frac{1}{2}\epsilon_{m+2})) = \epsilon_{m+2}$. Thus, as above, we can rescale and apply Lemma \ref{IGlem5}, to see that
\begin{align*}
    \E[E^{m+1,n}(y)|E^{m-1}(x),E^{m+1,n}(x)] 1_{\{E^{m-1}(x)>0, E^{m+1,n}(x)>0\}} \asymp \E[E^{m+1,n}(y)] 1_{\{E^{m-1}(x)>0, E^{m+1,n}(x)>0\}},
\end{align*}
that is,
\begin{align*}
    \E[E^{m-1}(x) E^{m+1,n}(x) E^{m+1,n}(y)] \asymp \E[E^{m-1}(x)E^{m+1,n}(x)]\E[E^{m+1,n}(y)].
\end{align*}
Repeating the above argument, noting that $\dist(K^1,K^2)/\textup{diam}(B(x,\frac{1}{2} \epsilon_{m+2})) \gtrsim 1$, we have that
\begin{align*}
    \E[E^{m-1}(x)E^{m+1,n}(x)] \asymp \E[E^{m-1}(x)] \E[E^{m+1,n}(x)],
\end{align*}
and the proof is done.
\end{proof}

\begin{lem}\label{lem2}
For each $x \geq 1$ and $m,n \in \N$ such that $m \leq n$, it holds that
\begin{align*}
    \E[E^n(x)] \asymp \E[E^m(x)] \E[E^{m,n}(x)],
\end{align*}
where the constants can depend on $\kappa$, $\rho$, $\delta$ and $M$.
\end{lem}
\begin{proof}
We begin by noting that by the first part of Lemma \ref{lem1},
\begin{align*}
    \E[E^n(x)] \leq \E[E^m(x) E^{m,n}(x)] \lesssim \E[E^m(x)] \E[E^{m,n}(x)].
\end{align*}
Thus, it remains to show that $\E[E^n(x)] \gtrsim \E[E^m(x)] \E[E^{m,n}(x)]$. In order to prove this, it suffices to prove that
\begin{align*}
    \E[E_{m+1}^2(x) | E^m(x),E^{m,n}(x)] 1_{\{E^m(x)>0,E^{m,n}(x)>0\}} \asymp 1_{\{E^m(x)>0,E^{m,n}(x)>0\}},
\end{align*}
that is, that with positive probability (which is independent of $m$ and $n$), $\eta^{x_m}|_{[\sigma_m^x,\infty)}$ will merge with $\eta^{x_{m+1}}|_{[0,\sigma_{m+1}^x)}$ before exiting $B(x,\frac{3}{2} \epsilon_{m+1})\setminus B(x,\frac{1}{M}\epsilon_{m+1})$ and not come too close to $\R$ (in the sense of property (ii) of $A_{m+1}^2(x)$). For the remainder of the proof, assume that $E^m(x),E^{m,n}(x)>0$ (if not, we are done). Let $K^1$ and $K^2$ be the closure of the complement of the unbounded connected component of $\UH \setminus (\eta([0,\tau_{\overline{\alpha}_{m+1}}]) \cup \eta^{x_m}([0,\sigma_m^x]))$ and $\UH \setminus( \eta^{x_{m+1}}([0,\hat{\tau}]) \cup \eta^{x_n}([0,\sigma_n^x]))$, respectively, where $\hat{\tau} = \sigma^{x_{m+1}}(B(x,\epsilon_{n+1}))$, and let $K = K^1 \cup K^2$. Let $K^{1,L}$ and $K^{1,R}$ denote the boundaries of $K^1$ to left and right of $\eta(\tau_{\overline{\alpha}_{m+1}})$ and let $K_M^{2,L} = \eta^{x_{m+1}}([0,\hat{\tau}]) \cap \partial K^2$ be the part of $K^2$ that $\eta^{x_m}$ should hit and merge with.

\begin{figure}[ht!]
\centering
\includegraphics[width=155mm]{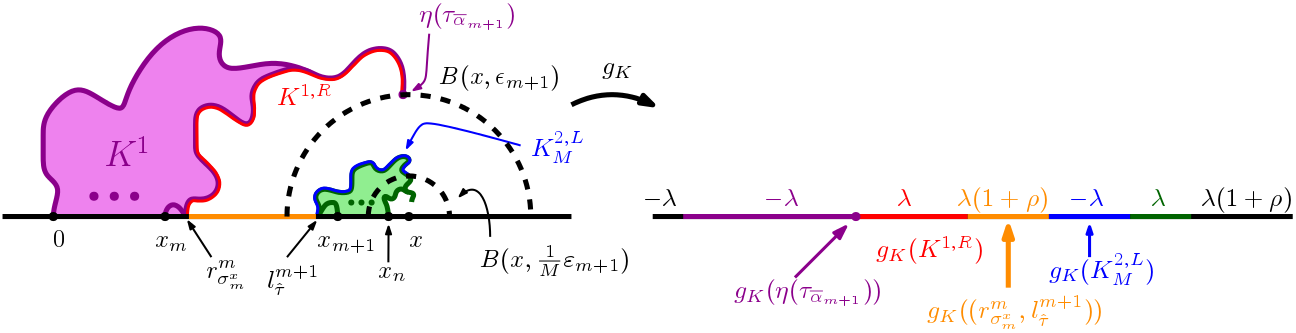}
\caption{$K^1$, shown in purple, and $K^2$, shown in green, are the closures of the complement of the unbounded connected component of $\UH \setminus (\eta([0,\tau_{\overline{\alpha}_{m+1}}]) \cup \eta^{x_m}([0,\sigma_m^x]))$ and $\UH \setminus (\eta^{x_{m+1}}([0,\hat{\tau}])\cup \eta^{x_n}([0,\sigma_n^x]))$, respectively. The boundary data on the left is as in Figure \ref{fig:lem1}. On the right, the images under $g_K$ of the boundary sets of the left are shown, along with their boundary data.  \label{fig:lemma45}}
\end{figure}

Before going on with the proof, we discuss the strategy. Note that the law of the flow line from $\eta(\tau_{\overline{\alpha}_{m+1}}) = \eta^{x_m}(\sigma_m^x)$, given $K^1$ and $K^2$ is that of an $\SLE_\kappa(\rho,-2-\rho,2,\rho)$ process with configuration $(\UH \setminus K,\eta(\tau_{\overline{\alpha}_{m+1}}),(r_{\sigma_m^x}^m,l_{\hat{\tau}}^{m+1},\eta^{x_{m+1}}(\hat{\tau}),r_{\sigma_n^x}^n),\infty)$, where $l_t^{m+1} = \min \{\eta^{x_{m+1}}([0,t]) \cap \R\}$.
The idea of the proof is to map $\UH \setminus K$ to $\UH$, so that the image of the flow line is an $\SLE_\kappa(\rho,-2-\rho,2,\rho)$ process in $\UH$ (see Figure \ref{fig:lemma45}), and use Lemma \ref{IGlem3} to see that the merging, as well as the geometric restriction of property (ii) of $A_{m+1}^2(x)$, occurs with positive probability, independent of $m$. (Note that if the conditions of Lemma \ref{IGlem3} are satisfied, then we can just choose some suitable deterministic curve, such that the properties are satisfied when the flow line follows that curve.) Let $g_K$ denote the mapping-out function of $K = K^1 \cup K^2$ (recall Section \ref{CA}). The image of the flow line from $\eta(\tau_{\overline{\alpha}_{m+1}})$ under $g_K$ is then an $\SLE_\kappa(\rho,-2-\rho,2,\rho)$ with configuration $(\UH, g_K(\eta(\tau_{\overline{\alpha}_{m+1}})),(g_K(r_{\sigma_m^x}^m),g_K(l_{\hat{\tau}}^{m+1}),g_K(\eta^{x_{m+1}}(\hat{\tau})),g_K(r_{\sigma_n^x}^n)),\infty)$, however, the images of the boundary sets under the mapping-out function decrease in length as $m$ increases and thus we need to rescale to be able to use Lemma \ref{IGlem3}. Thus, we want to show that we can find some scaling factor $k$, which will depend on $m$, so that the images of the boundary sets under $\hat{g}_K \coloneqq k g_K$ are of appropriate sizes. More precisely, recall that we want the flow line from $\eta(\tau_{\overline{\alpha}_{m+1}})$ to hit $\eta^{x_{m+1}}([0,\sigma_{m+1,M}^x])$, that is, we want the flow line from $\hat{g}_K(\eta(\tau_{\overline{\alpha}_{m+1}}))$ to hit $\hat{g}_K(K_M^{2,L}) =(\hat{g}_K(l_{\hat{\tau}}^{m+1}),\hat{g}_K(\eta^{x_{m+1}}(\sigma_{m+1,M}^x)))$. Then, to be able to use Lemma \ref{IGlem3}, we need to check that we can fix some $\tilde{\epsilon}>0$, which may depend on $\delta$ and $M$, but is independent of $m$, such that
\begin{align*}
    &\hat{g}_K(r_{\sigma_m^x}^m)-\hat{g}_K(\eta(\tau_{\overline{\alpha}_{m+1}})) \geq \tilde{\epsilon} \\
    &\hat{g}_K(\eta^{x_{m+1}}(\sigma_{m+1,M}^x))-\hat{g}_K(l_{\hat{\tau}}^{m+1}) \geq \tilde{\epsilon}, \textup{ and} \\
    &\hat{g}_K(l_{\hat{\tau}}^{m+1}) -\hat{g}_K(\eta(\tau_{\overline{\alpha}_{m+1}})) \leq \tilde{\epsilon}^{-1}.
\end{align*}
Note that since there is no force point to the left of $\hat{g}_K(\eta(\tau_{\overline{\alpha}_{m+1}}))$, we do not need to care about the force point $x_{2,L}$ of the lemma. Furthermore, while the lemma concerns hitting the interval between two force points, it is still applicable to the interval $(\hat{g}_K(l_{\hat{\tau}}^{m+1}),\hat{g}_K(\eta^{x_{m+1}}(\sigma_{m+1,M}^x)))$, as we can just consider $\hat{g}_K(\eta^{x_{m+1}}(\sigma_{m+1,M}^x))$ as a force point with weight $0$. 

Now, we want to estimate the length of intervals which are images of boundary sets, under $g_K$. Recalling \eqref{eq:harminflength}, it is natural to consider the harmonic measure from infinity of the boundary sets. Rephrasing the above, in terms of harmonic measure from infinity, we need to check that there exists some $\hat{\epsilon}>0$, independent of $m$, such that
\begin{align*}
    &k \omega_\infty(K^{1,R},\UH\setminus K) \geq \hat{\epsilon}, \\ 
    &k \omega_\infty(K_M^{2,L},\UH\setminus K) \geq \hat{\epsilon}, \textup{ and } \\
    &k \omega_\infty(K^{1,R} \cup (r_{\sigma_m^x}^m,l_{\hat{\tau}}^{m+1}), \UH \setminus K) \leq \hat{\epsilon}^{-1},
\end{align*}
if $k$ is chosen properly. We shall show that the three harmonic measures are actually proportional, with proportionality constants independent of $m$, and thus, letting $k = \omega_\infty(K^{1,R},\UH\setminus K)^{-1}$, we are in the setting of Lemma \ref{IGlem3} and the result follows.

We now turn to proving that the above harmonic measures are proportional. Arguing as in the proof of Lemma \ref{lemjust} (i.e., a Brownian motion must first hit the arc of $\partial B(x,\epsilon_{m+1})$ with endpoints $\eta(\tau_{\overline{\alpha}_{m+1}})$ and $x+\epsilon_{m+1}$, from there, the probabilities of hitting the different sets are proportional), we see that 
\begin{align}\label{eq:obvineq}
     \omega_\infty(K^{1,R},\UH\setminus K) \asymp \omega_\infty(K^{1,R},\UH\setminus K^1) \gtrsim \omega_\infty((r_{\sigma_m^x}^m,x],\UH\setminus K^1),
\end{align}
and that
\begin{align}\label{eq:obvineq2}
    \omega_\infty(K_M^{2,L},\UH\setminus K) \asymp \omega_\infty(\eta^{x_{m+1}}([0,\sigma_{m+1}^x]),\UH\setminus K) \asymp \omega_\infty((r_{\sigma_m^x}^m,x],\UH\setminus K^1),
\end{align}
and the implicit constants are independent of $m$. Condition (ii) of $A_m^1(x)$ states that
\begin{align*}
    \omega_\infty(\eta^{x_m,R}([0,\sigma_m^x]) \cup (r_{\sigma_m^x}^m,x], \UH \setminus \eta^{x_m}([0,\sigma_m^x])) \asymp \omega_\infty((r_{\sigma_m^x}^m,x], \UH \setminus \eta^{x_m}([0,\sigma_m^x])),
\end{align*}
and consequently
\begin{align}\label{eq:bypropii}
    \omega_\infty(K^{1,R} \cup (r_{\sigma_m^x}^m,x], \UH \setminus K^1) \asymp \omega_\infty((r_{\sigma_m^x}^m,x],\UH\setminus K^1).
\end{align}
Thus,
\begin{align*}
    \omega_\infty((r_{\sigma_m^x}^m,x],\UH\setminus K^1) \gtrsim \omega_\infty(K^{1,R},\UH\setminus K^1),
\end{align*}
and hence, by \eqref{eq:obvineq},
\begin{align}\label{eq:seqofineqs}
    \omega_\infty(K_M^{2,L},\UH\setminus K) \asymp \omega_\infty((r_{\sigma_m^x}^m,x],\UH\setminus K^1) \asymp \omega_\infty(K^{1,R},\UH\setminus K).
\end{align}
Next, we note that
\begin{align}\label{eq:lastineqs}
    \omega_\infty(K^{1,R} \cup (r_{\sigma_m^x}^m,x], \UH \setminus K^1) &\geq \omega_\infty(K^{1,R} \cup (r_{\sigma_m^x}^m,l_{\hat{\tau}}^{m+1}), \UH \setminus K) \\
    &\geq \omega_\infty(K^{1,R},\UH\setminus K) \nonumber
\end{align}
where the first inequality holds since the left-hand side is the harmonic measure from infinity of a larger set, with fewer obstacles for the Brownian paths. By \eqref{eq:bypropii}, \eqref{eq:seqofineqs} and \eqref{eq:lastineqs},
\begin{align*}
    \omega_\infty(K_M^{2,L},\UH\setminus K) \asymp \omega_\infty(K^{1,R},\UH\setminus K) \asymp \omega_\infty(K^{1,R} \cup (r_{\sigma_m^x}^m,l_{\hat{\tau}}^{m+1}), \UH \setminus K).
\end{align*}
Hence, rescaling by letting $k = \omega_\infty(K^{1,R},\UH\setminus K)^{-1}$, we can use Lemma \ref{IGlem3}, and the proof is done.
\end{proof}

\begin{lem}\label{lem3}
For each $\delta \in (0,\frac{1}{2})$, sufficiently small, there exist a constant $\tilde{c}(\delta)>0$ and a subexponential function $\psi$ such that the for each $x \geq 1$,
\begin{align*}
    \E[E^m(x)] \geq \tilde{c}(\delta)^{m+1} \left(\prod_{j=1}^{m+1}\psi(\alpha_j)^{-|\zeta|}\right) e^{\overline{\alpha}_{m+1}(\zeta\beta-\mu)(1+\rho/2)}.
\end{align*}
\end{lem}
\begin{proof}
By Lemma \ref{lem1},
\begin{align*}
    \E[E_k^1(x)|E^{k-1}(x)] 1_{\{E^{k-1}(x)>0\}} \asymp \E[E_k^1(x)] 1_{\{E^{k-1}(x)>0\}},
\end{align*}
so we need to show that there exist a constant $\tilde{c}(\delta)$ and a subexponential function $\psi$ such that
\begin{align}
    &\E[E_k^1(x)] \geq \tilde{c}(\delta) \psi(\alpha_{k+1})^{-|\zeta|} e^{\alpha_{k+1}(\zeta\beta-\mu)(1+\rho/2)}, \label{lem3,1} \\
    &\E[E_k^2(x)|E^{k-1}(x),E_k^1(x)] 1_{\{E^{k-1}(x)>0,E_k^1(x)>0\}} \asymp 1_{\{E^{k-1}(x)>0,E_k^1(x)>0\}}. \label{lem3,2}
\end{align}
However, (\ref{lem3,2}) follows from the very same argument as Lemma \ref{lem2}, so we will now concern ourselves with (\ref{lem3,1}). Since $1_{A_k^1}$ is $\mathscr{F}_{\sigma_k^x}$-measurable, we have
\begin{align*}
    \E[E_k^1(x)] &= \E\left[1_{A_k^1} I_k^{u,\Lambda,k}\right] = \E\left[1_{A_k^1} \E \left[ \tilde{I}_{\frac{\alpha_{k+1}}{a}+C_k^*(x)}^{u,\Lambda,k} \middle| \mathscr{F}_{\sigma_k^x} \right] \right] \\
    &= \prob\left(A_k^1 \cap \tilde{I}_{\frac{\alpha_{k+1}}{a}+C_k^*(x)}^{u,\Lambda,k}\right).
\end{align*}
As stated above, $\eta^{x_k}$ is an $\SLE_\kappa(2+\rho,-2-\rho;\rho)$ curve with configuration $(\UH,x_k,(0,x_k^-),x_k^+,\infty)$ and by Lemma \ref{IGlem5}, the Radon-Nikodym derivative between the law of $\eta^{x_k}$ and an $\SLE_\kappa(-2-\rho;\rho)$ curve with configuration $(\UH,x_k,x_k^-,x_k^+,\infty)$, both stopped upon exiting $B(x,\frac{1}{2}\epsilon_k)$ is bounded above and below by constants and hence we can (and will) instead consider the latter. Also, we can translate and rescale the process so that we consider an $\SLE_\kappa(-2-\rho;\rho)$ curve, $\hat{\eta}$, started from $0$ and the point that we want the curve to get close to being $1$. Then, since $\dist(x_k,x) = \frac{1}{4} \epsilon_k$, the event $\{ \sigma_k^x < \sigma^{x_k}(\UH \setminus B(x,\frac{1}{2}\epsilon_k)) \}$ turns into the event $\{ \hat{\eta} \text{ hits } B(1,e^{-\alpha_{k+1}}) \text{ before leaving } B(1,2) \}$ and the event $\{ Q_{\sigma_{k,M}^x}^k, Q_{\sigma_k^x}^k \in [\delta,1-\delta] \}$ remains roughly the same. More precisely, let $\hat{Q}_t$ denote the process defined by \eqref{eq:Qharm} (but with $\hat{\eta}$ in place of $\eta$), $\sigma_M = \inf \{ t \geq 0: \dist(\hat{\eta}([0,t]),1) < 4/M \}$ and $\sigma_k = \inf \{ t \geq 0: \dist(\hat{\eta}([0,t]),1) < e^{-\alpha_{k+1}} \}$, then (by translation and scaling invariance of $Q_t^k$), $\{ \hat{Q}_{\sigma_k}, \hat{Q}_{\sigma_M} \in [\delta,1-\delta] \} = \{ Q_{\sigma_{k,M}^x}^k, Q_{\sigma_k^x}^k \in [\delta,1-\delta] \}$. We denote by $(g_t)$ the Loewner chain corresponding to $\hat{\eta}$ and weight the probability measure $\prob$ with the local martingale (recall \eqref{eq:krmtg})
\begin{align*}
    M_t^\zeta(1) = g_t'(1)^\zeta \hat{Q}_t^\mu \delta_t^{-\mu(1+\rho/2)}(g_t(1)-V_t^L)^{\mu(1+\rho/2)},
\end{align*}
and denote the resulting measure by $\prob^*$. Note that it is under $\prob^*$ that we can choose $u$ such that $\tilde{I}_{\frac{\alpha_{k+1}}{a}+C_k^*(x)}^{u,\Lambda,k}$ has probability arbitrarily close to $1$ (in the case with no force point to the left). We note that Lemma \ref{condset} implies that on $A_k^1(x) \cap \tilde{I}_{\frac{\alpha_{k+1}}{a}+C_k^*(x)}^{u,\Lambda,k}$
\begin{align*}
    \psi(\alpha_{k+1})^{-|\zeta|} e^{-\alpha_{k+1} \zeta\beta(1+\rho/2)} \leq g_{\sigma_k}'(x)^\zeta \leq \psi(\alpha_{k+1})^{|\zeta|} e^{-\alpha_{k+1} \zeta\beta(1+\rho/2)},
\end{align*}
for some subexponential function $\psi$. Furthermore, $\hat{Q}_{\sigma_k} \asymp 1$ and by Lemma \ref{deltadist} $\delta_{\sigma_k} \asymp e^{-\alpha_{k+1}}$, that is, $\delta_{\sigma_k}^{-\mu(1+\rho/2)} \asymp e^{\alpha_{k+1}\mu(1+\rho/2)}$. Moreover, we see that $g_{\sigma_k}(1) - V_{\sigma_k}^L \asymp 1$, since it is the harmonic measure from infinity of the left side of $\hat{\eta}$, so that it is upper bounded by $\omega_\infty(B(1,2),\UH)$, which is finite, and lower bounded by $\omega_\infty([0,1/2],\UH)=1/2\pi$. Since $\prob^*(A) = \E[M_{\sigma_k}^\zeta(1)1_A]$ and
\begin{align*}
    \psi(\alpha_{k+1})^{-|\zeta|} e^{-\alpha_{k+1}(\zeta\beta-\mu)(1+\rho/2)} \lesssim M_{\sigma_k}^\zeta(1) \lesssim \psi(\alpha_{k+1})^{|\zeta|} e^{-\alpha_{k+1}(\zeta\beta-\mu)(1+\rho/2)},
\end{align*}
we have that
\begin{align*}
    \psi(\alpha_{k+1})^{-|\zeta|} e^{\alpha_{k+1}(\zeta\beta-\mu)(1+\rho/2)} &\prob^*(A_k^1(x) \cap \tilde{I}_{\frac{\alpha_{k+1}}{a}+C_k^*(x)}^{u,\Lambda,k}) \\
    &\lesssim \prob(A_k^1(x) \cap \tilde{I}_{\frac{\alpha_{k+1}}{a}+C_k^*(x)}^{u,\Lambda,k}) \\
    &\lesssim \psi(\alpha_{k+1})^{|\zeta|} e^{\alpha_{k+1}(\zeta\beta-\mu)(1+\rho/2)} \prob^*(A_k^1(x) \cap \tilde{I}_{\frac{\alpha_{k+1}}{a}+C_k^*(x)}^{u,\Lambda,k}).
\end{align*}

\begin{figure}[ht!]
\centering
\includegraphics[width=140mm]{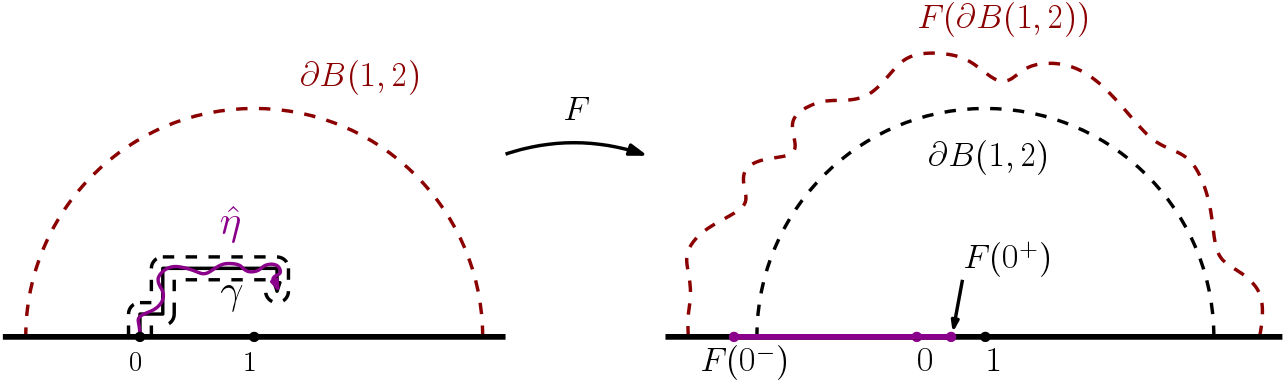}
\caption{We may fix a curve $\gamma$ (black) such that if the curve $\hat{\eta}$ (purple) does not leave the $\tilde{\epsilon}$-neighborhood before coming close to the tip, then after mapping back to $\UH$, the points are as indicated in the picture. \label{fig:OPE}}
\end{figure}

Now we need only show that $\prob^*(A_k^1(x) \cap \tilde{I}_{\frac{\alpha_{k+1}}{a}+C_k^*(x)}^{u,\Lambda,k}) \asymp 1$. Note that under $\prob^*$, $\hat{\eta}$ is an $\SLE_\kappa(-2-\rho;\rho,-\mu\kappa)$ curve with configuration $(\UH,0,0^-,(0^+,1),\infty)$. We shall begin by reducing this to the case of an $\SLE$ process with no force points to the left of $0$. The following procedure is illustrated in Figure \ref{fig:OPE}. Let $\gamma:[0,1] \rightarrow \overline{\UH}$ be a deterministic curve starting at $0$ and remaining in $\UH$ after that, and $\hat{\epsilon} > 0$ be such that if $\hat{\eta}$ comes within distance $\hat{\epsilon}$ of the tip $\gamma(1)$ before exiting the $\hat{\epsilon}$-neighborhood of $\gamma$, then
\begin{align*}
    \dist(1,F(\partial B(1,2) \cap \UH)) \geq 2 \text{ and } F(\min\{\hat{K}_{\hat{\sigma}_1} \cap \R \}) < -2,
\end{align*}
where
\begin{align*}
    F(z) = \frac{\hat{f}_{\hat{\sigma}_1}(z)}{\hat{f}_{\hat{\sigma}_1}(1)},
\end{align*}
$\hat{\sigma}_1 = \inf\{ t \geq 0: \dist(\hat{\eta}([0,t]),\gamma(1)) < \hat{\epsilon}\}$, and $(\hat{f}_t)$ and $(\hat{K}_t)$ are the centered Loewner chain and hulls respectively. Then, the curve $\overline{\eta}(t) = F(\hat{\eta}(\hat{\sigma}_1+t))$ has the law of a time-changed $\SLE_\kappa(-2-\rho;\rho,-\mu\kappa)$ curve with configuration $(\UH,0,x_L,(x_R,1),\infty)$, where $x_L < -2$ and $x_R \in [0^+,1)$. Note that we may choose $\gamma$ and $\hat{\epsilon}$ so that $1-x_R > \hat{\delta}$ for some $\hat{\delta} > 0$ (to get a bound on the constant $C^*$, chosen as remarked after Lemma \ref{time}). By Lemma \ref{IGlem1}, the above happens with positive probability, say $p_0$. By Lemma \ref{IGlem5}, the Radon-Nikodym derivative between the law of $\overline{\eta}$ and the law of a correspondingly time-changed $\SLE_\kappa(\rho,-\mu\kappa)$ curve with force points $(x_R,1)$, is bounded above and below by some constants. Thus we may consider such an $\SLE_\kappa(\rho,-\mu\kappa)$ process. Note that, if $(\hat{g}_t)$ and $(\overline{g}_t)$ are the Loewner chains of $\hat{\eta}$ and $\overline{\eta}$, respectively, then
\begin{align*}
    \hat{g}_{\hat{\sigma}_1+t}'(1) = \frac{1}{\hat{f}_{\hat{\sigma}_1}(1)} \hat{g}_{\hat{\sigma}_1}'(1) \overline{g}_t'(1).
\end{align*}
Thus, we can choose $\Lambda$ to be sufficiently large, so that 
\begin{align*}
    \tilde{I}_{\frac{\alpha_{k+1}}{a}+C_k^*(x)}^{u,\Lambda,k} \supset \{ \hat{\sigma}_1 \leq \hat{\sigma}_2 \} \cap \tilde{I}_{\frac{\alpha_{k+1}}{a}+C_k^*(x)}^{u}(\overline{\eta})
\end{align*}
where $\hat{\sigma}_2 = \inf\{t\geq 0: \dist(\hat{\eta}(t),\gamma([0,1])) > \hat{\epsilon}\}$ and $\tilde{I}_{t}^{u}(\overline{\eta})$ is the event of Proposition \ref{set} for $\overline{\eta}$. We have lower bounded the probability of $\{ \hat{\sigma}_1 \leq \hat{\sigma}_2 \}$ and next, we prove that $\{ \overline{\eta} \textup{ hits } 1 \textup{ before } \partial B(1,2) \} \cap \tilde{I}_{\frac{\alpha_{k+1}}{a}+C_k^*(x)}^{u}(\overline{\eta}) \cap \{ \hat{Q}_{\sigma_k}, \hat{Q}_{\sigma_M} \in [\delta,1-\delta] \}$ has positive probability, which completes the proof of the lemma. (Note that we do not need $\overline{\eta}$ to hit $1$ before exiting $B(1,2)$, but rather that it hits some small set, separating $1$ from $\infty$, before exiting $B(1,2)$, which is of course weaker.)

We begin by lower bounding $\prob^*(\overline{\eta} \textup{ hits } 1 \textup{ before } \partial B(1,2))$. We now make a conformal coordinate change with the Möbius transformation
\begin{align*}
    \varphi(z) = \frac{z}{1-z}.
\end{align*}
The image of an $\SLE_\kappa(\rho,-\mu\kappa)$ curve in $\UH$ from $0$ to $\infty$ is an $\SLE_\kappa(\rho_L;\rho)$ curve in $\UH$ from $0$ to $\infty$, with force points $\varphi(\infty) = -1$ and $\varphi(x_R) = \frac{x_R}{1-x_R}$, where $\rho_L = \kappa-6-(\rho-\mu\kappa) = \kappa(1+\mu)-6-\rho$ (see \cite{SW05}). Furthermore, $1$ is mapped to $\infty$ and $\varphi(\partial B(1,2)) = \partial B(-1,\frac{1}{2})$ and thus the event of hitting $1$ before exiting $B(1,2)$ turns into hitting the event that $\varphi(\overline{\eta})$ does not hit $B(-1,\frac{1}{2})$. We have that $\kappa \leq 4$ and $\rho_L > \frac{\kappa}{2}-2$, so the probability of $\varphi(\overline{\eta})$ avoiding $B(-1,\frac{1}{2})$ is positive.

Next, choosing $\delta$ to be sufficiently small and $M$ to be large enough, we can (by Corollary \ref{Qprop}) guarantee that the $\prob^*$-probability of the event $\{ \hat{Q}_{\sigma_k}, \hat{Q}_{\sigma_M} \in [\delta,1-\delta] \}$ is as close to $1$ as we want. This holds, since as $\hat{\eta}$ comes closer to $1$, a time change $\hat{Q}_{\tilde{t}(s)}$ of $\hat{Q}$ (as in Section \ref{weighted}) converges to its invariant distribution $\tilde{X}_s$, so by choosing $M$ large, we can make sure that $\hat{Q}_{\tilde{t}(s)}$ will be as close to $\tilde{X}_s$ as necessary, in the sense of \eqref{invconvQ}. Moreover, letting $\delta$ be sufficiently small, we have that, for each $s$, $\prob^*(\tilde{X}_s \in [\delta,1-\delta])$ is sufficiently close to $1$ and hence the same follows for $\hat{Q}_{\sigma_k}$ and $\hat{Q}_{\sigma_M}$.

By then choosing $u>0$ sufficiently large, we have that the $\prob^*$-probability of $\tilde{I}_{\frac{k}{a}+C_k^*(x)}^{u,k}$ is arbitrarily close to $1$, and hence that $\prob^*(A_k^1 \cap \tilde{I}_{\frac{k}{a}+C_k^*(x)}^{u,k}) \gtrsim 1$. Thus \eqref{lem3,1} is proven, which gives the result.
\end{proof}
\begin{rmk}
In the above proof, we actually proved an upper bound as well:
\begin{align}\label{eq:IGUB}
    \E[E^m(x)] \leq \tilde{C}(\delta)^{m+1} \left( \prod_{j=1}^{m+1} \psi(\alpha_j)^{|\zeta|} \right) e^{\overline{\alpha}_{m+1}(\zeta\beta-\mu)(1+\rho/2)}.
\end{align}
\end{rmk}

\begin{proof}[Proof of Proposition \ref{2PEIG}]
It holds that
\begin{align*}
    \E[E^n(x) E^n(y)] &\leq \E[E^{m-1}(x) E^{m+1,n}(x) E^{m+1,n}(y)] \\
    &\lesssim \E[E^{m-1}(x)]\E[E^{m+1,n}(x)]\E[E^{m+1,n}(y)] \\
    &= \frac{\E[E^{m-1}(x)]}{\E[E^{m+1}(x)]\E[E^{m+1}(y)]} \\
    & \quad \times \E[E^{m+1}(x)]\E[E^{m+1,n}(x)] \E[E^{m+1}(y)]\E[E^{m+1,n}(y)] \\
    &\lesssim \frac{\E[E^{m-1}(x)]}{\E[E^{m+1}(x)]\E[E^{m+1}(y)]} \E[E^n(x)] \E[E^n(y)] \\
    &\leq \frac{\tilde{C}(\delta)^m \left(\prod_{j=1}^m \psi(\alpha_j)^{|\zeta|}\right)e^{\overline{\alpha}_m (\zeta\beta-\mu)(1+\rho/2)}}{\tilde{c}(\delta)^{2m+4} \left(\prod_{j=1}^{m+2} \psi(\alpha_j)^{-|\zeta|}\right) e^{\overline{\alpha}_{m+2} (\zeta\beta-\mu)(1+\rho/2)}} \E[E^n(x)] \E[E^n(y)] \\
    &\leq c(\delta)^{m+2} \left(\prod_{j=1}^{m+2} \psi(\alpha_j)^{3|\zeta|}\right) e^{-(\overline{\alpha}_{m+2} + \alpha_{m+1} + \alpha_{m+2})(\zeta\beta-\mu)(1+\rho/2)} \E[E^n(x)] \E[E^n(y)] \\
    &= c(\delta)^{m+2} \left(\prod_{j=1}^{m+2} \psi(\alpha_j)^{3|\zeta|}\right) \epsilon_{m+2}^{(1+o_m(1))(\zeta\beta-\mu)(1+\rho/2)} \E[E^n(x)] \E[E^n(y)]
\end{align*}
where we used Lemma \ref{lem1} in the second inequality, Lemma \ref{lem2} in the fourth and Lemma \ref{lem3} and \eqref{eq:IGUB} in the fifth. Thus, we can find a subpower function $\tilde{\Psi}_\delta$ as in the statement of the proposition, and we are done.
\end{proof}

\section{Dimension spectrum}\label{secdim}
In this section, we compute the almost sure Hausdorff dimension of the random sets (\ref{eq:Vbeta*}). We will, however, compute the almost sure dimension of the sets
\begin{align}
V_\beta^* = \left\{ x >0: \lim_{n \rightarrow \infty} \frac{1}{n} \log g_{\tau_n}'(x) = -\beta(1+\rho/2), \ \tau_n = \tau_n(x) < \infty \ \forall n>0 \right\}
\end{align}
and note that this is sufficient, due to the monotonicity of $t \mapsto g_t'$. The theorem that we prove in this section is the following.
\begin{thm}
\label{mainresult*}
Let $\kappa > 0$, $\rho \in ((-2)\vee(\frac{\kappa}{2}-4), \frac{\kappa}{2}-2)$, $x_R = 0^+$ and write $a = 2/\kappa$. Define 
\begin{align*}
    d^*(\beta) \coloneqq 1 - \frac{a\beta}{2}\left( \frac{1-a\rho}{2a} - \frac{1+2\beta}{\beta} \right)^2\left(1+\frac{\rho}{2}\right)
\end{align*}
and let $\beta_-^* = \inf \{\beta:d^*(\beta)>0\}$ and $\beta_+^* = \sup \{\beta:d^*(\beta)>0\}$. Then, almost surely, if $\kappa\in(0,4]$
\begin{align*}
    \textup{dim}_H V_\beta^* = d^*(\beta) \text{ for } \beta \in [\beta_-^*,\beta_+^*].
\end{align*}
\end{thm}
With this theorem at hand, noting that $V_\beta = V_{\beta/(1+\rho/2)}^*$ gives the dimension $\textup{dim}_H V_\beta = \newline  \textup{dim}_H V_{\beta/(1+\rho/2)}^* = d^*(\beta/(1+\rho/2))$, we immediately get Theorem \ref{mainresult}.

We start by showing that $\text{dim}_H V_\beta^*$ is an almost surely constant quantity. The proof of this is contained in \cite{ABV16}, but we repeat the proof here for completeness.
\begin{lem}
\label{dimconst}
Let $x_R = 0^+$. For each $\beta$, $\textup{dim}_H V_\beta^*$ is almost surely constant.
\end{lem}
\begin{proof}
Let $x>0$ and write $S_x = V_\beta^* \cap (0,x)$. Since $x_R = 0^+$, the $\SLE_\kappa(\rho)$ process is scaling invariant and hence, the law of $S_x$ is identical to the law of $xS_1$. However, since $\text{dim}_H xS_1 = \text{dim}_H S_1$ (due to the invariance of the Hausdorff dimension under linear scaling), we see that the law of $\text{dim}_H S_x$ is not depending on $x$. The sets $S_x$ are decreasing as $x \rightarrow 0^+$, and hence $\text{dim}_H S_x$ has an almost sure limit (as $x \rightarrow 0^+$) which is measurable with respect to $\mathscr{F}_{0^+}$, since $S_x$ is measurable with respect to $\mathscr{F}_{T_x}$. By Blumenthal's 0-1 law, the limit must be constant and the same for every $x>0$.
\end{proof}
In the following two sections, we prove the upper and lower bounds on the dimension, Theorem \ref{UB} and Theorem \ref{LBIG}, and together they imply Theorem \ref{mainresult}.
\subsection{Upper bound}
We define the random sets
\begin{align*}
    \overline{V}_\beta = \left\{ x \in \R_+: g_{\tau_n}'(x) \geq e^{-\beta(1+\rho/2) n}, \ \tau_n(x) < \infty \ \text{i.o. in} \ n \right\}, \\
    \underline{V}_\beta = \left\{ x \in \R_+: g_{\tau_n}'(x) \leq e^{-\beta(1+\rho/2) n}, \ \tau_n(x) < \infty \ \text{i.o. in} \ n \right\},
\end{align*}
and note that for $\beta_1 < \beta < \beta_2$, we have $V_\beta^* \subset \underline{V}_{\beta_1}$ and $V_\beta^* \subset \overline{V}_{\beta_2}$. Thus, in this subsection, we will find a suitable cover of the above sets, and bound the Hausdorff measure of the above sets using the Minkowski content of the covers. Using this, we prove the following theorem, which gives the upper bound on the dimension. We let $d^*(\beta) = 1 + (\zeta \beta - \mu)(1+\rho/2)$ and write $\beta_-^*$ and $\beta_+^*$ for its left and right zero, respectively. For $\beta_0^* = \beta(0) = \frac{2a}{4a-1+a\rho}$ (where $\beta(0)$ means $\beta(\zeta)$, evaluated at $\zeta=0$), we have $(d^*)'(\beta_0^*)=0$, and hence $d^*(\beta)$ it is increasing for $\beta \in [\beta_-^*,\beta_0^*]$ and decreasing for $\beta \in [\beta_0^*, \beta_+^*]$.
\begin{thm}
\label{UB}
Let $\kappa > 0$, $\rho \in ((-2)\vee(\frac{\kappa}{2}-4),\frac{\kappa}{2}-2)$ and $x_R = 0^+$. Then, the following hold:
\begin{enumerate}[(i)]
\item if $\beta \in [ \beta_-^*, \beta_0^*)$, then $\textup{dim}_H \overline{V}_\beta \leq d^*(\beta)$ almost surely,
\item if $\beta < \beta_-^*$, then $\overline{V}_\beta = \emptyset$ almost surely,
\item if $\beta \in ( \beta_0^*, \beta_+^*]$, then $\textup{dim}_H \underline{V}_\beta \leq d^*(\beta)$ almost surely,
\item if $\beta > \beta_+^*$, then $\underline{V}_\beta = \emptyset$ almost surely.
\end{enumerate}
Together, they imply that if $\beta \in [\beta_-^*,\beta_+^*]$, then
\begin{align*}
    \textup{dim}_H V_\beta^* \leq d^*(\beta),
\end{align*}
almost surely.
\end{thm}
Now, we will construct the cover. It is sufficient to prove the above theorem for the sets intersected with every closed subinterval of $\R_+$. As in \cite{ABV16}, we will do this for the set $[1,2]$, but it will be clear that the very same construction works for any other closed set. We start by constructing the cover for $\overline{V}_\beta \cap [1,2]$.
For every $n \geq 1$, let 
\begin{align*}
J_{j,n} = [1+je^{-n}/2, 1+(j+1)e^{-n}/2), \ j = 0,1,\dots, \lceil 2e^n \rceil -1.
\end{align*}
Then every interval has length $e^{-n}/2$. We denote by $x_{j,n}$ the midpoint of the interval $J_{j,n}$ and write $\mathscr{J}_n \coloneqq \{J_{j,n} \}$. By distortion estimates, we have that there is some constant $c' > 0$, such that if 
\begin{align} \label{eq:lbg'}
    g_{\tau_n(x)}'(x) \geq e^{-\beta(1+\rho/2) n} \ \text{for some} \ x \in J_{j,n},
\end{align}
then,
\begin{align} \label{eq:lbg'2}
    g_{\tau_n(x)}'(x_{j,n}) \geq c'e^{-\beta(1+\rho/2) n}.
\end{align}
We also have that
\begin{align*}
    \tau_{n-2}(x_{j,n}) \leq \inf_{x \in J_{j,n}} \tau_n(x),
\end{align*}
since the curve must hit the ball of radius $e^{-(n-2)}$, centered at $x_{j,n}$ before it hits the ball of radius $e^{-n}$, centered at any point $x$ in $J_{j,n}$, as the former ball contains the latter for any $x \in J_{j,n}$. Combining this with the fact that $t \mapsto g_t(x)$ is decreasing for every fixed $x$, (\ref{eq:lbg'2}) and writing $c = c'e^{-2\beta(1+\rho/2)}$, shows that (\ref{eq:lbg'}) implies that
\begin{align*}
    g_{\tau_{n-2}(x_{j,n})}'(x_{j,n}) \geq ce^{-\beta(1+\rho/2)(n-2)}.
\end{align*}
We let $\mathscr{I}_n^-(\beta) = \left\{ j \in \{ 0,1,..., \lceil 2e^n \rceil -1 \} : g_{\tau_{n-2}(x_{j,n})}'(x_{j,n}) \geq ce^{-\beta(1+\rho/2)(n-2)} \right\}$, and define $J_{n,-}(\beta)$ by
\begin{align*}
    J_{n,-}(\beta) = \bigcup_{j \in \mathscr{I}_n^-(\beta)} J_{j,n}.
\end{align*}
Then $\left\{ x \in [1,2]: g_{\tau_n(x)}'(x) \geq e^{-\beta(1+\rho/2) n} \right\} \subset J_{n,-}(\beta)$, and thus, for every positive integer $m$,
\begin{align*}
    \overline{V}_\beta \cap [1,2] \subset \bigcup_{n \geq m} J_{n,-}(\beta),
\end{align*}
i.e., for every positive integer $m$, $\bigcup_{n \geq m} J_{n,-}(\beta)$ is a cover of $\overline{V}_\beta \cap [1,2]$. We write
\begin{align} \label{eq:Nnu}
    \overline{\mathcal{N}}_n(\beta) = \sum_{j=0}^{\lceil 2e^n \rceil -1} 1\left\{ g_{\tau_{(n-2)}(x_{j,n})}'(x_{j,n}) \geq ce^{-\beta(1+\rho/2)(n-2)} \right\},
\end{align}
that is, $\overline{\mathcal{N}}_n(\beta)$ is the number of intervals $J_{j,n}$ that make up $J_{n,-}(\beta)$.
\begin{lem}
\label{Nu}
Let $\kappa > 0$, $\rho \in ((-2)\vee(\frac{\kappa}{2}-4),\frac{\kappa}{2}-2)$ and $\zeta > 0$, that is, $\beta < \beta_0^* = \frac{2a}{4a-1+a\rho}$. Then,
\begin{align*}
    \E[\overline{\mathcal{N}}_n(\beta)] \lesssim e^{n(1+(\zeta\beta - \mu)(1+\rho/2))}.
\end{align*}
\end{lem}
\begin{proof}
By (\ref{eq:Nnu}), we get
\begin{align*}
    \E\left[\overline{\mathcal{N}}_n(\beta)\right] &= \sum_{j=0}^{\lceil 2e^n \rceil -1} \prob \left( g'_{\tau_{n-2}}(x_{j,n}) \geq c e^{-\beta(1+\rho/2)(n-2)} \right) \\
    &= \sum_{j=0}^{\lceil 2e^n \rceil -1} \prob \left( g'_{\tau_{n-2}}(x_{j,n})^\zeta \geq c^\zeta e^{-\zeta \beta(1+\rho/2) (n-2)} ; \tau_{n-2} < \infty \right) \\
    &\lesssim \sum_{j=0}^{\lceil 2e^n \rceil -1} e^{\zeta \beta(1+\rho/2) (n-2)} \E \left[ g_{\tau_{n-2}}'(x_{j,n})^\zeta 1\{\tau_{n-2} < \infty \} \right] \\
    &\lesssim e^{n(1+(\zeta \beta - \mu)(1+\rho/2))},
\end{align*}
using Chebyshev's inequality in the fourth row and Corollary \ref{1PE} in the last.
\end{proof}
Next, we construct the cover for $\underline{V}_\beta \cap [1,2]$. Let $x \in \underline{V}_\beta \cap [1,2]$ and $n > 0$ be such that $g_{\tau_n}'(x) \leq e^{-\beta(1+\rho/2) n}$. By Lemma \ref{time}, there is a constant $C^*$, such that
\begin{align*}
    g_{\tilde{t}(s_n)}'(x) = \tilde{g}_{s_n}'(x) \leq e^{-\beta(1+\rho/2) n},
\end{align*}
where $s_n = \frac{n}{a} + C^*$ (here, $C^*$ can be chosen so that the above holds for every $x \in [1,2]$). By the distortion principle, there is a smallest nonnegative integer $k$ such that, if $J$ is the unique interval in $\mathscr{J}_{n+k}$ such that $x \in J$, then for every $z \in J$, 
\begin{align*}
\tilde{g}_{s_n,x}'(z) \asymp \tilde{g}_{s_n}'(x),
\end{align*}
where the second subscript denotes the point which the time change $\tilde{t}(s)$ is made with respect to (if no second subscript is written out, the time change corresponds to the point in which we evaluate the function). Let $x_J$ denote the midpoint of $J$. Then $\tilde{g}_{s_n,x}(x_J) \lesssim e^{-\beta(1+\rho/2) n}$. Since $\dist(x,x_J) \leq e^{-(n+k)}/4$, we have for geometric reasons and by Lemma \ref{time}, that
\begin{align*}
    \tilde{t}_x \left(\frac{n}{a} + C^* \right) \leq \tau_{n+2aC^*}(x) \leq \tau_{n+2aC^*+1}(x_J) \leq \tilde{t}_{x_J} \left( \frac{n+1}{a} + 3C^* \right), 
\end{align*}
that is, there is a constant $c_1$ such that
\begin{align*}
    \tilde{t}_x \left(s_n \right) \leq \tilde{t}_{x_J} \left( \frac{n}{a} + c_1 \right) = \tilde{t}_{x_J}(s_n').
\end{align*}
Therefore, there is a constant $c_2$ such that $\tilde{g}_{s_n'}'(x_J) \leq c_2 e^{-\beta(1+\rho/2) n}$. The constants above can be chosen to be universal. Let us recap what we have done above; we concluded that there are universal constants $k$, $c_1$, $c_2$ such that every $x \in \underline{V}_\beta$ is contained in an interval $J$ in $\mathscr{J}_{n+k}$ and $\tilde{g}_{\frac{n}{a} + c_1}'(x_J) \leq c_2 e^{-\beta(1+\rho/2) n}$, where $x_J$ is the midpoint of $J$. Therefore, choosing universal constants $C_1$ and $C_2$, we have that if
\begin{align*}
    J_{n,+}(\beta) = \bigcup_{j \in \mathscr{I}_n^+(\beta)} J_{j,n},
\end{align*}
where $\mathscr{I}_n^+(\beta) = \left\{ j \in \{ 0,1,..., \lceil 2e^n \rceil -1 \} : \tilde{g}_{\frac{n}{a} + C_1}'(x_{j,n}) \leq C_2e^{-\beta(1+\rho/2) n} \right\}$, then
\begin{align*}
    \underline{V}_\beta \subset \bigcup_{n \geq m} J_{n,+}(\beta)
\end{align*}
for every $m$. We let $\underline{\mathcal{N}}_n(\beta)$ denote the number of intervals $J_{j,n}$ that make up $J_{n,+}(\beta)$, i.e., 
\begin{align*}
    \underline{\mathcal{N}}_n(\beta) = \sum_{j=0}^{\lceil 2e^n \rceil -1} 1\left\{ \tilde{g}_{\frac{n}{a} + C_1}'(x_{j,n}) \leq C_2 e^{-\beta(1+\rho/2) n} \right\}.
\end{align*}
\begin{lem}
\label{Nl}
Let $\kappa > 0$, $\rho \in ((-2)\vee(\frac{\kappa}{2}-4),\frac{\kappa}{2}-2)$ and $\zeta < 0$, that is, $\beta > \beta_0^* = \frac{2a}{4a-1+a\rho}$. Then,
\begin{align*}
    \E\left[\underline{\mathcal{N}}_n(\beta)\right] \lesssim e^{n(1+(\zeta\beta - \mu)(1+\rho/2))}.
\end{align*}
\end{lem}
\begin{proof}
Applying Chebyshev's inequality and Proposition \ref{1PEts} in the same way as in the previous lemma gives the result.
\end{proof}
With these, we will now prove Theorem \ref{UB}.
\begin{proof}[Proof of Theorem \ref{UB}]
We begin with \textit{(i)}. Lemma \ref{Nu} implies, for $s > d^*(\beta)$ and $n$ sufficiently large, that
\begin{align*}
    \E\left[\overline{\mathcal{N}}_n(\beta)\right] \lesssim e^{ns}.
\end{align*}
Since $\overline{V}_\beta \cap [1,2] \subset \cup_{n \geq m} J_{n,-}(\beta)$ for every $m$, the $t$-dimensional Hausdorff measure of $\overline{V}_\beta \cap [1,2]$, $\mathcal{H}^t(\overline{V}_\beta \cap [1,2])$, is bounded by a constant times the $t$-dimensional lower Minkowski content of \newline
$\cup_{n \geq m} J_{n,-}(\beta)$ (see Section 5.5 of \cite{Mat95}), which implies that
\begin{align*}
    \E\left[\mathcal{H}^t(\overline{V}_\beta \cap [1,2])\right] \leq \lim_{m \rightarrow \infty} c \sum_{n=m}^\infty \E\left[\overline{\mathcal{N}}_n(\beta)\right] e^{-nt} = 0
\end{align*}
for $t>s$. Since such a construction works for every interval, we thus have that $\text{dim}_H\overline{V}_\beta \leq d^*(\beta)$ almost surely. Similarly, using Lemma \ref{Nl}, $\text{dim}_H \underline{V}_\beta \leq d^*(\beta)$ almost surely.
For \textit{(ii)}, note that $1\{ \overline{V}_\beta \cap [1,2] \} \leq \sum_{n=m}^\infty \overline{\mathcal{N}}_n(\beta)$, and thus
\begin{align*}
    \prob \left( \overline{V}_\beta \cap [1,2] \neq \emptyset \right) \leq \lim_{m \rightarrow \infty} \sum_{n=m}^\infty \E\left[\overline{\mathcal{N}}_n(\beta)\right] = 0,
\end{align*}
since $\beta < \beta_-^*$ implies that $d^*(\beta) = 1+ (\zeta \beta - \mu)(1+\rho/2) < 0$, so $\overline{V}_\beta = \emptyset$ for $\beta < \beta_-^*$. The same argument shows that $\underline{V}_\beta = \emptyset$ almost surely for $\beta > \beta_+^*$.

What is left, is to use \textit{(i)} and \textit{(iii)} to prove that $\text{dim}_H V_\beta^* \leq d^*(\beta)$ almost surely for $\beta \in [\beta_-^*,\beta_+^*]$. First, let $\beta \in [\beta_-^*,\beta_0^*)$ and $\epsilon > 0$ be such that $\beta + \epsilon \in (\beta_-^*, \beta_0^*)$. Since $V_\beta^* \subset \overline{V}_\beta$, \textit{(i)} gives that $\text{dim}_H V_\beta^* \leq \text{dim}_H \overline{V}_\beta \leq d^*(\beta+\epsilon)$. By then letting $\epsilon \rightarrow 0^+$ and the fact that $d^*(\beta)$ is increasing on $[\beta_-^*,\beta_0^*]$ gives the upper bound for chosen $\beta$. In the same way, letting $\beta \in (\beta_0^*,\beta_+^*]$ and $\epsilon > 0$ be such that $\beta - \epsilon \in (\beta_0^*,\beta_+^*)$, then since $V_\beta^* \subset \underline{V}_{\beta-\epsilon}$, \textit{(iii)} implies that $\text{dim}_H V_\beta^* \leq \text{dim}_H \underline{V}_\beta \leq d^*(\beta-\epsilon)$. Again, letting $\epsilon \rightarrow 0^+$ and noting that $d^*(\beta)$ is decreasing on $[\beta_0^*,\beta_+^*]$ gives the desired upper bound. Finally, we comment on the case $\beta = \beta_0^*$. In this case, $\text{dim}_H V_{\beta_0^*}^*$ is trivially bounded by $d(\beta_0^*)$, as that is equal to the dimension of the intersection of the curve with $\R_+$, a set which clearly contains $V_{\beta_0^*}^*$. While this is proven previously, we remark that the upper bound follows immediately from Corollary \ref{1PE} with $\zeta = 0$, and a less involved covering argument than the above. Hence, the proof of the upper bound on the dimension is done.
\end{proof}

\subsection{Lower bound}
We shall prove the lower bound using Frostman's lemma, that is, we let $E_s(\nu)$ be the $s$-dimensional energy of the measure $\nu$, i.e.,
\begin{align*}
    E_s(\nu) = \iint \frac{d\nu(x) d\nu(y)}{|x-y|^s}.
\end{align*}
We then construct a Frostman measure on $V_\beta^*$ and show that it has finite $s$-dimensional energy for every $s < d^* = d^*(\beta)$, which implies that the $s$-dimensional Hausdorff measure of $V_\beta^*$ is infinite (see Theorem 8.9 of \cite{Mat95}), and thus that the Hausdorff dimension must be greater than or equal to $s$. Just like in the previous section, we do this for $V_\beta^*$ intersected with the interval $[1,2]$, but again it will be clear that this can be done for any closed interval to the right of $0$. In the following, we will construct a family of Frostman measures and show that it gives the correct lower bound on the dimension of $V_\beta^*$.

\begin{thm}
\label{LBIG}
Let $\kappa \in (0,4]$, $\rho \in (-2,\frac{\kappa}{2}-2)$ and $x_R = 0^+$. Then, for every $\varsigma >0$,
\begin{align*}
    \prob(\textup{dim}_H V_\beta^* \geq d^*(\beta)-\varsigma ) = 1.
\end{align*}
\end{thm}
\begin{proof}
Fix $\kappa \in (0,4)$, $\rho \in (-2,\frac{\kappa}{2}-2)$ and $\delta \in (0,\frac{1}{2})$ small and $u,\Lambda,M>0$ large enough for Proposition \ref{2PEIG} to hold. We fix $n \in \N$, divide $[1,2]$ into $\epsilon_n^{-1}$ intervals of length $\epsilon_n$, and let $x_{j,n} = 1+(j-\frac{1}{2})\epsilon_n$ be the midpoint of the $j$th of these intervals. Let $\mathscr{D}_n=\{x_{j,n}:j=1,...,\epsilon_n^{-1}\}$, let $\mathscr{C}_n = \{ x \in \mathscr{D}_n: E^n(x) > 0 \}$ and let $J_n(x) = [x-\frac{\epsilon_n}{2},x+\frac{\epsilon_n}{2}]$. Then,
\begin{align*}
    \mathscr{C} = \bigcap_{k \geq 1} \overline{\bigcup_{n \geq k} \bigcup_{x \in \mathscr{C}_n} J_n(x)} \subseteq V_\beta^*.
\end{align*}
For each $n \geq 1$, we define the measure $\nu_n$ by
\begin{align*}
    \nu_n(A) = \int_A \sum_{x \in \mathscr{D}_n} \frac{E^n(x)}{\E[E^n(x)]} 1_{J_n(x)}(t) dt,
\end{align*}
for Borel sets $A \subset [1,2]$. We want to take a subsequential limit of the sequence of measures $(\nu_n)$, which we will prove converges to the Frostman measure on $V_\beta^*$. To see that this limit exists, we need that the event on which we want to take the subsequential limit has positive probability and that the support of the limit is contained in $V_\beta^*$. That the support of the limit is contained in $V_\beta^*$ is obvious by construction, so we turn to proving that the event has positive probability. Clearly $\E[\nu_n([1,2])] = 1$. We need to show that there is some constant $c>0$ such that $\E[\nu_n([1,2])^2]<c$. Then, by the Cauchy-Schwarz inequality,
\begin{align*}
    \prob(\nu_n([1,2])>0) \geq \frac{\E[\nu_n([1,2])]^2}{\E[\nu_n([1,2])^2]} \geq \frac{1}{c},
\end{align*}
which implies that the event on which we want to take a subsequence has positive probability. We have that
\begin{align*}
    \E[\nu_n([1,2])^2] = \epsilon_n^2 \sum_{x,y \in \mathscr{D}_n} \frac{\E[E^n(x)E^n(y)]}{\E[E^n(x)]\E[E^n(y)]},
\end{align*}
and we will bound the diagonal and the off-diagonal parts separately below. For the diagonal terms, we have (since $E^n(x)^2 \leq E^n(x)$)
\begin{align*}
    \epsilon_n^2 \sum_{j=1}^{\epsilon_n^{-1}} \frac{\E[E^n(x_{j,n})^2]}{\E[E^n(x_{j,n})]^2} &\leq \epsilon_n^2 \sum_{j=1}^{\epsilon_n^{-1}} \frac{1}{\E[E^n(x_{j,n})]} \lesssim \left(\prod_{j=1}^{n+1} \psi(\alpha_j)^{|\zeta|}\right) \epsilon_n^{d^*+o_n(1)} \lesssim 1,
\end{align*}
for large enough $\alpha_0$ and $n$, since $\psi$ is a subpower function and $o_n(1)$ tends to $0$ as $n \rightarrow \infty$. For the off-diagonal terms, we have, again for large enough $\alpha_0$ and $n$, by Proposition \ref{2PEIG},
\begin{align*}
    \epsilon_n^2 \sum_{j \neq k} &\frac{\E[E^n(x_{j,n})E^n(x_{k,n})]}{\E[E^n(x_{j,n})] \E[E^n(x_{k,n})]} \lesssim \epsilon_n^2 \sum_{j \neq k} \Psi_\delta\left(1/|x_{j,n}-x_{k,n}|\right) |x_{j,n}-x_{k,n}|^{d^*-1} \lesssim 1,
\end{align*}
since $\Psi_\delta$ is a subpower function. Thus, $\E[\nu_n([1,2])]$ is finite, and $\prob(\nu_n([1,2])>0) > 0$. What is left to do is to show that $E_{d^*-\varsigma}(\nu_n)$ is almost surely finite for each $n$ and every $\varsigma > 0$. To do this, it suffices to bound $\E[E_{d^*-\varsigma}(\nu_n)]$ for each $n$ and every $\varsigma > 0$. We have that
\begin{align*}
    \E[E_{d^*-\varsigma}(\nu_n)] = \sum_{x,y \in \mathscr{D}_n} \frac{\E[E^n(x) E^n(y)]}{\E[E^n(x)] \E[E^n(y)]} \iint_{J_n(x) \times J_n(y)} \frac{1}{|u-v|^{d^*-\varsigma}} du dv.
\end{align*}
In order to bound $\E[E_{d^*-\varsigma}(\nu_n)]$ we will use the following:
\begin{align}
    &\iint_{J_n(x_{j,n}) \times J_n(x_{j,n})} \frac{du dv}{|u-v|^s} = \frac{2}{(2-s)(1-s)}\epsilon_n^{2-s}, \label{eq:jj} \\
    &\iint_{J_n(x_{j,n}) \times J_n(x_{k,n})} \frac{du dv}{|u-v|^{s}} \lesssim \epsilon_n^2 \frac{1}{|x_{j,n}-x_{k,n}|^s}. \label{eq:jk}
\end{align}
We first consider the diagonal terms. By \eqref{eq:jj},
\begin{align*}
    &\sum_{j=1}^{\epsilon_n^{-1}} \frac{1}{\E[E^n(x_{j,n})]} \iint_{J_n(x_{j,n}) \times J_n(x_{j,n})} \frac{1}{|u-v|^{d^*-\varsigma}} du dv \\
    &= C \sum_{j=1}^{\epsilon_n^{-1}} \frac{1}{\E[E^n(x_{j,n})]} \epsilon_n^{(2-d^*+\varsigma)} \lesssim \left( \prod_{j=1}^{n+1} \psi(\alpha_j)^{|\zeta|} \right) \epsilon_n^{\varsigma + o_n(1)}.
\end{align*}
Clearly, the sum is uniformly bounded in $n$ for large enough $\alpha_0$. For the off-diagonal terms we have, by Proposition \ref{2PEIG} and \eqref{eq:jk}, that
\begin{align*}
    &\sum_{j \neq k} \frac{\E[E^n(x_{j,n}) E^n(x_{k,n})]}{\E[E^n(x_{j,n})] \E[E^n(x_{k,n})]} \iint_{J_n(x_{j,n}) \times J_n(x_{k,n})} \frac{1}{|u-v|^{d^*-\varsigma}} du dv \\
    &\lesssim \epsilon_n^2 \sum_{j \neq k} \Psi_\delta\left( 1/|x_{j,n}-x_{k,n}| \right) |x_{j,n}-x_{k,n}|^{\varsigma - 1} \lesssim 1.
\end{align*}
The implicit constants do not depend on $n$, and hence we have that for every $\varsigma>0$, the limiting measure $\nu$ satisfies $E_{d^*-\varsigma}(\nu)<\infty$ on an event of positive probability. By Lemma \ref{dimconst}, the Hausdorff dimension is almost surely constant, so a lower bound with positive probability is an almost sure lower bound. Thus, the proof is done.
\end{proof}

Now Theorem \ref{mainresult*} follows from Theorem \ref{UB} and Theorem \ref{LBIG} and thus, as remarked, Theorem \ref{mainresult} is proven.

\appendix

\section{The diffusion $\tilde{Q}_s$}
In this appendix we will state and prove some of the main properties of the diffusion $\tilde{Q}_s$, given by \eqref{eq:Qs} (we will consider it under the measure $\prob^*$) and $\tilde{Q}_0 = \frac{x-x_R}{x}$. First, we note that if we let $Z_s$ be the stochastic process, taking values in $[0,\pi]$, such that
\begin{align*}
    \tilde{Q}_s = \frac{1-\cos(Z_s)}{2},
\end{align*}
then by It\^{o}'s formula we have that
\begin{align*}
    dZ_s = \left[\frac{1-3a-a\rho+\mu}{\sin(Z_s)} + \left(\frac{1}{2}-a+\mu\right)\cot(Z_s)\right] ds + d\tilde{B}_s^*.
\end{align*}
Thus, $Z_s$ is asymptotically Bessel-$\left(\frac{3}{2}-4a-a\rho+2\mu\right)$ at the origin and asymptotically Bessel-\newline
$\left(\frac{1}{2}-2a-a\rho\right)$ at $\pi$ (recall that if $dX_s = b(X_s)ds + dB_s$, then we say that $X_s$ is asymptotically Bessel-$\alpha$ at the origin if $|b(x) - \frac{\alpha}{x}| \leq cx$ and $|b'(x) + \frac{\alpha}{x^2}| \leq c$ for some constant $c$ and all $x \in (0,\frac{\pi}{2}]$, and we say that $X_s$ is asymptotically Bessel-$\alpha$ at $\pi$ if $|b(x) - \frac{\alpha}{\pi-x}| \leq c(\pi-x)$ and $|b'(x) + \frac{\alpha}{(\pi-x)^2}| \leq c$ for $x \in [\frac{\pi}{2},\pi)$). Since
\begin{align*}
    \frac{3}{2}-4a-a\rho+2\mu > \frac{1}{2}
\end{align*}
a standard result on Bessel processes implies that $Z_s$ will almost surely not reach $0$ in finite time, hence, $\tilde{Q}_s$ will almost surely not reach $0$ in finite time, and thus, $\prob^*(\tilde{t}(s) < \infty) = 1$ for every $s$.

The rest of this appendix is devoted to showing that $\tilde{Q}_s$ has an invariant distribution, to which it converges exponentially fast.
We prove this via the eigenvalue method. This was done in \cite{Zha16}, Appendix B, and most of this section follows almost verbatim from that, but we choose to include it here for completeness. We consider a transformation of $\tilde{Q}_s$ as the rather standard form of the eigenvalue problem for the transformed process makes it easier to solve. Let $Y_t$ be the diffusion process $Y_t = 2\tilde{Q}_t-1$, that is, $Y_t$ follows the SDE
\begin{align}\label{SDEY}
    dY_t = \left[-\frac{\delta_+}{4}(Y_t+1) - \frac{\delta_-}{4}(Y_t-1)\right] dt - \sqrt{1-Y_t^2} dB_t,
\end{align}
where $\delta_+ = 4-8a-2a\rho+4\mu>0$, $\delta_- = 4a+2a\rho>0$ and $Y_0 \in (-1,1]$. Note that this diffusion can be defined for all positive time (through reflection at $1$). First, we assume that $Y$ has a smooth transition density. It is then given by the Kolmogorov backward equation,
\begin{align}\label{KBE}
    \partial_t p = \frac{1-x^2}{2} \partial_x^2 p - \left(\frac{\delta_+}{4}(x+1) + \frac{\delta_-}{4}(x-1)\right) \partial_x p.
\end{align}
Assuming that both sides equal $\tilde{\lambda} p$, we arrive at the following differential equation 
\begin{align*}
    (1-x^2)p''(x)+\left( \frac{\delta_-}{2}- \frac{\delta_+}{2} - \left(\frac{\delta_+}{2}+\frac{\delta_-}{2} \right)x \right)p'(x)-2\tilde{\lambda} p(x) = 0,
\end{align*}
which has a solution if and only if $\tilde{\lambda} = \tilde{\lambda}_n = - \frac{n}{2}(n+ \frac{\delta_+}{2} + \frac{\delta_-}{2} -1 )$, where $n$ is a nonnegative integer. Then, for each $n$, the solution is given by the Jacobi polynomial $P_n^{(\frac{\delta_+}{2}-1,\frac{\delta_-}{2}-1)}(x)$ (see \cite{EMOT53}), that is, the orthogonal polynomials on $(-1,1)$ with respect to the inner product
\begin{align*}
\langle f,g \rangle_{\delta_+,\delta_-} = \int_{-1}^1 f(x) g(x) (1-x)^{\frac{\delta_+}{2}-1} (1+x)^{\frac{\delta_-}{2}-1} dx.  
\end{align*}
Hence
\begin{align*}
    p_n(t,x) = P_n^{(\frac{\delta_+}{2}-1,\frac{\delta_-}{2}-1)}(x) \exp\left(-\frac{n}{2}(n+\frac{\delta_+}{2}+\frac{\delta_-}{2}-1)t\right)
\end{align*}
solves \eqref{KBE} for $t>0$ and $-1<x<1$. Thus, the candidate for the transition density of $Y$ is
\begin{align}\label{tdens}
    p_Y(t,x,y) = &\sum_{n=0}^\infty \frac{(1-y)^{\frac{\delta_+}{2}-1} (1+y)^{\frac{\delta_-}{2}-1}P_n^{(\frac{\delta_+}{2}-1,\frac{\delta_-}{2}-1)}(x) P_n^{(\frac{\delta_+}{2}-1,\frac{\delta_-}{2}-1)}(y)}{\big\lVert P_n^{(\frac{\delta_+}{2}-1,\frac{\delta_-}{2}-1)}\big\rVert_{\delta_+,\delta_-}^2} \nonumber \\
    &\qquad \times \exp\left(-\frac{n}{2}(n+\frac{\delta_+}{2}+\frac{\delta_-}{2}-1)t\right),
\end{align}
where $\lVert f \rVert_{\delta_+,\delta_-}^2 = \langle f,f \rangle_{\delta_+,\delta_-}$. First, we need that \eqref{tdens} is absolutely convergent for $t>0$. This holds, since
\begin{align}\label{eq:normasymp}
    \big\lVert P_n^{(\frac{\delta_+}{2}-1,\frac{\delta_-}{2}-1)}\big\rVert_{\delta_+,\delta_-}^2 \asymp \frac{1}{n}
\end{align}
and
\begin{align}\label{eq:maxasymp}
    \max_{-1 \leq x \leq 1} |P_n^{(\frac{\delta_+}{2}-1,\frac{\delta_-}{2}-1)}(x)| \asymp n^{\frac{\delta_+}{2}-1}.
\end{align}
Next, we shall check that
\begin{align*}
    \E^{y_0}[f(Y_t)] = \int_{-1}^{1} f(y) p_Y(t,y_0,y) dy.
\end{align*}
It is sufficient to show this for polynomials. Let $q(x)$ be a polynomial, then
\begin{align*}
    a_n = \frac{\big\langle q,P_n^{(\frac{\delta_+}{2}-1,\frac{\delta_-}{2}-1)}\big\rangle_{\delta_+,\delta_-}}{\big\lVert P_n^{(\frac{\delta_+}{2}-1,\frac{\delta_-}{2}-1)}\big\rVert_{\delta_+,\delta_-}^2}, \quad n \geq 0,
\end{align*}
is zero for all but finitely many $n$ and
\begin{align*}
    q(x) = \sum_{n=0}^\infty a_n P_n^{(\frac{\delta_+}{2}-1,\frac{\delta_-}{2}-1)} (x).
\end{align*}
Next, letting
\begin{align*}
    \tilde{q}(t,x) = \sum_{n=0}^\infty a_n p_n(t,x),
\end{align*}
we note that $\tilde{q}(t,x)$ solves \eqref{KBE} and that $\tilde{q}(0,x) = q(x)$. We fix $t_0 > 0$ and let $M_t = \tilde{q}(t_0-t,Y_t)$ for $0<t \leq t_0$. Then $M_t$ is bounded and by It\^{o}'s formula,
\begin{align*}
    dM_t = -\sqrt{1-Y_t^2} \tilde{q}'(t_0-t,Y_t) dB_t
\end{align*}
(where $'$ denotes the spatial derivative) and hence $M_t$ is a bounded martingale. Furthermore, we have that $\lim_{t \rightarrow t_0} M_t = q(Y_{t_0})$, so by the optional stopping theorem, $\E^{y_0}[q(Y_{t_0})] = M_0 = \tilde{q}(t_0,y_0)$, that is,
\begin{align*}
    \E^{y_0}[q(Y_{t_0})] = \int_{-1}^1 q(y) p_Y(t_0,y_0,y) dy.
\end{align*}
Thus, $p_Y(t,x,y)$ is the transition density of $Y$ and sending $t$ to infinity, we see that $Y$ has a unique stationary distribution with density given by
\begin{align}\label{invdens}
    p_Y(y) = \frac{(1-y)^{\frac{\delta_+}{2}-1} (1+y)^{\frac{\delta_-}{2}-1}}{\int_{-1}^1 (1-x)^{\frac{\delta_+}{2}-1} (1+x)^{\frac{\delta_-}{2}-1} dx}, \quad y \in (-1,1),
\end{align}
that is, the $n=0$ term in \eqref{tdens}, we therefore see that there is a constant $C$, such that
\begin{align}
    |p_Y(t,x,y) - p_Y(y)| \leq C e^{-\frac{\delta_+ +\delta_-}{4} t}, \quad x,y \in [-1,1].
\end{align}
Thus, $p_Y(t,x,y) \rightarrow p_Y(y)$ uniformly in $x,y \in [-1,1]$ as $t \rightarrow \infty$. In fact, a stronger statement is true. Note that, writing $\hat{c} = (\int_{-1}^1 (1-x)^{\frac{\delta_+}{2}-1} (1+x)^{\frac{\delta_-}{2}-1} dx)^{-1}$, we have
\begin{align*}
    p_Y(t,x,y) &= p_Y(y)\left( 1 + \sum_{n=1}^\infty \frac{P_n^{(\frac{\delta_+}{2}-1,\frac{\delta_-}{2}-1)}(x) P_n^{(\frac{\delta_+}{2}-1,\frac{\delta_-}{2}-1)}(y)}{\hat{c}\big\lVert P_n^{(\frac{\delta_+}{2}-1,\frac{\delta_-}{2}-1)}\big\rVert_{\delta_+,\delta_-}^2} e^{-\frac{n}{2}(n+\frac{\delta_+}{2}+\frac{\delta_-}{2}-1)t} \right) \\
    &=p_Y(y) \left(1 + e^{-\frac{\delta_+ +\delta_-}{4}t} \sum_{n=1}^\infty \frac{P_n^{(\frac{\delta_+}{2}-1,\frac{\delta_-}{2}-1)}(x) P_n^{(\frac{\delta_+}{2}-1,\frac{\delta_-}{2}-1)}(y)}{\hat{c}\big\lVert P_n^{(\frac{\delta_+}{2}-1,\frac{\delta_-}{2}-1)}\big\rVert_{\delta_+,\delta_-}^2} e^{-\frac{n-1}{2}(n+\frac{\delta_+ + \delta_-}{2})t} \right).
\end{align*}
Next, we note that by \eqref{eq:normasymp} and \eqref{eq:maxasymp}, uniformly in $x,y \in [-1,1]$,
\begin{align*}
    &\left| \sum_{n=1}^\infty \frac{P_n^{(\frac{\delta_+}{2}-1,\frac{\delta_-}{2}-1)}(x) P_n^{(\frac{\delta_+}{2}-1,\frac{\delta_-}{2}-1)}(y)}{\hat{c}\big\lVert P_n^{(\frac{\delta_+}{2}-1,\frac{\delta_-}{2}-1)}\big\rVert_{\delta_+,\delta_-}^2} \exp\left(-\frac{n-1}{2}(n+\frac{\delta_+ + \delta_-}{2})t\right) \right| \\
    &\lesssim \sum_{n=1}^\infty n^{\delta_+-1} \exp\left(-\frac{n-1}{2}(n+\frac{\delta_+ + \delta_-}{2})t\right).
\end{align*}
Noting that the sum on the second line is bounded for all $t>0$ and decreasing in $t$, we have that for an arbitrary $\epsilon > 0$,
\begin{align*}
    p_Y(t,x,y) = p_Y(y)(1+O(e^{-\frac{\delta_+ +\delta_-}{4}t}))
\end{align*}
for $t>\epsilon$, where the implicit constant depends only on $\epsilon$. Thus, if we denote by $\mathcal{Y}$ the process satisfying the SDE \eqref{SDEY}, started according to the invariant density \eqref{invdens}, then for $y_0 >0$ and $t \geq 1$,
\begin{align}\label{invconvY}
    \E^{y_0}[f(Y_t)] = \E[f(\mathcal{Y}_t)](1+O(e^{-\frac{\delta_++\delta_-}{4} t})).
\end{align}
We have thus proven the following.
\begin{lem}
The transition density, $p_Y(t,x,y)$ for $Y$ is given by \eqref{tdens}. Furthermore, $Y$ has a unique invariant density, $p_Y(y)$, given by \eqref{invdens}, \eqref{invconvY} holds and $Y$ is ergodic.
\end{lem}
As a corollary, we have the following (noting that $\frac{\delta_+ + \delta_-}{4} = 1-a+\mu$ and letting $\E^*$ denote the expectation under $\prob^*$, as in Sections 2 and 3).
\begin{cor}\label{Qprop}
The transition density of $\tilde{Q}_s$ is given by $p_{\tilde{Q}}(s,x,y) = 2p_Y(s,2x-1,2y-1)$. Furthermore, it has invariant density $p_{\tilde{Q}}(y) = 2p_Y(2y-1)$ and if $\tilde{X}$ follows the same SDE as $\tilde{Q}$ and is started according to the invariant density, then for each $f \in L^1(\nu_{\tilde{Q}})$, where $\nu_{\tilde{Q}}(dx) = p_{\tilde{Q}}(x)dx$, we have, for $s \geq 1$, that
\begin{align}\label{invconvQ}
    \E^*[f(\tilde{Q}_s)] = \E^*[f(\tilde{X}_s)](1+O(e^{-(1-a+\mu)s})),
\end{align}
for each value of $\tilde{Q}_0 \in (0,1]$. Moreover, $\tilde{Q}$ is ergodic.
\end{cor}
\renewcommand{\abstractname}{Acknowledgements}
\begin{abstract}
We thank Fredrik Viklund, Jason Miller and Ewain Gwynne for helpful discussions. We also thank Nathana\"{e}l Berestycki and an anonymous referee for good comments on an earlier version. The research was conducted at KTH Royal Institute of Technology, Stockholm, as part of the author's PhD thesis and was supported by the Knut and Alice Wallenberg Foundation.
\end{abstract}

\end{document}